\documentclass[12pt]{amsart}
\usepackage{amscd,amsthm,amssymb,amsfonts,amsmath,euscript}
\allowdisplaybreaks
\usepackage{nicematrix}
\usepackage{tikz}
\usetikzlibrary{decorations.pathreplacing}

\usepackage{tikz-cd}
\usepackage[shortlabels, inline]{enumitem}
\usepackage{mathrsfs,mathtools}
\usepackage{xcolor}

\usepackage{comment}
\usepackage{extarrows}
\theoremstyle{plain}
\newtheorem{thm}{Theorem}[section]
\newtheorem{lemma}[thm]{Lemma}
\newtheorem{prop}[thm]{Proposition}
\newtheorem{cor}[thm]{Corollary}

\theoremstyle{definition}
\newtheorem{defn}[thm]{Definition}
\newtheorem{eg}[thm]{Example}
\newtheorem{question}[thm]{Question}
\theoremstyle{remark}
\newtheorem{remark}[thm]{Remark}

\newcommand{\nc}{\newcommand}

\def\makeop#1{\expandafter\def\csname#1\endcsname
  {\mathop{\rm #1}\nolimits}\ignorespaces}
\makeop{Hom}   \makeop{End}   \makeop{Aut}   \makeop{Isom}  \makeop{Pic} 
\makeop{Gal}   \makeop{ord}   \makeop{Char}  \makeop{Div}   \makeop{Lie} 
\makeop{PGL}   \makeop{Corr}  \makeop{PSL}   \makeop{sgn}   \makeop{Spf}
\makeop{Spec}  \makeop{Tr}    \makeop{Nr}    \makeop{Fr}    \makeop{disc}
\makeop{Proj}  \makeop{supp}  \makeop{ker}   \makeop{im}    \makeop{dom}
\makeop{coker} \makeop{Stab}  \makeop{SO}    \makeop{SL}    
\makeop{Cl}    \makeop{cond}  \makeop{Br}    \makeop{inv}   \makeop{rank}
\makeop{id}    \makeop{Fil}   \makeop{Frac}  \makeop{GL}    \makeop{SU}
\makeop{Nrd}   \makeop{Sp}    \makeop{Tr}    \makeop{Trd}   \makeop{diag}
\makeop{Res}   \makeop{ind}   \makeop{depth} \makeop{Tr}    \makeop{st}
\makeop{Ad}    \makeop{Int}   \makeop{tr}    \makeop{Sym}   \makeop{can}
\makeop{length}   \makeop{torsion} \makeop{GSp} \makeop{Ker}
\makeop{Adm}   \makeop{Mat}
\makeop{Q-isog}
\makeop{Rad}
\makeop{Ind}

\def\makebb#1{\expandafter\def
  \csname bb#1\endcsname{{\mathbb{#1}}}\ignorespaces}
\def\makebf#1{\expandafter\def\csname bf#1\endcsname{{\bf
      #1}}\ignorespaces} 
\def\makegr#1{\expandafter\def
  \csname gr#1\endcsname{{\mathfrak{#1}}}\ignorespaces}
\def\makescr#1{\expandafter\def
  \csname scr#1\endcsname{{\EuScript{#1}}}\ignorespaces}
\def\makecal#1{\expandafter\def\csname cal#1\endcsname{{\mathcal
      #1}}\ignorespaces} 

\def\doLetters#1{#1A #1B #1C #1D #1E #1F #1G #1H #1I #1J #1K #1L #1M
                 #1N #1O #1P #1Q #1R #1S #1T #1U #1V #1W #1X #1Y #1Z}
\def\doletters#1{#1a #1b #1c #1d #1e #1f #1g #1h #1i #1j #1k #1l #1m
                 #1n #1o #1p #1q #1r #1s #1t #1u #1v #1w #1x #1y #1z}
\doLetters\makebb   \doLetters\makecal  \doLetters\makebf
\doLetters\makescr 
\doletters\makebf   \doLetters\makegr   \doletters\makegr

\normalsize

\makeop{Bl}

\DeclareMathOperator{\fchar}{char}

\def\Spec{{\rm Spec}\,}

\newcommand{\Z}{\mathbb Z}
\newcommand{\Q}{\mathbb Q}
\newcommand{\R}{\mathbb R}

\renewcommand{\O}{\mathbb O} 
\newcommand{\F}{\mathbb F}
\renewcommand{\P}{\mathbb P} 
\renewcommand{\L}{\mathbb L} 
\newcommand{\M}{\mathbb M}    

\DeclareMathOperator{\Nm}{Nm}

\DeclareMathOperator{\xint}{int}

\nc{\embed}{\hookrightarrow}

\newcommand{\ab}{abelian variety }
\newcommand{\abs}{abelian varieties }
\nc{\ol}{\overline}
\nc{\wt}{\widetilde}
\nc{\opp}{\mathrm{opp}}

\makeop{Ram}
\makeop{Rep}

\newcommand{\ModOp}{\mathrm{Mod}\text{-}O_B}
\newcommand{\OpMod}{O_B\text{-}\mathrm{Mod}}
\newcommand{\ulm}{\underline{m}}

\usepackage{color}
\usepackage{marvosym}

\begin{document}
\numberwithin{equation}{section}

\numberwithin{thm}{section} 


\title[Ribet bimodules]{Ribet bimodules and principally polarized superspecial abelian varieties with quaternion action}

 \author{Jiangwei Xue}

\address{(Xue) School of
  Mathematics and Statistics, Wuhan University, Luojiashan, 430072,
  Wuhan, Hubei, P.R. China}

\email{xue\_j@whu.edu.cn}

\author{Xiangning Yang}

\address{(Yang) School of Mathematics and Statistics, Wuhan University, Luojiashan, 430072, Wuhan, Hubei, P.R. China}   

\email{yangxn@whu.edu.cn}

\dedicatory{Dedicated to the memory of Yuri G.~Zarhin, with deepest gratitude and respect}

\date{\today}

\begin{abstract}
 In an influential paper [K.~Ribet, Bimodules and abelian surfaces, in Algebraic number theory, 359--407, Adv.~Stud.~Pure Math., Vol.~17,  1989]  on the bad reduction of Shimura curves, Ribet studies certain superspecial abelian surfaces over $\overline{\mathbb{F}}_p$ with quaternion multiplication by a maximal order $\mathcal{O}$ in an indefinite quaternion $\mathbb{Q}$-algebra ramified at $p$.  In particular, he classifies the $p$-divisible groups of such $\mathcal{O}$-abelian surfaces by classifying $(\mathcal{O}_p, \mathcal{O}_p)$-bimodules $L_p$ that are free over $\Z_p$ (i.e.~bilattices) under an additional admissible assumption. In this paper, we generalize Ribet's result by removing the admissible assumption and producing a complete classification of $(\mathcal{O}_p, \mathcal{O}_p)$-bilattices $L_p$. Equip the right order $\calO_p$ with the canonical involution, and  suppose additionally that  the left order $\calO_p$ is equipped with an orthogonal involution $*$. We derive the necessary and sufficient condition for the existence of a perfect quaternion hermitian form $\langle~,~\rangle_p:L_p\times L_p\to \calO_p$ on the right $\calO_p$-lattice $L_p$ inducing the given involution $*$ on the left order $\mathcal{O}_p$, and give a complete classification of  such self-dual quaternion hermitian 
   $(\mathcal{O}_p, *, \mathcal{O}_p)$-bilattices $(L_p, \langle~,~\rangle_p)$.  Globally, we apply these classification results to the study of the existence of principal polarizations on  superspecial abelian varieties over $\overline{\mathbb{F}}_p$ equipped with $\mathcal{O}$-action.  
\end{abstract}

\maketitle 
\section{Introduction}
Let $k$ be an algebraically closed field of characteristic $p>0$.  An abelian variety over $k$ is called superspecial (resp.~supersingular) if it is isomorphic (resp.~isogenous) to a product of supersingular elliptic curves over $k$. Very often, supersingular abelian varieties are studied by connecting them with superspecial ones through the minimal isogeny \cite[Lemma~1.8]{li-oort}, so in the current paper we focus on superspecial abelian varieties.

Fix a supersingular elliptic curve $E$ over $k$. It is well known \cite[Theorem~4.2]{waterhouse:thesis} that the endomorphism ring $\calR\coloneqq \End(E)$ is a maximal order in the endomorphism algebra 
$\End^0(E)\coloneqq \End(E)\otimes_\Z \Q$, which is isomorphic to the unique quaternion $\Q$-algebra $D\coloneqq D_{p, \infty}$ ramified exactly at $\{p, \infty\}$.  For this reason, the theory of superspecial abelian varieties is closely tied with quaternion arithmetic.  For example, the Deuring-Eichler correspondence \cite[Theorem~42.3.2]{voight-quat-book} establishes a bijection between the set of isomorphism classes of supersingular elliptic curves over $k$ and the set of isomorphism classes of (invertible) fractional right  $\calR$-ideals. In the higher dimensional case, a famous theorem of Deligne, Ogus and Shioda  \cite[\S 1.6, p.~13]{li-oort} shows that every superspecial abelian $k$-variety of dimension $n\geq 2$  is isomorphic to $E^n$, so in particular, for each fixed $n\geq 2$ there is exactly one isomorphism class of superspecial abelian $k$-varieties of dimension $n$.  Ibukiyama, Katsura and Oort \cite[Theorem~2.10]{Ibukiyama-Katsura-Oort-1986} further establish a bijection between the set of isomorphism classes of principal polarizations on $E^n$ and the set of isometric classes of positive definite self-dual quaternion hermitian $\calR$-lattices of rank $n\geq 2$ (i.e.~quaternion hermitian $\calR$-lattices belonging to the \emph{principal genus $\calL_n(p, 1)$}).  Their result can  be interpreted as a lattice description of superspecial points on the Siegel moduli space $\calA_{n, 1}$. Following this line of pursuit, the theory culminates in the result of Li and Oort \cite[Theorem, \S4.9]{li-oort} who show that the number of irreducible components of the supersingular locus $\calS_{n, 1} \subset \calA_{n, 1}$ is equal to the class number of the principal genus (resp.~non-principal genus) if $n$ is  odd (resp.~even). Recently, lattice descriptions of (not necessarily supersingular) abelian varieties over finite fields were  pursued by Centeleghe and Stix \cite{centeleghe-stix-I, Centeleghe-Stix-II} in a different direction,  whose results generalize a classical result of Deligne~\cite{deligne:ord}.

Naturally, the next objects of study after the Siegel moduli space are PEL-type Shimura varieties, among which Shimura curves parametrizing abelian surfaces with quaternion multiplication are the simplest examples.   
Let $\calO$ be a maximal $\Z$-order in an indefinite quaternion $\Q$-algebra $B$. An $\calO$-abelian scheme over a base scheme $S$ is a pair $(X, \iota)$, where $X$ is an abelian scheme over $S$ and $\iota$ is a monomorphism of rings $\calO\hookrightarrow \End_S(X)$.  For the moment, suppose that the reduced discriminant of $\calO$ is the product of two distinct primes $p, q$, and let $M$ be a positive integer coprime to $pq$. 
 In \cite{Ribet-bimod}, Ribet considers the Shimura curve $\scrC$ over $\Spec(\Z)$ that classifies $\calO$-abelian surfaces satisfying the \emph{Kottwitz determinant condition} \cite[\S2.3]{Yu-Grenoble-2021} (or equivalently the \emph{special condition} as in \cite{MR422290, Ribet-bimod} by \cite[Lemma~5.2]{Yu-Grenoble-2021}) together with a $\Gamma_0(M)$-level structure.  He shows that the singular points of the reduction $\scrC_{\F_p}$ in characteristic $p$ are in (canonical) bijection with the singular points of the standard modular curve $\calX_0(Mpq)_{\F_q}$ in characteristic $q$.
  As the singular points on $\scrC_{\F_p}$ are represented by certain superspecial abelian surfaces $X=E^2$ over $k$  with $\calO$-action, Ribet remarks that ``\emph{to give an action of $\calO$ on $X$ is then to give an $(\calO, \calR)$-bimodule which is $\Z$-free of rank $8$}''. From the point of view of Jordan et al.~\cite{Poonen-et:av}, 
 the bimodule in question is given by $L\coloneqq \Hom(E, X)$. 
An important initial step in Ribet's work is to classify such bimodules locally at $p$, which is achieved under an additional \emph{admissible} assumption in \eqref{eq:adm} below.  The $p$-adic completions $\calO_p\coloneqq \calO\otimes_\Z \Z_p$ and $\calR_p\coloneqq \calR\otimes_\Z \Z_p$ are isomorphic since both are maximal $\Z_p$-orders in the unique quaternion division $\Q_p$-algebra. This allows us to fix an identification of $\calR_p$ with $\calO_p$ and view $L_p\coloneqq L\otimes_\Z \Z_p$ as an $(\calO_p, \calO_p)$-bimodule.

\begin{thm}[{\cite[Theorems~1.2--1.3]{Ribet-bimod}}]\label{thm:Ribet}
    Let $\grP_p$ be the unique maximal two-sided $\calO_p$-ideal with residue field $\F_{p^2}$. Let $L_p$ be an $(\calO_p, \calO_p)$-bimodule which is free of finite rank over $\Z_p$. We say that $L_p$ is admissible if it satisfies the equality 
    \begin{equation}\label{eq:adm}
        \grP_p L_p=L_p \grP_p.  
    \end{equation}
  Assume that this is the case. Then 
  \begin{equation*}
      L_p\simeq \calO_p\oplus\cdots\oplus \calO_p\oplus \grP_p\oplus \cdots\oplus\grP_p,
  \end{equation*}
 where   $\calO_p$ and $\grP_p$ are regarded as bimodules via the natural left- and right-multiplications of $\calO_p$ on itself and on $\grP_p$ respectively. Moreover, two admissible $(\calO_p, \calO_p)$-bimodules $L_p$ and $L_p'$ are isomorphic if and only if their associated $(\F_{p^2}, \F_{p^2})$-bimodules $L_p/(\grP_pL_p)$ and $L_p'/(\grP_pL_p')$ are isomorphic.  
\end{thm}
Due to the above theorem, such bimodules over local or global  quaternion orders are called \emph{Ribet bimodules} in the literature \cite{MR2931385}. To emphasize the freeness of $L_p$ over $\Z_p$ and distinguish it from torsion bimodules of the form $L_p/\grP_p L_p$, we shall call it an $(\calO_p, \calO_p)$-\emph{bilattice} instead. The first result of the current paper provides a complete classification of $(\calO_p, \calO_p)$-bilattices without the admissible assumption. When the $\Z_p$-rank of the bilattice is $8$, such a classification was previously obtained jointly by Terakado, Yu and the first named author in \cite{terakado-yu-xue-2026} using a different method. 
Observe that an $(\calO_p, \calO_p)$-bilattice is the same as  a left $\calO_p\otimes_{\Z_p}\calO_p^\opp$-lattice, where $\calO_p^{\opp}$ denotes the opposite ring of $\calO_p$. From the Krull-Schmidt-Azumaya Theorem \cite[Theorem~6.12]{curtis-reiner:1}, to classify left  $\calO_p\otimes_{\Z_p}\calO_p^\opp$-lattices, it is enough to classify all the indecomposable ones.  Clearly, both $\calO_p$ and $\grP_p$ are indecomposable. In \eqref{eq:defn-varphi-1}--\eqref{eq:defn-four-indecomp-latt}, we define two more indecomposable left $\calO_p\otimes_{\Z_p}\calO_p^\opp$-lattices $\L_1, \L_2$ that are of rank $2$ over $\calO_p$. It turns out that this completes the list of indecomposable left $\calO_p\otimes_{\Z_p}\calO_p^\opp$-lattices.

\begin{thm}\label{thm:4-indec-intro}
 The $\Z_p$-order $\calO_p\otimes_{\Z_p}\calO_p^{\opp}$ has finite representation type, and up to isomorphism there are exactly four indecomposable $(\calO_p, \calO_p)$-bilattices: $\calO_p$, $\grP_p$, $\L_1$ and $\L_2$. 
Every $(\calO_p, \calO_p)$-bilattice $L_p$ is isomorphic to a direct sum 
    \begin{equation*}
        \calO_p^{\oplus r_1}\oplus \grP_p^{\oplus r_2}\oplus \L_1^{\oplus t_1}\oplus \L_2^{\oplus t_2}
    \end{equation*}
for a unique quadruple $(r_1, r_2, t_1, t_2)\in\Z_{\geq 0}^4$, which will be called the structural invariant of $L_p$ and denoted by $\ulm(L_p)$.  Moreover, the following are equivalent for any  two  $(\calO_p, \calO_p)$-bilattices $L_p$ and $L_p'$: 
\begin{enumerate}
    \item $L_p$ and $L_p'$ are isomorphic;
    \item their associated $(\F_{p^2}, \calO_p/p\calO_p)$-bimodules $L_p/(\grP_pL_p)$ and $L_p'/(\grP_pL_p')$ are isomorphic;
    \item their associated $(\calO_p/p\calO_p, \F_{p^2})$-bimodules
    $L_p/(L_p\grP_p)$ and $L_p'/(L_p'\grP_p)$ are isomorphic.
\end{enumerate}
\end{thm}
More generally, a similar result holds when $\calO_p$ is replaced by a maximal order in a division quaternion algebra over a nonarchimedean local field; see Theorem~\ref{thm:conj-class} and Corollary~\ref{cor:one-to-one-corre}.
Ribet observes in \cite{Ribet-bimod} that $\calO_p\otimes_{\Z_p}\calO_p^{\opp}$ is not a hereditary order. As an application of Theorem~\ref{thm:4-indec-intro}, we sketch a proof that $\calO_p\otimes_{\Z_p}\calO_p^{\opp}$ is a Gorenstein order but not a Bass order in Remark~\ref{rem:not-bass}.

We return to the global case and let $\calO$ be  a maximal $\Z$-order in an indefinite quaternion $\Q$-algebra $B$  with $p\mid \disc(\calO)$ (and no additional restrictions on the discriminant otherwise).   
Fix an identification $\calR_p\simeq \calO_p$ so that we can apply the classification result of Theorem~\ref{thm:4-indec-intro} to a global $(\calO, \calR)$-bilattice $L$.
From Lemma~\ref{lem:emb-exist-neces-cond-2}, an $n$-dimensional superspecial abelian variety $X$ over $k$  admits an embedding $\iota: \calO\hookrightarrow \End(X)$ if and only if $n$ is even.  In this case, the structural invariant $\ulm^{(p)}(X, \iota)$ at $p$ is defined as the structural invariant $\ulm^{(p)}(L)\coloneqq \ulm(L_p)=(r_1, r_2, t_1, t_2)$ of its associated Ribet $(\calO, \calR)$-bilattice $L\coloneqq \Hom(E, X)$ at $p$, where $E$ is the fixed supersingular elliptic $k$-curve with $\calR=\End(E)$ as before. An easy rank consideration shows that $r_1+r_2+2t_1+2t_2=n$.  When $n=2$, the number of isomorphism classes of superspecial $\calO$-abelian surfaces $(X, \iota)$ over $k$
has been computed by Terakado, Yu and the first named author in \cite{terakado-yu-xue-2026} using the Eichler trace formula when studying the superspecial points on the Shimura curve $\scrC_{\F_p}$ with full level structure. 
 On the other hand, when $n\geq 4$ and is even,  a similar argument as the proof of the Deligne-Ogus-Shioda theorem shows that for each quadruple $(r_1, r_2, t_1, t_2)\in \Z_{\geq 0}^4$ with  $r_1+r_2+2t_1+2t_2=n$, there is a unique isomorphism class of superspecial  $\calO$-abelian varieties $(X, \iota)$ over $k$ with structural invariant $(r_1, r_2, t_1, t_2)$; see Example~\ref{eg:F}. 
Naturally, one further asks whether there exists a principal polarization on $X$ that is compatible with $\iota$ (in a suitable sense). This  question is answered by our next main result of this paper, which will be explained after we set up the basic background material.  

Previously we have discussed superspecial $\calO$-abelian varieties without bringing  polarization into the picture. For the study of Shimura curves  parameterizing abelian surfaces with quaternion multiplication, very often there is no need to consider polarizations. Indeed, given an $\calO$-abelian surface $(X, \iota)$ over a base scheme $S$, a classical result of Drinfeld  \cite[Proposition~4.3]{MR422290} shows that if $(X, \iota)$  satisfies the Kottwitz determinant condition, then it admits a canonical principal polarization $\lambda: X\to X^\vee$ such that the Rosati involution induced by $\lambda$ keeps $\iota(\calO)$ stable and its restriction defines an involution via $\iota$ on $\calO$   given by $\alpha\mapsto \gamma \bar{\alpha} \gamma^{-1}$ with $\gamma^2=-\disc(\calO)$.  Here
$X^{\vee}$ denotes the dual abelian scheme of $X$, and 
 $\alpha\mapsto \bar{\alpha}$ denotes the canonical involution on the quaternion $\Q$-algebra $B$.

The Shimura curves discussed above are special cases of PEL Shimura varieties of type C. 
Let $F$ be a totally real number field of degree $d\coloneqq [F:\Q]$ with ring of integers $O_F$, and $B$ be a totally indefinite quaternion $F$-algebra.  An involution $*$ on $B$ is said to be \emph{positive} if  $\Tr_{B/\Q}(\alpha\alpha^*)>0$ for every nonzero $\alpha\in B$.  
Fix a positive involution $*$ on $B$ and let $\calO$ be a maximal $O_F$-order in $B$ stable under $*$. A (principally) polarized $\calO$-abelian scheme over a base scheme $S$ is a triple $(X, \lambda, \iota)$, where $(X, \iota)$ is an $\calO$-abelian scheme over $S$, and $\lambda: X\to X^{\vee}$ is a (principal) polarization such that 
\begin{equation}\label{eq:QM-eq-intro}
    \lambda\circ\iota(\alpha^*)=\iota(\alpha)^{\vee}\circ\lambda,\qquad \forall \alpha\in \calO.
\end{equation}
Here $\iota(\alpha)^{\vee}$ denotes the dual morphism of $\iota(\alpha)$.
Equivalently, condition \eqref{eq:QM-eq-intro} requires the positive involution $*$ on $\calO$ to coincide via $\iota$ with the restriction of the Rosati involution induced by $\lambda$, in which case we  say that $\lambda$ is a polarization on $(X, \iota)$ (provided that $*$ is clear from the context). 
Fix a positive integer $m$. We consider the Shimura variety of type C whose  canonical model over $\Z_{(p)}$ can be constructed as the  coarse moduli scheme $\bfM$  parametrizing $2dm$ dimensional principally polarized $\calO$-abelian varieties $(X, \lambda, \iota)$ satisfying the Kottwitz determinant condition (with suitable level structure). In \cite[Theorem~3.6]{terakado-xue-yu:2023}, Terakado, Yu and the first named author determine the necessary and sufficient condition  for the existence of a $2dm$ dimensional principally polarized superspecial $\calO$-abelian variety $(X, \lambda, \iota)$ over $k$ satisfying the Kottwitz determinant condition. In particular, they prove that the superspecial locus of the geometric special fiber $\bfM\otimes k$ is nonempty if and only if its generic fiber $\bfM\otimes \Q$  is nonempty.

Note that both the result of Drinfeld and that of Terakado-Xue-Yu require the  Kottwitz determinant condition. However, 
 in practical applications, it is often necessary to study $\calO$-abelian varieties beyond those satisfying the Kottwitz determinant condition. For example, in Ribet's case, the superspecial $\calO$-abelian surfaces that do not satisfy the  Kottwitz determinant condition parametrize the irreducible components of the Shimura curve $\scrC_{\F_p}$.   Therefore, we shall forgo the Kottwitz determinant condition and consider arbitrary superspecial $\calO$-abelian varieties.  

  \begin{question}\label{que:our-que}
      Given a superspecial $\calO$-abelian variety $(X, \iota)$ over $k$ and a positive involution $*$ on $\calO$, when is $(X, \iota)$ principally polarizable?
  \end{question}

 To study this question, we follow the approaches of Ibukiyama, Katsura and Oort \cite{Ibukiyama-Katsura-Oort-1986}  to reduce the above problem to that of 
 certain self-dual hermitian $(\calO, \calR)$-bilattices, where $\calR\coloneqq \End(E)$ is the endomorphism ring of the fixed supersingular elliptic $k$-curve $E$ as before.    Let $(X, \iota)$ be a  superspecial $\calO$-abelian variety over $k$ of dimension $n$ with associated  $(\calO, \calR)$-bilattice $L\coloneqq\Hom(E, X)$.  By Lemma~\ref{lem:B-D-bimod}, $L$ is necessarily a free right $\calR$-module of rank $n$, with $n$ divisible by $2d$. 
 From \cite[Proposition~2.8]{Ibukiyama-Katsura-Oort-1986} and \cite[\S4]{Ibukiyama-Karemaker-Yu-2025}, each polarization $\lambda: X\to X^{\vee}$ corresponds to a positive definite 
 hermitian form $\langle~,~\rangle: L\times L\to \calR$ as defined in \eqref{eq:herm-form-defn}. Moreover, $\lambda$ is principal if and only if the corresponding pairing $\langle~,~\rangle$ is perfect. The compatibility condition \eqref{eq:QM-eq-intro} is further translated into 
\begin{equation}\label{eq:herm-compat-intro}
    \langle \alpha x, y\rangle=\langle x, \alpha^* y\rangle,\quad\forall x, y\in L,\, \forall \alpha\in \calO.
\end{equation}
Thus we have established the following bijection (see Lemma~\ref{lem:pol-herm-pair})
\begin{equation}\label{eq:bij-pol-hermlatt}
     \left\{\parbox{3.4cm}{principal polarizations on $(X, \iota)$}\right\}\quad \xlongleftrightarrow{1-1}\quad \left\{\parbox{5.6cm}{perfect positive definite hermitian forms $\langle~,~\rangle: L\times L\to \calR$ satisfying \eqref{eq:herm-compat-intro} on $L$}\right\} 
\end{equation}
Henceforth an ordered pair $(L, \langle~,~\rangle)$ consisting of an $(\calO, \calR)$-bilattice $L$ and a non-degenerate hermitian form $\langle~,~\rangle$ satisfying \eqref{eq:herm-compat-intro} will be called a \emph{hermitian $(\calO, *, \calR)$-bilattice} for short. Analogously, at each prime $\ell$ we define \emph{hermitian $(\calO_\ell, *, \calR_\ell)$-bilattices}. Clearly, a global hermitian $(\calO, *, \calR)$-bilattice $(L, \langle~,~\rangle)$ is self-dual if and only if it is self-dual locally at every prime $\ell$ (including $\ell=p$). Since $\calR_\ell\simeq \Mat_2(\Z_\ell)$ when $\ell\neq p$, the classification of self-dual hermitian $(\calO_\ell, *, \calR_\ell)$-bilattices for $\ell\neq p$ is made easy by the Morita equivalence. Thus the main difficulty lies in classifying the self-dual hermitian $(\calO_p, *, \calR_p)$-bilattices locally at $p$.

From the decomposition $O_F\otimes_\Z \Z_p=\oplus_{v\mid p} O_{F_v}$ with $v$ ranging through all places of $F$ above $p$, the  $\Z_p$-order $\calO_p\coloneqq\calO\otimes_{\Z}\Z_p$ decomposes into a direct sum $\oplus_{v\mid p} \calO_v$ of maximal quaternion $O_{F_v}$-orders $\calO_v\coloneqq \calO\otimes_{O_F}O_{F_v}$. Correspondingly, the $(\calO_p, \calR_p)$-bilattice $L_p\coloneqq L\otimes_\Z\Z_p$ splits into a direct sum   
\begin{equation}\label{eq:latt-decomp-at-p}
    L_p=\bigoplus_{v\mid p} L_v, 
\end{equation}
where each $L_v\coloneqq L\otimes_{O_F}O_{F_v}$ is an $(\calO_v, \calR_p)$-bilattice of $\calO_v$-rank $n/d$; see Lemma~\ref{lem:rank-of-summand}. Moreover, if $\langle~,~\rangle_p: L_p\times L_p\to \calR_p$ is a hermitian form on $L_p$ satisfying \eqref{eq:herm-compat-intro}, then the decomposition~\eqref{eq:latt-decomp-at-p} is orthogonal. In particular, $L_p$ is self-dual if and only if every $L_v$ is self-dual for all $v\mid p$. Thus for the local classification, we may treat each $(\calO_v, \calR_p)$-bilattice $L_v$ individually, which can also be viewed as an $(\calO_v, \calR_p\otimes_{\Z_p}O_{F_v})$-bilattice in the natural way.
To apply Theorem~\ref{thm:4-indec-intro} to the  classification of self-dual hermitian $(\calO_p, *, \calR_p)$-bilattices, we further assume that $\calR_p\otimes_{\Z_p}O_{F_v}$ is a maximal order in $D_p\otimes_{\Q_p}F_v$ for all  $v\mid p$, which is equivalent to assuming that the global quaternion $O_F$-order $\calR\otimes_{\Z}O_F$ is maximal in $D\otimes_{\Q}F$. For example, this holds when $p$ splits completely in $F$.
From Lemma~\ref{lem:max-order-under-base-change}, this maximality assumption also implies that the quaternion $F$-algebra $D\otimes_\Q F$ is ramified at all finite places $v$ of $F$ above $p$. By Lemma~\ref{lem:lift-pairing}, every perfect hermitian $(\calO_v, *, \calR_p)$-pairing on $L_v$ lifts to a perfect hermitian $(\calO_v, *, \calR_p\otimes_{\Z_p}O_{F_v})$-pairing, where  $\calO_v$ and $\calR_p\otimes_{\Z_p}O_{F_v}$ share the same central base order $O_{F_v}$.

If $v$ is split in $B$, i.e., $\calO_{v}\simeq\Mat_2(O_{F_v})$, the classification of hermitian $(\calO_v, *, \calR_p\otimes_{\Z_p}O_{F_v})$-bilattices is again made easy by the Morita equivalence, so we focus on the case where $v$ is ramified in $B$. Suppose that this is the case,  and we write  $\ord_{B_v}: B_v^\times \twoheadrightarrow \Z$ for the normalized discrete valuation on $B_v$. 
We fix an identification $\calR_p\otimes_{\Z_p} O_{F_v}\simeq\calO_v$  and view $L_v$ as an $(\calO_v, \calO_v)$-bilattice.
From \cite[\S21]{mumford:av}, the positive involution $*$ on $B$ is given by an element $\gamma\in B^\times$ such that 
\begin{equation}\label{eq:positive-involution-intro}
    \gamma^2\in F_{<0}^\times \quad\text{and}\quad
    \alpha^*=\gamma \ol{\alpha} \gamma^{-1},\quad\forall \alpha\in B, 
\end{equation}
where $F_{<0}^\times$ denotes the subset of totally negative elements of $F$.

\begin{thm}\label{thm:self-dual-local-bilatt-intro}
  Let $v$ be a place of $F$ above $p$ ramified in $B$ as above, and $L_v$ be an $(\calO_v, \calO_v)$-bilattice with structural invariant $(r_1, r_2, t_1, t_2)$. Then there exists a perfect hermitian  $(\calO_v, *, \calO_v)$-pairing $\langle~,~\rangle_v$ on $L_v$ if and only if the  following condition holds
\begin{equation}\label{eq:cond-for-quadruple-intro}
    \begin{dcases*}
        r_1=r_2,  &  if $\ord_{B_{v}}(\gamma)$  is odd;\\
        2|r_1, 2|r_2, \text{ and } t_1=t_2, & if   $\ord_{B_v}(\gamma)$ is even.
    \end{dcases*}
\end{equation}
Moreover, when  condition \eqref{eq:cond-for-quadruple-intro} holds,  up to isometry there is a unique  self-dual hermitian $(\calO_v, *, \calO_v)$-bilattice $(L_v, \langle~,~\rangle_v)$ with structural invariant $(r_1, r_2, t_1, t_2)$.
\end{thm}

In Section~\ref{sec:self-dual-local-latt}, we show that every self-dual hermitian $(\calO_v, *, \calO_v)$-bilattice decomposes into an orthogonal direct sum of self-dual sub-bilattices, each with $\calO_v$-rank $2$ or $4$. We write down the complete list of such ``basic" self-dual hermitian $(\calO_v, *, \calO_v)$-bilattices in Theorems~\ref{thm:self-dual-latt-odd} and \ref{thm:self-dual-latt-even}.  In a sense, these theorems follow the same spirit as the classical results of Shimura \cite{Shimura1963-AltHermForms} and Jacobowitz \cite{Jacobowitz-HermForm}, who show that locally every (maximal) quaternion hermitian lattice is an orthogonal direct sum of modular lines and planes. 
 
We return to the study of superspecial $\calO$-abelian varieties.
Keep the assumption that $\calR\otimes_{\Z}O_F$ is a maximal order in $D\otimes_{\Q}F$ as above and let $\Sigma$ be the set of  places of $F$ above $p$ ramified in $B$.  
For each $v\in\Sigma$, we fix an identification $\calR_p\otimes_{\Z_p} O_{F_v}\simeq\calO_v$ once and for all. Given  an $(\calO, \calR)$-bilattice $L$, we  define its structural invariant $\ulm^{(v)}(L)$  at $v$  to  be the structural invariant of the $(\calO_{v}, \calO_{v})$-bilattice $L_{v}$, that is, 
\begin{equation}
    \ulm^{(v)}(L)\coloneqq \ulm(L_{v})=(r_1^{(v)}, r_2^{(v)}, t_1^{(v)},  t_2^{(v)}).
\end{equation}
Since $\rank_{\calO_v}(L_{v})=n/d$, this quadruple is subject to the following constraint:
 \begin{equation}\label{eq:compute-rank-intro}
     r_1^{(v)}+r_2^{(v)}+2t_1^{(v)}+2t_2^{(v)}=n/d.
 \end{equation}  
Similarly, given a superspecial $\calO$-abelian variety $(X, \iota)$, we put $\ulm^{(v)}(X, \iota)\coloneqq \ulm^{(v)}(L)$ with $L\coloneqq \Hom(E, X)$ as before.  From Lemma~\ref{lem:emb-exist-neces-cond-2}, an $n$-dimensional superspecial abelian variety $X$ over $k$ admits an embedding $\iota: \calO\hookrightarrow \End(X)$ if and only if $n$ is divisible by $2d$.

\begin{thm}\label{thm:nec-suff-cond-intro}
Let $*$ be a positive involution on the totally indefinite quaternion $F$-algebra $B$ as in \eqref{eq:positive-involution-intro}. Assume that $\calR\otimes_{\Z}O_F$ is a maximal order in $D\otimes_{\Q}F$. Let $\Sigma$ be the set of places of $F$ above $p$ ramified in $B$, and  $\Xi$ be the set of finite places of $F$ coprime to $p$ and ramified in $B$.  Then the following holds true. 

\begin{enumerate}[label={(\roman*)},  leftmargin=*]
    \item If there exists a principally polarized superspecial $\calO$-abelian variety over $k$ of dimension $n$, then the following condition necessarily holds:
      \begin{equation}\label{eq:cond-for-S-1-intro}
      n/(2d)\text{ is even or }\ord_{B_{w}}(\gamma)\text{ is odd for every }w\in \Xi.
      \end{equation}
    \item Let $n$ be a fixed positive integer divisible by $2d$ and suppose that condition~\eqref{eq:cond-for-S-1-intro} holds. For each $v\in\Sigma$, fix a quadruple
     \begin{equation}\label{eq:str-inv-av}
        (r_1^{(v)}, r_2^{(v)}, t_1^{(v)}, t_2^{(v)})\in \Z_{\geq 0}^4 
   \end{equation} 
   satisfying~\eqref{eq:compute-rank-intro}.
   Then there exists a principally polarized superspecial $\calO$-abelian variety $(X, \lambda, \iota)$ over $k$ of dimension $n$ with $\ulm^{(v)}(X, \iota)=(r_1^{(v)}, r_2^{(v)}, t_1^{(v)}, t_2^{(v)})$ for all $v\in\Sigma$ if and only if the following condition holds:
    \begin{equation} \label{eq:cond-nec-at-p-intro}
     \forall v\in\Sigma, \quad  \begin{dcases*}
        r_1^{(v)}=r_2^{(v)},  &  if $\ord_{B_{v}}(\gamma)$  is odd;\\
        2|r_1^{(v)}, 2|r_2^{(v)}, \text{ and } t_1^{(v)}=t_2^{(v)}, & if   $\ord_{B_{v}}(\gamma)$ is even. 
    \end{dcases*}   
   \end{equation}   
   \item Fix both $n$ and  the quadruples $(r_1^{(v)}, r_2^{(v)}, t_1^{(v)}, t_2^{(v)})$ for all $v\in\Sigma$ as in part (ii) and suppose that both conditions \eqref{eq:cond-for-S-1-intro} and \eqref{eq:cond-nec-at-p-intro} hold. Denote the narrow class number of $F$
    by $h^+(F)$.
    If further $n>2d$, then there are exactly $h^+(F)$  isomorphism classes of (unpolarized) superspecial $\calO$-abelian $k$-varieties $(X', \iota')$ of dimension $n$ with $\ulm^{(v)}(X', \iota')=(r_1^{(v)}, r_2^{(v)}, t_1^{(v)}, t_2^{(v)})$ for all $v\in\Sigma$, among which exactly one of them is principally polarizable. 
\end{enumerate}
\end{thm}

We will reformulate Theorem~\ref{thm:nec-suff-cond-intro} in the language of global bilattices as in Theorem~\ref{thm:genus-char-2}, and then obtain it as a corollary as in Theorem~\ref{thm:genus-char}.
When $h^+(F)=1$ (e.g.~$F=\Q$) and $n>2d$, Theorem~\ref{thm:nec-suff-cond-intro} provides the necessary and sufficient condition for a given superspecial $\calO$-abelian $k$-variety $(X, \iota)$ of dimension $n$ to be principally polarizable, thus answering Question~\ref{que:our-que} completely in this case; see Corollary~\ref{prop:ext-pol}. In general, the theorem shows that there are subtle global obstructions for the existence of principal polarizations on a given superspecial $\calO$-abelian $k$-variety $(X, \iota)$ in addition to the necessary local conditions listed above.

This paper is organized as follows. In Section~\ref{sec:1}, we obtain a complete classification of $(\calO_v, \calO_v)$-bilattices and 
prove  Theorem~\ref{thm:4-indec-intro}  by combining Ribet's classification Theorem~\ref{thm:Ribet} with the classification of truncated $(\calO_v, \calO_v)$-bimodules. The results are then applied in Section~\ref{sec:2} to the classification of global quaternion bilattices and superspecial $\calO$-abelian varieties. In Section~\ref{sec:self-dual-local-latt}, we  give a complete classification of self-dual hermitian $(\calO_v, *, \calO_v)$-bilattices and prove Theorem~\ref{thm:self-dual-local-bilatt-intro}. To study hermitian forms on global quaternion bilattices in \eqref{eq:bij-pol-hermlatt},  we also treat  self-dual hermitian local bilattices at places where one of the quaternion orders  is split using variants of Morita equivalence techniques.
In Section~\ref{sec:QM}, we apply these classification results to the study of principally polarized superspecial $\calO$-abelian varieties and complete the proof of Theorem~\ref{thm:nec-suff-cond-intro}.

\section{Classification of local Ribet bilattices}
\label{sec:1}
\numberwithin{thm}{subsection} 
Let $F$ be a nonarchimedean local field with ring of integers $O_F$ and residue field $\F_q$, where $q$ is a power of a prime $p$.  
Let $B$ be the unique quaternion division algebra over $F$, and  $O_B$ be the unique maximal $O_F$-order in $B$. 
In this section, we provide a complete classification of $(O_B, O_B)$-bilattices up to isomorphism. 

Throughout this paper we follow the convention of \cite[\S1.1]{MR674652} for bimodules.  More precisely, let $R$ be a commutative ring with unity, and  $A$ and $D$ be two associative $R$-algebras. An $(A, D)$-bimodule $M$ is simultaneously a left $A$-module and a right $D$-module such that
\[
    (ax)d=a(xd),\qquad rx=xr
\]
for all $a\in A$, $d\in D$, $r\in R$ and $x\in M$. Equivalently, a $(A, D)$-bimodule is the same as a left $A\otimes_{R} D^{\opp}$-module, where $D^{\opp}$ is the opposite ring of $D$. Henceforth, we use the two equivalent formulations interchangeably and choose whichever is the most convenient. 
 
By definition,  an $O_F$-\emph{lattice} is a finitely generated torsion-free $O_F$-module. Therefore, an $(O_B, O_B)$-bilattice always means a finitely generated $O_F$-torsion-free $(O_B, O_B)$-bimodule, or  equivalently,  a left $O_B\otimes_{O_F} O_B^{\opp}$-lattice.
Since $B$ is division, every left (or right) $O_B$-lattice is necessarily free over $O_B$. In particular,  given an $(O_B, O_B)$-bilattice $L$, we have 
\[\rank_{\OpMod}(L)=\frac{1}{4}\rank_{O_F}(L)=\rank_{\ModOp} (L). \]
Thus it makes sense to say that ``$L$ has $O_B$-rank $n$'' without specifying ``left'' or ``right'', and we simply write $\rank_{O_B}(L)=n$.
Let $O_B^{\oplus n}$ be the free right $O_B$-module of rank $n$, whose elements are written as column vectors. Then the endomorphism ring $\End_{\ModOp}(O_B^{\oplus n})=\Mat_n(O_B)$ acts on $O_B^{\oplus n}$ from the left, so every embedding $\varphi: O_B\hookrightarrow\Mat_n(O_B)$ of $O_F$-algebras equips $O_B^{\oplus n}$ with a left $O_B$-module structure. Taking into account the natural right $O_B$-module structure on $O_B^{\oplus n}$, such a $\varphi$ further determines an $(O_B, O_B)$-bilattice structure on $O_B^{\oplus n}$, which will be denoted by $(O_B^{\oplus n}, \varphi)$. Conversely,  every $(O_B, O_B)$-bilattice is isomorphic to an $(O_B, O_B)$-bilattice of the form $(O_B^{\oplus n}, \varphi)$ since it is free as a right $O_B$-module. It is routine to check the following lemma; see \cite[\S 1, p.~10]{Ribet-bimod} or \cite[Remark~3.2]{MR2931385}.

\begin{lemma}\label{lem:conj-emb-bimodule}
 For any fixed positive integer $n\in \Z_{>0}$, the assignment $\varphi\mapsto (O_B^{\oplus n}, \varphi)$ induces a bijection between the following two sets:
  \[\left\{\parbox{4cm}{$\GL_n(O_B)$-conjugacy classes of embeddings $O_B\hookrightarrow
\Mat_n(O_B)$}\right\}\quad \xlongleftrightarrow{1-1}\quad \left\{\parbox{3.9cm}{Isomorphism
classes of $(O_B, O_B)$-bilattices of $O_B$-rank $n$}\right\} \]
\end{lemma}

Therefore, the classification of the $\GL_n(O_B)$-conjugacy classes of embeddings $O_B\hookrightarrow\Mat_n(O_B)$ is the same as the classification of the isomorphism classes of $(O_B, O_B)$-bilattices of $O_F$-rank $4n$, or equivalently, the isomorphism classes of left $O_B\otimes_{O_F}O_B^{\opp}$-lattices of $O_F$-rank $4n$. Both of the sets in Lemma~\ref{lem:conj-emb-bimodule} are finite by the Jordan-Zassenhaus Theorem \cite[Theorem~26.4]{reiner:mo}. Moreover, 
according to the Krull-Schmidt-Azumaya Theorem \cite[Theorem~6.12]{curtis-reiner:1}, every left $O_B\otimes_{O_F}O_B^{\opp}$-lattice is  uniquely  a finite direct sum of indecomposable left $O_B\otimes_{O_F}O_B^{\opp}$-lattices,  up to isomorphism and order of occurrence of the summands. To classify the $(O_B, O_B)$-bilattices, it is enough to classify the indecomposable ones.

Next, we write down four  $(O_B, O_B)$-bilattices in terms of suitable embeddings $O_B\hookrightarrow
\Mat_n(O_B)$. The first two of them appeared in Ribet's theorem as recalled in Theorem~\ref{thm:Ribet}, the other two were introduced by Terakado, Yu and the first named author in \cite{terakado-yu-xue-2026}. 
As we shall see in Theorem~\ref{thm:conj-class}, they form  a complete list of indecomposable $(O_B, O_B)$-bilattices up to isomorphism. 
Let $K$ be the unique unramified quadratic extension of $F$ with ring of integers $O_K$, and $\sigma: a\mapsto a^\sigma$ be the unique nontrivial automorphism of $K/F$.
Following \cite[Corollaire~II.1.7]{vigneras}, there exists a uniformizer $\Pi\in O_B$ such that $O_B=O_K\oplus O_K\Pi$ with multiplication rules  
\begin{equation}\label{eq:quat-multi-rul}
\Pi^2=\pi, \qquad \Pi a= a^\sigma \Pi, \quad \forall a\in O_K,  
\end{equation}
where $\pi$ denotes a fixed uniformizer of $F$. Consider the following  embeddings $O_B\hookrightarrow\Mat_n(O_B)$ with $n\leq 2$:
\begin{align}
    \id&:  O_B\hookrightarrow O_B,&\quad &\alpha\mapsto \alpha,& &&\forall &\alpha\in O_B;\label{eq:defn-id}\\
    \tau&: O_B\hookrightarrow O_B,&\quad &\alpha\mapsto \Pi^{-1}\alpha\Pi,& &&\forall &\alpha\in O_B;\label{eq:defn-tau}\\
    \varphi_1&: O_B\hookrightarrow\Mat_2(O_B),&\quad &\Pi\mapsto\begin{bmatrix}
        0 & \pi\\ 1 & 0 
    \end{bmatrix},&\qquad &a\mapsto\begin{bmatrix}
        a & 0\\ 0 & a^\sigma
    \end{bmatrix},&\quad \forall &a\in O_K;\label{eq:defn-varphi-1}\\
    \varphi_2&: O_B\hookrightarrow\Mat_2(O_B),&\quad &\Pi\mapsto\begin{bmatrix}
        0 & \pi\\ 1 & 0 
    \end{bmatrix},&\qquad &a\mapsto\begin{bmatrix}
        a^\sigma & 0\\ 0 & a
    \end{bmatrix},&\quad \forall &a\in O_K. \label{eq:defn-varphi-2}
\end{align}
We denote the corresponding $(O_B, O_B)$-bilattices respectively by
\begin{equation}\label{eq:defn-four-indecomp-latt}
      \mathbb{O}\coloneqq(O_B, \id),\quad \mathbb{P}\coloneqq(O_B, \tau),\quad \mathbb{L}_1\coloneqq(O_B^{\oplus 2},\varphi_1),\quad \mathbb{L}_2\coloneqq(O_B^{\oplus 2}, \varphi_2).
\end{equation}
Observe that $\O$ is just $O_B$ equipped with the natural $(O_B, O_B)$-bilattice structure, while the left multiplication by $\Pi$ defines an $(O_B, O_B)$-bimodule isomorphism from $\P$ to the maximal ideal $\grP\coloneqq\Pi O_B$ of $O_B$. In the argument of this paper, it is often more convenient to use $\mathbb{P}$ than $\grP$, especially in Section~\ref{sec:self-dual-local-latt} for the classification of self-dual hermitian local bilattices.
Given two embeddings $\psi_1: O_B\hookrightarrow \Mat_{n_1}(O_B)$ and
$\psi_2: O_B\hookrightarrow \Mat_{n_2}(O_B)$, we define $\psi_1\oplus
\psi_2$ as the block diagonal embedding
\[\psi_1\oplus
\psi_2: O_B\hookrightarrow \Mat_{n_1+n_2}(O_B), \qquad
  \alpha\mapsto
  \begin{bmatrix}
    \psi_1(\alpha) & \\ & \psi_2(\alpha)
  \end{bmatrix}, \quad \forall \alpha\in O_B. \]
Clearly, this direct sum operation is associative, and we can inductively define the direct sum of an arbitrary finite set of such embeddings. In terms of bilattices, this just says that $(O_B^{\oplus (n_1+n_2)}, \psi_1\oplus\psi_2)=(O_B^{\oplus n_1}, \psi_1)\oplus (O_B^{\oplus n_2},\psi_2)$.

Our main result on the classification of the $\GL_n(O_B)$-conjugacy classes of embeddings $O_B\hookrightarrow\Mat_n(O_B)$ is as follows.

\begin{thm}\label{thm:conj-class}
Every embedding $O_B\hookrightarrow\Mat_n(O_B)$ is $\GL_n(O_B)$-conjugate to 
    \begin{equation}\label{eq:emb-str-inv}
    \id^{\oplus r_1}\oplus \tau^{\oplus r_2}\oplus
  \varphi_1^{\oplus t_1}\oplus \varphi_2^{\oplus t_2}
  \end{equation}
for a unique quadruple $(r_1, r_2, t_1, t_2)\in\Z_{\geq 0}^4$ of nonnegative integers. Equivalently, in terms of lattices, $\{\O, \P, \L_1, \L_2\}$ forms a complete list of indecomposable $(O_B, O_B)$-bilattices up to isomorphism, and every $(O_B, O_B)$-bilattice $L$ is isomorphic to a direct sum 
\begin{equation}\label{eq:bilatt-str-inv}
    \O^{\oplus r_1}\oplus \P^{\oplus r_2}\oplus
  \L_1^{\oplus t_1}\oplus \L_2^{\oplus t_2}\simeq
    O_{B}^{\oplus r_1}\oplus \grP^{\oplus r_2}\oplus
  \L_1^{\oplus t_1}\oplus \L_2^{\oplus t_2}\
\end{equation}
for some  $(r_1, r_2, t_1, t_2)\in\Z_{\geq 0}^4$ uniquely determined by $L$.  The quadruple $(r_1, r_2, t_1, t_2)$ will be called the \emph{structural invariant} of the $(O_B, O_B)$-bilattice $L$ and denoted by $\ulm(L)\coloneqq (r_1, r_2, t_1, t_2)$.
\end{thm}

In particular, the $O_F$-order $O_B\otimes_{O_F}O_B\cong O_B\otimes_{O_F}O_B^{\opp}$ has \emph{finite representation type}.   Recall from Theorem~\ref{thm:Ribet} that a special form of Theorem~\ref{thm:conj-class} was previously obtained by Ribet \cite{Ribet-bimod} under an extra \emph{admissible} condition.  Our proof of Theorem~\ref{thm:conj-class} is based on Ribet's result and the classification of truncated $(O_B, O_B)$-bimodules $L/L \Pi $ to be discussed in Section~\ref{subsec:trunc-bimod}.  See \cite{terakado-yu-xue-2026} for an alternative proof of Theorem~\ref{thm:conj-class} in the case $n=2$.

\begin{remark}\label{rem:str-inv-comp-conj}
    For each $\beta\in B^\times$, the inner automorphism $\xint(\beta): x\mapsto \beta x \beta^{-1}$ of $B$ stabilizes $O_B$. Thus taking the composition of an embedding $\psi: O_B\hookrightarrow \Mat_n(O_B)$ with $\xint(\beta)$ produces another such  embedding $\psi\circ \xint(\beta)$.  Let $\ord_B: B^\times\twoheadrightarrow \Z$ be the normalized discrete valuation on $B$. If $\ord_B(\beta)$ is even (i.e.~$\beta=u \pi^s$ for some $u\in O_B^\times$ and $s\in \Z$), then $\psi\circ \xint(\beta)$ is $\GL_n(O_B)$-conjugate to $\psi$. On the other hand,  the composition with $\xint(\Pi)$ acts as transpositions on $\{\id, \tau\}$ and $\{\varphi_1, \varphi_2\}$ respectively. Therefore, if we write $\ulm(O_B^n, \psi)=(r_1, r_2, t_1, t_2)$, then 
    \begin{equation}
       \ulm\big(O_B^n, \psi\circ\xint(\beta)\big)=\begin{dcases*}
           \ulm(O_B^n, \psi), & if $\ord_B(\beta)$ is even;\\
           (r_2, r_1, t_2, t_1), & if $\ord_B(\beta)$ is odd.
       \end{dcases*} 
    \end{equation}
\end{remark}

\subsection{Classification of truncated $(O_B, O_B)$-bimodules}
\label{subsec:trunc-bimod}
Let $L$ be an $(O_B, O_B)$-bilattice of $O_B$-rank $n$. Since $O_B$ is normalized by its uniformizer $\Pi$, the sublattice  $L\Pi\subset L$ is actually an $(O_B, O_B)$-sub-bimodule. To study isomorphism classes of $(O_B, O_B)$-bilattices, we first take a close look at the \emph{truncated  $(O_B, O_B)$-bimodule}  $\overline{L}\coloneqq L/(L\Pi)$.

 We present $O_B$ as $O_K\oplus O_K\Pi$ with multiplication rules \eqref{eq:quat-multi-rul} as before. 
 Since $O_B/(\Pi)=O_K/\pi O_K=\F_{q^2}$, the quotient $\overline{L}$ is naturally a right $\F_{q^2}$-vector space of dimension $n$. The left $O_B$-module structure on $\overline{L}$ factors through the quotient $O_B\to O_B/\pi O_B$.  If we denote the canonical image of $\Pi$ in $O_B/\pi O_B$ by $\varepsilon$, then $O_B/\pi O_B=\F_{q^2}[\varepsilon]=\F_{q^2}\oplus\F_{q^2}\varepsilon$, whose multiplication rules  are given by
 \begin{equation}\label{eq:emb-11}
   \varepsilon^2=0, \qquad \varepsilon a =a^{q} \varepsilon, \qquad
   \forall  a\in \F_{q^2}. 
 \end{equation}
Therefore, $\overline{L}$ is equipped with an $(\F_{q^2}[\varepsilon], \F_{q^2})$-bimodule structure, or equivalently, a left $\F_{q^2}[\varepsilon]\otimes_{\F_q}\F_{q^2}$-module structure.

We first present a description of the structure of left $\F_{q^2}[\varepsilon]\otimes_{\F_q}\F_{q^2}$-modules in terms of linear algebra. It is routine to derive the following lemma.
\begin{lemma}\label{lem:linear-alg-descrip-trun-bimod}
    Let $W$ be a left $\F_{q^2}[\varepsilon]\otimes_{\F_q}\F_{q^2}$-module and view it as a left $\F_{q^2}\otimes_{\F_q}\F_{q^2}$-module by restriction of scalars. Then the isomorphism
\begin{equation}\label{eq:tens-decomp}
    \F_{q^2}\otimes_{\F_q}\F_{q^2}\xrightarrow{\sim}\F_{q^2}\oplus \F_{q^2},\quad a\otimes b\mapsto (ab, a^{q}b)
\end{equation}
induces a decomposition of $W$ into the direct sum $W^0\oplus W^1$ of two right $\F_{q^2}$-vector spaces such that $\varepsilon W^0\subseteq W^{1}$ and $\varepsilon W^1\subseteq W^0$, where
\begin{equation}\label{eq:defn-W-0-W-1}
W^0\coloneqq \{w\in W\mid aw=wa,\, \forall a\in \F_{q^2}\}, \quad
  W^1\coloneqq \{w\in W\mid aw=wa^{q}, \, \forall a\in \F_{q^2}\}.
   \end{equation}
Conversely, given a right $\F_{q^2}$-vector space $W$ and the following data:
\begin{enumerate}[(i)]
    \item a decomposition $W=W^0\oplus W^1$ of $W$ into the direct sum of two right $\F_{q^2}$-vector spaces,
     \item two right $\F_{q^2}$-linear maps $\varepsilon^{(0)}: W^0\to W^1$ and $\varepsilon^{(1)}:W^1\to W^0$ such that $\varepsilon^{(i)}\varepsilon^{(i+1)}=0$ for all $i\in\Z/2\Z$,
\end{enumerate}
the space $W$ can be endowed with a left $\F_{q^2}[\varepsilon]$-module structure such that it forms an $(\F_{q^2}[\varepsilon], \F_{q^2})$-bimodule by defining
\begin{equation}
        a(w^0, w^1)\coloneqq (w^0a, w^1a^q),\qquad  \varepsilon(w^0, w^1)\coloneqq (\varepsilon^{(1)}w^{1}, \varepsilon^{(0)}w^0), \qquad \forall a\in \F_{q^2}, \forall w^i\in W^i. 
\end{equation}
\end{lemma}

\begin{eg}\label{eg:indecomp-trunc}
Consider the $(O_B, O_B)$-bilattices $\O=(O_B, \id)$, $\P=(O_B, \tau)$, $\L_1=(O_B^{\oplus 2},\varphi_1)$ and $\L_2=(O_B^{\oplus 2}, \varphi_2)$ in \eqref{eq:defn-four-indecomp-latt}. We write down the corresponding truncated $(O_B, O_B)$-bimodules as follows:
    \begin{align*}
       W_{\id}&=\F_{q^2}\oplus 0, & \varepsilon^{(0)}&=0,
       &\varepsilon^{(1)}&=0,\\
              W_{\tau}&=0\oplus \F_{q^2}, & \varepsilon^{(0)}&=0,
       &\varepsilon^{(1)}&=0,\\
       W_{\varphi_1}&=\F_{q^2}\oplus \F_{q^2}, & \varepsilon^{(0)}&=1,
       &\varepsilon^{(1)}&=0,\\
              W_{\varphi_2}&=\F_{q^2}\oplus \F_{q^2}, & \varepsilon^{(0)}&=0,
       &\varepsilon^{(1)}&=1.
     \end{align*}
     In particular, $W_{\varphi_1}$ and $W_{\varphi_2}$ are non-isomorphic.
\end{eg}

In more technical terms, the $\F_{q}$-algebra $\F_{q^2}[\varepsilon]\otimes_{\F_q}\F_{q^2}$ is an example of a \emph{Nakayama algebra}, which we now recall. Let $A$ be a finite dimensional algebra over a field. Following \cite[Chapter~IV, pp.~111--112]{MR1314422}, a left $A$-module $M$ is \emph{uniserial} if there is only one composition series for $M$. The algebra $A$ is called a Nakayama algebra if all indecomposable projective left $A$ and $A^{\opp}$-modules are uniserial.
\begin{prop}\label{prop:trun-bimod-2}
    The $\F_{q}$-algebra $\F_{q^2}[\varepsilon]\otimes_{\F_q}\F_{q^2}$ is a Nakayama algebra. Moreover, the following statements hold.
 \begin{enumerate}[(1)]
     \item Up to isomorphism, $\{W_{\varphi_1}, W_{\varphi_2}\}$ forms a complete list of indecomposable projective left modules over $\F_{q^2}[\varepsilon]\otimes_{\F_q}\F_{q^2}$, and there is an isomorphism of left $\F_{q^2}[\varepsilon]\otimes_{\F_q}\F_{q^2}$-modules:
    \begin{equation}\label{eq:decom-trunc-ring}
        \F_{q^2}[\varepsilon]\otimes_{\F_q}\F_{q^2}\simeq W_{\varphi_1}\oplus W_{\varphi_2}.
    \end{equation}
     \item Up to isomorphism, $\{W_{\id}, W_{\tau}, W_{\varphi_1}, W_{\varphi_2}\}$ forms a complete list of indecomposable left $\F_{q^2}[\varepsilon]\otimes_{\F_q}\F_{q^2}$-modules. Therefore, every finitely generated $(\F_{q^2}[\varepsilon], \F_{q^2})$-bimodule $W$ admits an isomorphism 
\begin{equation}\label{eq:trunc-bimod-decomp}
      W\simeq W_{\id}^{\oplus r_1}\oplus W_{\tau}^{\oplus r_2}\oplus 
                   W_{\varphi_1}^{\oplus t_1}\oplus W_{\varphi_2}^{\oplus t_2}
\end{equation}
for a unique quadruple $(r_1, r_2, t_1, t_2)\in \Z_{\geq 0}^4$.
 \end{enumerate}
\end{prop}

\begin{proof}
    By definition, to check that $A\coloneqq\F_{q^2}[\varepsilon]\otimes_{\F_q}\F_{q^2}$ is a Nakayama algebra, we should first find all indecomposable projective left $A$ and $A^{\opp}$-modules. We identify $\F_{q^2}\otimes_{\F_q}\F_{q^2}$ with $\F_{q^2}\oplus\F_{q^2}$ via isomorphism \eqref{eq:tens-decomp}. Then we rewrite $A=\F_{q^2}[\varepsilon]\otimes_{\F_q}\F_{q^2}$ as 
\begin{equation*}
    A=\F_{q^2}[\varepsilon]\otimes_{\F_q}\F_{q^2}=(\F_{q^2}\otimes_{\F_q}\F_{q^2})[\varepsilon\otimes 1] 
     =(\F_{q^2}\oplus \F_{q^2})[\varepsilon\otimes 1].
\end{equation*}
Since $(\varepsilon\otimes 1)(a\otimes b)=(a^q\otimes b)(\varepsilon\otimes 1)$ for all $a, b\in \F_{q^2}$, the multiplication rules of $A=(\F_{q^2}\oplus \F_{q^2})[\varepsilon\otimes 1]$ are given by
\begin{equation}\label{eq:mult-rul-trun-order}
    (\varepsilon\otimes 1)(a_0, a_1)=(a_1, a_0)(\varepsilon\otimes1), \quad\forall a_0, a_{1}\in \F_{q^2}.
\end{equation}
Let $e_1\coloneqq (1, 0)$ and $e_2\coloneqq (0, 1)$ be the nontrivial idempotents of $\F_{q^2}\oplus\F_{q^2}$. Then
\begin{equation*}\label{eq:A-decom-into-right-ideal}
    A=Ae_1\oplus Ae_2 
\end{equation*}
is a decomposition of $A$ into a direct sum of its left ideals. From \eqref{eq:mult-rul-trun-order}, it follows that $(\varepsilon\otimes 1)A$ is a nilpotent two-sided ideal. On the other hand, the $\F_{q}$-algebra $A/(\varepsilon\otimes 1)A=\F_{q^2}\oplus\F_{q^2}$ is semisimple, so the Jacobson radical $\Rad(A)=(\varepsilon\otimes 1)A$ by \cite[Proposition~3.3, Chapter~I]{MR1314422}. Then by \cite[Theorem~6.7 (iii)]{curtis-reiner:1} both $Ae_1$ and $Ae_2$ are indecomposable left ideals. Further, we express $Ae_1$ as
\[
          Ae_1=(\F_{q^2}e_1)\oplus (\varepsilon\otimes 1)(\F_{q^2}e_1),
\]
from which it follows that $(Ae_1)^0=\F_{q^2}e_1$ and $(Ae_1)^1=(\varepsilon\otimes 1)(\F_{q^2}e_1)$.
 Clearly, the left multiplication by $\varepsilon\otimes 1$ on $Ae_1$ induces an isomorphism $\varepsilon^{(0)}: (Ae_1)^{0}\to (Ae_1)^{1}$ and a zero map $\varepsilon^{(1)}: (Ae_1)^{1}\to (Ae_1)^{0}$. By the description of $W_{\varphi_1}$ in Example~\ref{eg:indecomp-trunc}, we find that $Ae_1$ is isomorphic to $W_{\varphi_1}$. Likewise,  $Ae_2$ is isomorphic to $W_{\varphi_2}$. In particular, $Ae_1$ and $Ae_2$ are non-isomorphic. Hence, by \cite[Theorem~6.17(i)]{curtis-reiner:1}, $\{W_{\varphi_1}, W_{\varphi_2}\}=\{Ae_1, Ae_2\}$ forms a complete list of indecomposable projective left $\F_{q^2}[\varepsilon]\otimes_{\F_q}\F_{q^2}$-modules up to isomorphism.

Let $U$ be a maximal proper  $(\F_{q^2}[\varepsilon],\F_{q^2})$-sub-bimodule of $W_{\varphi_j}$ for $j\in \{1, 2\}$. From the maximality of $U$ and Nakayama's lemma, we have  $U\supseteq \Rad(A)W_{\varphi_j}=\varepsilon W_{\varphi_j}$.
 On the other hand, from Example~\ref{eg:indecomp-trunc}, $W_{\varphi_j}/\varepsilon W_{\varphi_j}$ has $\F_{q^2}$-dimension $1$, and so $U=\varepsilon W_{\varphi_j}$. Since $\varepsilon W_{\varphi_j}$ also has $\F_{q^2}$-dimension $1$, there can be only one composition series for $W_{\varphi_j}$, namely, 
\[
       W_{\varphi_j}\supsetneq \varepsilon W_{\varphi_j} \supsetneq \varepsilon^2 W_{\varphi_j}=0.
\]
From \eqref{eq:emb-11}, $\F_{q^2}[\varepsilon]^{\opp}\cong \F_{q^2}[\varepsilon]$, and so $A^{\opp}=\F_{q^2}[\varepsilon]^{\opp}\otimes_{\F_q}\F_{q^2}\cong A$. Hence every indecomposable projective left $A^{\opp}$-module also has exactly one composition series. 
Now we have proven that $A=\F_{q^2}[\varepsilon]\otimes_{\F_q}\F_{q^2}$ is a Nakayama $\F_q$-algebra and statement~(1) of the proposition.

According to \cite[Theorem~2.1, Chapter~VI]{MR1314422},  every indecomposable module over a Nakayama algebra is uniserial and is a quotient of an indecomposable projective module. Therefore, every indecomposable left $A$-module is of the form $W_{\varphi_j}/\varepsilon^i W_{\varphi_j}$ for some $1\leq i, j\leq 2$. By Example~\ref{eg:indecomp-trunc}, we have $W_{\varphi_1}/\varepsilon W_{\varphi_1}\simeq W_{\id}$ and $W_{\varphi_2}/\varepsilon W_{\varphi_2}\simeq W_{\tau}$. The proof of statement~(2) of the proposition is now complete by the Krull-Schmidt-Azumaya Theorem \cite[Theorem~6.12]{curtis-reiner:1}.
\end{proof}

Combining Lemma~\ref{lem:linear-alg-descrip-trun-bimod} with Proposition~\ref{prop:trun-bimod-2}, we obtain the following result. 

\begin{cor}\label{cor:find-decomp-for-trunc-bimod}
    Let $W$ be a finitely generated $(\F_{q^2}[\varepsilon], \F_{q^2})$-bimodule. Put
    \begin{equation*}
  m_0\coloneqq \dim_{\F_{q^2}} W^0, \quad  m_1\coloneqq \dim_{\F_{q^2}} W^1, \quad  
  d_0\coloneqq \dim_{\F_{q^2}}(\varepsilon W^0),\quad 
  d_1\coloneqq \dim_{\F_{q^2}} (\varepsilon W^1). 
\end{equation*}
Then we have an isomorphism
\begin{equation*}
W\simeq  W_{\id}^{\oplus (m_0-d_0-d_1)}\oplus W_{\tau}^{\oplus (m_1-d_0-d_1)}\oplus
  W_{\varphi_1}^{\oplus d_0}\oplus W_{\varphi_2}^{\oplus d_1}.
\end{equation*}
\end{cor}

We conclude this subsection by writing down all indecomposable projective left $O_B\otimes_{O_F}O_B^{\opp}$-lattices, which play a key role in our proof for Theorem~\ref{thm:conj-class}. For this, we need the following lemma, which shows that to examine whether a homomorphism of $(O_B, O_B)$-bilattices is an isomorphism, we may pass to truncated $(O_B, O_B)$-bimodules.
\begin{lemma}\label{lem:lem13}
Let $f: L\to L'$ be a homomorphism of $(O_B, O_B)$-bilattices. If the induced homomorphism $\bar{f}: L/L\Pi \to L'/L'\Pi$ between the truncated $(O_B, O_B)$-bimodules is injective (resp.~surjective), then $f$ itself is injective (resp.~surjective).
\end{lemma}
\begin{proof}
    The surjectivity assertion follows immediately from  Nakayama's lemma. Suppose that $\bar{f}$ is injective. 
    If $x\in\ker f$, then $\bar{x}\in\ker\bar{f}$, where $\bar{x}$ denotes the image of $x$  in $L/L\Pi$. By the injectivity of $\bar{f}$ we get  $\bar{x}=0$, that is, $x\in L\Pi$.  
    Write $x=y\Pi$ for some $y\in L$. Then $0=f(x)=f(y)\Pi$, so $f(y)=0$ since $L'$ is torsion-free. Likewise, we have $y=z\Pi$ for some $z\in L$, and $f(z)=0$. Thus $x=z\Pi^2\in  L\Pi^2$. Recursively, we find that $x\in\bigcap_{i=1}^\infty L\Pi^i=\{0\}$, which proves the injectivity of $f$.
\end{proof}
\begin{lemma}\label{lem:pro-latt}
   The bilattices  $\L_1, \L_2$ in \eqref{eq:defn-four-indecomp-latt} form a complete list of indecomposable projective left $O_B\otimes_{O_F}O_B^{\opp}$-modules up to isomorphism.
    Moreover, we have $O_B\otimes_{O_F} O_B^{\opp}\simeq \L_1\oplus \L_2$ as  left $O_B\otimes_{O_F} O_B^{\opp}$-modules.
\end{lemma}
\begin{proof}
    By Proposition~\ref{prop:trun-bimod-2}, we have written down in \eqref{eq:decom-trunc-ring}  the decomposition  of $\F_{q^2}[\varepsilon]\otimes_{\F_q}\F_{q^2}$ into its indecomposable left ideals. 
Let $\Lambda\coloneqq O_B\otimes_{O_F} O_B^{\opp}$ and $I\coloneqq O_B\otimes (\Pi O_B^{\opp})$. Then $I$ is a two-sided ideal of $\Lambda$ contained in the Jacobson radical of $\Lambda$ with quotient ring
     \[        \Lambda/I=O_B \otimes_{O_F} (O_B^{\opp}/\Pi O_B^{\opp})=O_B\otimes_{O_F}\F_{q^2}=\F_{q^2}[\varepsilon]\otimes_{\F_q} \F_{q^2}.        \]
   According to \cite[Theorems~6.8 and 6.17(i)]{curtis-reiner:1}, there is a decomposition 
   \[
       O_B\otimes_{O_F} O_B^{\opp}=N_1\oplus N_2
   \]
   of $O_B\otimes_{O_F} O_B^{\opp}$
   into its indecomposable left ideals such that $N_j/N_j\Pi\simeq W_{\varphi_j}$ for all $1\leq j\leq 2$, and $\{N_1, N_2\}$ forms a complete list of indecomposable projective left $O_B\otimes_{O_F} O_B^{\opp}$-modules up to isomorphism.
We finish the proof by showing that $N_j\simeq \L_j$ for each $1\leq j\leq 2$. By the definition of $W_{\varphi_j}$ in Example~\ref{eg:indecomp-trunc}, there is an isomorphism $g_j: N_j/ N_j\Pi\to \L_j/ \L_j\Pi$ for each $1\leq j\leq 2$. Since $N_j$ is projective, each $g_j$ lifts to a homomorphism $f_j: N_j\to \L_j$ of $(O_B, O_B)$-bilattices, which is in fact an isomorphism by Lemma~\ref{lem:lem13}.
\end{proof}

\subsection{Classification of $(O_B, O_B)$-bilattices}\label{subsec:emb-conj-cls}
In this subsection, we prove our main Theorem~\ref{thm:conj-class} on the classification of $(O_B, O_B)$-bilattices.
The proof makes use of the fact that $O_B\otimes_{O_F}O_B^{\opp}$ is a \emph{Gorenstein order}. For the convenience of the reader, we review the notion of Gorenstein orders and \emph{Bass orders}; for detailed discussions, see \cite[\S 37]{curtis-reiner:1}, \cite[Chapter~IX]{Roggenkamp-Latt-II}, or \cite{Drozd-Kirichenko-Roiter-1967}. Let $R$ be a Dedekind domain, let $A$ be a finite dimensional separable algebra over the quotient field of $R$, and let $\Lambda$ be an $R$-order in $A$. We call $\Lambda$ a Gorenstein order if every short exact sequence of left $\Lambda$-lattices
\[ 0\to \Lambda \to M \to N \to 0 \]
is split over $\Lambda$. For a left $\Lambda$-lattice $M$, its dual  $M^*$ is defined as  
\[ M^*\coloneqq\Hom_R(M, R), \]
which has a natural right $\Lambda$-lattice structure induced by the left $\Lambda$-lattice structure of $M$.  
In particular, the dual $\Lambda^*$ of $\Lambda$ has a natural $(\Lambda, \Lambda)$-bilattice structure induced by the natural $(\Lambda, \Lambda)$-bilattice structure of $\Lambda$. By \cite[Proposition~37.8]{curtis-reiner:1}, $\Lambda$ is a Gorenstein order if and only if $\Lambda^*$ is projective as a left $\Lambda$-module, or equivalently, as a right $\Lambda$-module.
The $R$-order $\Lambda$ is called a Bass order if every $R$-order in $A$ containing $\Lambda$ (including $\Lambda$ itself) is a Gorenstein order. In particular, Bass orders must be Gorenstein orders.

Classically, Janusz \cite{Janusz-1979-JLMS} first  gave a criterion for the tensor product of two $R$-orders $\Lambda_1\otimes_R\Lambda_2$ to be hereditary, and his work was further extended by Hijikata and Nishida \cite{MR1648352}. Naturally, one may consider similar questions by replacing the adjective ``hereditary" by ``Gorenstein" or ``Bass". The following result seems to be  well known to the experts, but we have failed to locate a better reference. 

\begin{lemma}
 The tensor product $\Lambda_1\otimes_R\Lambda_2$ of two Gorenstein $R$-orders $\Lambda_1, \Lambda_2$ remains Gorenstein.     
\end{lemma}

\begin{proof}
   Since each $(\Lambda_i)^*$ is a projective left $\Lambda_i$-module for  $i\in\{1, 2\}$, there exists a left $\Lambda_i$-module $N_i$ such that there is an isomorphism
$\Lambda_i^*\oplus N_i\simeq (\Lambda_i)^{\oplus s_i} $
of left $\Lambda_i$-modules for some integer $s_i>0$. It follows that 
\[
    (\Lambda_1\otimes_{R}\Lambda_2)^*\simeq \Lambda_1^*\otimes_{R}\Lambda_2^*
\]
is a direct summand of $(\Lambda_1\otimes_{R}\Lambda_2)^{\oplus s_1s_2}$, and hence is a projective left $\Lambda_1\otimes_{R}\Lambda_2$-module. 
\end{proof}

Consequently, $O_B\otimes_{O_F}O_B^{\opp}$ is a Gorenstein order since maximal orders are Gorenstein.

\begin{lemma}\label{lem:pro-lift-direct-summ}
     Let $L$ be an $(O_B, O_B)$-bilattice. If $W_{\varphi_1}$ (resp.~$W_{\varphi_2}$) embeds in $L/L\Pi$, then $L$ has a direct summand isomorphic to $\L_1$ (resp.~$\L_2$).
\end{lemma}
\begin{proof}
    Since $\L_1$ is projective by Lemma~\ref{lem:pro-latt} and $\L_1/ \L_1\Pi\cong W_{\varphi_1}$, an inclusion $\bar{j}: W_{\varphi_1}\hookrightarrow L/ L\Pi$ lifts to a homomorphism $\L_1\to L$, which is actually an embedding $j: \L_1\hookrightarrow L$ by Lemma~\ref{lem:lem13}. For the rest of the proof, we identify $\L_1$ with its image $j(\L_1)$ in $L$. We claim that the lemma would follow from the fact that $L/\L_1$ is $O_F$-torsion-free, that is, $L/\L_1$ is also an $(O_B, O_B)$-bilattice. Indeed, if this is true, then we have an exact sequence of left $O_B\otimes_{O_F} O_B^{\opp}$-lattices
\begin{equation}\label{eq:e13-1}
    0\to \L_1 \xrightarrow{j}  L \to L/\L_1 \to 0,
\end{equation}
which is necessarily split over $O_F$. So applying $\Hom_{O_F}(\cdot, O_F)$ to it yields an exact sequence
\begin{equation}\label{eq:e13-2}
    0\to (L/\L_1)^* \to L^* \xrightarrow{j^*} \L_1^* \to 0.
\end{equation}
Recall that $\Lambda\coloneqq O_B\otimes_{O_F} O_B^{\opp}$ is a Gorenstein order, so 
$\Lambda^*$ is a projective right $\Lambda$-module. By Lemma~\ref{lem:pro-latt},  $\L_1^*$ is a direct summand of $\Lambda^*$, so $\L_1^*$ is also a projective right $\Lambda$-module, and hence the exact sequence \eqref{eq:e13-2} is split over $\Lambda$. Applying $\Hom_{O_F}(\cdot, O_F)$ to \eqref{eq:e13-2}, we recover the original exact sequence \eqref{eq:e13-1}, so that it is also split over $\Lambda$. It follows that $\L_1$ is a direct summand of $L$ and verifies our claim. 

Now it remains to check that $L/\L_1$ is $O_F$-torsion-free. Suppose otherwise. Then there exists $x\in L$ and $l\in\Z_{>0}$ such that $x\notin \L_1$, but $\pi^l x\in \L_1$. 
In view of $\Pi^2=\pi$ in \eqref{eq:quat-multi-rul}, there exists an integer $s\geq 0$ such that $y\coloneqq  x\Pi^s$ does not belong to $\L_1$ but $z\coloneqq  y\Pi$ lies in $\L_1$.
The condition $y\notin \L_1$ implies that the image $\bar{z}$ of $z$ in $\L_1/\L_1\Pi$ is nonzero, so $\bar{j}(\bar{z})$ is nonzero since $\bar{j}$ is an inclusion. On the other hand, $\bar{j}(\bar{z})$ is equal to the image of $j(z)=z=y\Pi$ in $L/L\Pi$, which is clearly zero; this is a contradiction, and proves that $L/\L_1$ is $O_F$-torsion-free.
\end{proof}

\begin{proof}[Proof of Theorem~\ref{thm:conj-class}]
From Lemma~\ref{lem:conj-emb-bimodule}, the classification of $\GL_n(O_B)$-conjugacy classes of embeddings $O_B\hookrightarrow \Mat_n(O_B)$ is equivalent to the classification of isomorphism classes of $(O_B, O_B)$-bilattices, so we focus on the latter in this proof. 
The uniqueness of the decomposition \eqref{eq:bilatt-str-inv} is guaranteed  by the Krull-Schmidt-Azumaya Theorem \cite[Theorem~6.12]{curtis-reiner:1}, so it is enough to prove that every $(O_B, O_B)$-bilattice $L$ is isomorphic to a direct sum of copies of $\O$, $\P$, $\L_1$ and $\L_2$ given by \eqref{eq:defn-four-indecomp-latt}. 
Applying Lemma~\ref{lem:pro-lift-direct-summ} repeatedly if necessary, we may assume that $L/L\Pi$ is isomorphic to a direct sum of copies of $W_{\id}$ and $W_{\tau}$. By the definitions of $W_{\id}$ and $W_{\tau}$ in Example~\ref{eg:indecomp-trunc}, the left multiplication by $\Pi$ induces a zero map on $L/L\Pi$, that is, $\Pi L\subseteq L\Pi$. Since $L/\Pi L$ and $L/ L\Pi$ have the same $\F_q$-dimension, we have the equality $\Pi L=L \Pi$. Then by the classification result of Ribet \cite[Theorem~1.2]{Ribet-bimod}, $L$ is isomorphic to a direct sum of copies of $\O$, $\P$;  see also Theorem~\ref{thm:Ribet}.
\end{proof}

The following corollary generalizes Ribet \cite[Theorem~1.3]{Ribet-bimod} and follows immediately by combining Theorem~\ref{thm:conj-class} with Proposition~\ref{prop:trun-bimod-2} (2).

\begin{cor}\label{cor:one-to-one-corre}
    Taking the right truncation $L/L\Pi$ of an $(O_B, O_B)$-bilattice $L$ induces a one-to-one correspondence between the set of isomorphism classes of $(O_B, O_B)$-bilattices and the set of isomorphism classes of finitely generated $(\F_{q^2}[\varepsilon], \F_{q^2})$-bimodules. 

    By symmetry, the same holds true if the right truncation $L/L\Pi$ and the $(\F_{q^2}[\varepsilon], \F_{q^2})$-bimodules are replaced respectively by the left truncation $L/\Pi L$ and the  $(\F_{q^2}, \F_{q^2}[\varepsilon])$-bimodules.
\end{cor}

Now we explain briefly that the $O_F$-order $O_B\otimes_{O_F} O_B\cong O_B\otimes_{O_F} O_B^{\opp}$ is not a Bass order. This fact will not be used elsewhere in this paper, so we leave the details of the proof to the interested reader.  To state an equivalent characterization of Bass orders, we recall the following  notion from \cite[\S 37]{curtis-reiner:1}. 
Let $M, N$ be two left (or right) modules over an order $\Lambda$. We say that $M$ \emph{covers} $N$, and write $M\succ N$, if $N=\sum f(M)$ with $f$  running through all elements of  $\Hom_{\Lambda}(M,N)$ in the summation.
 If $M$ does not cover $N$, we  write $M\nsucc N$. 

\begin{lemma}[{\cite[Lemma~37.14]{curtis-reiner:1}}]
Suppose that $R$ is a Dedekind domain and that $\Lambda$ is 
    an $R$-order in a separable algebra over the quotient field of $R$. Then $\Lambda$ is a Bass order if and only if for all left $\Lambda$-lattices $M$ and $N$, $M\succ N$ implies  $M^*\succ N^*$. 
\end{lemma}

\begin{remark}\label{rem:not-bass}
  As an application of Corollary~\ref{cor:one-to-one-corre}, 
it is straightforward  to check that for the four indecomposable $(O_B, O_B)$-bilattices given by \eqref{eq:defn-four-indecomp-latt},  $\O$ and $\P$ are dual to each other, and both $\L_1$ and $\L_2$ are self-dual. Furthermore, one  also checks that $\L_1\succ \O$ and $\L_1\nsucc \P$, from which it follows directly that $ O_B\otimes_{O_F} O_B^{\opp}$ is not a Bass order. 
 Indeed, suppose otherwise. Then $\L_1\succ \O$ implies that $\L_1^*\succ \O^*$, that is, $\L_1\succ \P$, leading to a contradiction.  
\end{remark}

\section{Global Ribet bilattices and superspecial abelian varieties with quaternion action}\label{sec:2}
 As the title suggests, the goal of this section is twofold. In Section~\ref{subsec:glob-latt}, we  classify both the genera of global Ribet bilattices over given maximal quaternion orders and also the isomorphism classes of such bilattices within each genus. Such classifications are then applied  in Section~\ref{subsec:ssp-O-1-ab-var} to the study of (unpolarized) superspecial abelian varieties with quaternion action. 

\subsection{Classification of global $(O_1, O_2)$-bilattices}\label{subsec:glob-latt}
Let $F$ be a global field with ring of integers $O_F$. More explicitly, we fix a finite set $\bfS_\infty$ of places of $F$ including all the archimedean ones and write $O_F$ for the ring of $\bfS_\infty$-integers.  As usual, $\bfS_\infty$ will be called the set of infinite places of $F$, and places outside $\bfS_\infty$ will be called finite places. Let $\widehat{O}_F\coloneqq \prod_{v\notin \bfS_\infty} O_{F_v}$ be the profinite completion of $O_F$ and  $\widehat{F}\coloneqq F\otimes_{O_F}\widehat{O}_F$ be the $\bfS_\infty$-finite adele ring of $F$.

For each $i\in\{1, 2\}$, let $B_i$ be a quaternion $F$-algebra and $O_i$ be a maximal $O_F$-order in $B_i$. By an $(O_1, O_2)$-bilattice, we mean a finitely generated $O_F$-torsion-free $(O_1, O_2)$-bimodule, or equivalently, a left $O_1\otimes_{O_F} O_2^{\opp}$-lattice. Clearly, the ambient space $L\otimes_{O_F} F$  of  an $(O_1, O_2)$-bilattice $L$ is a $(B_1, B_2)$-bimodule, or equivalently, a left $B_1\otimes_{F} B_2^{\opp}$-module. Henceforth we assume that $B_1$ and $B_2$ are \emph{non-isomorphic} unless specified otherwise. Let $B_3$ be the quaternion $F$-algebra which represents the product of $B_1$ and $B_2$ in the Brauer group $\Br(F)$. Then $B_3$ is  division by the non-isomorphism assumption, and we have
\begin{equation}\label{eq:defn-B-3}
     B_1\otimes_{F} B_2^{\opp}\simeq \Mat_2(B_3).
\end{equation}
A left $\Mat_2(B_3)$-module $V$ is of the form $\begin{bmatrix} B_3 \\ B_3 \end{bmatrix}^{\oplus m}$ for some $m\geq 0$. By an easy dimension consideration, we obtain the following lemma, which is a restatement of \cite[Proposition~2.1]{Ribet-bimod}; see also \cite[\S 3]{MR2931385}.

\begin{lemma}\label{lem:emb-exist-neces-cond}
Keep the assumption that $B_1$ and $B_2$ are two non-isomorphic quaternion $F$-algebras. Then any finitely generated $(B_1, B_2)$-bimodule $V$ is necessarily finite free as a left $B_1$-module and also as a right $B_2$-module, and $\rank_{B_1}V=\rank_{B_2}V$ is even. Conversely, for every non-negative even number $n\in 2\Z_{\geq 0}$, there is a unique $(B_1, B_2)$-bimodule of $B_2$-rank $n$ up to isomorphism.
\end{lemma}

\begin{remark}\label{rem:quater-bimod-rank}
      It is easy to see that Lemma~\ref{lem:emb-exist-neces-cond} remains valid if $F$ is assumed to be a local field, in which case one of the $B_i$'s is the unique quaternion division $F$-algebra, and  the other is the split matrix algebra $\Mat_2(F)$. 

  Note that if $B_1$ and $B_2$ are isomorphic quaternion algebras over a global field or local field (with the possibility of both of them being split), then any finitely generated $(B_1, B_2)$-bimodule $V$ is necessarily free as a left $B_1$-module and also as a right $B_2$-module, and $\rank_{B_1}V=\rank_{B_2}V$ (except that the rank is possibly odd in this case).
\end{remark}

We return to the global setting with $B_1\not\simeq B_2$.  Given an $(O_1, O_2)$-bilattice $L$ and a finite place $v$ of $F$, we write $L_v\coloneqq L\otimes_{O_F}O_{F_v}$ for the $v$-adic completion of $L$, which is naturally an $(O_{1, v}, O_{2, v})$-bilattice. 

\begin{defn}\label{defn:genus}
Two $(O_1, O_2)$-bilattices $L, L'$ are said to \emph{belong to the same genus} if they are locally isomorphic everywhere, that is, $L_v'\simeq L_v$ for every finite place $v$ of $F$.
\end{defn}

Let $\Sigma$ be the set of finite places of $F$ that are ramified in both $B_1$ and $B_2$.
For each $v\in\Sigma$, there exists an $O_{F_v}$-isomorphism  $f_v: O_{2, v}\xrightarrow{\simeq} O_{1, v}$ since both are maximal orders in quaternion division $F_{v}$-algebras. This allows us to identify $O_{1, v}$ with $O_{2, v}$ and 
view $L_v$ as an $(O_{2, v}, O_{2, v})$-bilattice via $f_v$ so that its structural invariant $\ulm(L_v)$ is defined as in Theorem~\ref{thm:conj-class}.  Taking the global point of view, we shall call $\ulm(L_v)$ \emph{the structural invariant} of the $(O_1, O_2)$-bilattice $L$ at $v$ with respect to $f_v$ and denote it by 
\begin{equation}\label{eq:defn-str-inv-at-v}
    \ulm^{(v, f_v)}(L)\coloneqq \ulm(L_v).
\end{equation}
If $\ulm^{(v, f_v)}(L)=(r_1, r_2, t_1, t_2)\in\Z_{\geq 0}^4$, then  an easy rank computation shows that  
\begin{equation}\label{eq:rank-comp}
    r_1+r_2+2t_1+2t_2=n,\quad\text{where}\quad n=\rank_{B_2}(L\otimes_{O_F}F).
\end{equation}

\begin{lemma}\label{lem:bilatt-genus}
Fix an identification $f_v: O_{2, v}\xrightarrow{\simeq} O_{1, v}$ for each $v\in \Sigma$.      Two $(O_1, O_2)$-bilattices $L$ and $L'$ belong to the same genus if and only if they have the same $O_F$-rank and $\ulm^{(v, f_v)}(L)=\ulm^{(v, f_v)}(L')$ for all $v\in\Sigma$.
\end{lemma}
 If $\Sigma\neq\emptyset$, then the same-rank assumption can be omitted by \eqref{eq:rank-comp}.
\begin{proof}
  From Theorem~\ref{thm:conj-class}, the equalities $\ulm^{(v, f_v)}(L)=\ulm^{(v, f_v)}(L')$ for all $v\in\Sigma$ guarantee that $L_v\simeq L_v'$ for all $v\in \Sigma$.  It remains to check that $L_v\simeq L_v'$ for every finite place  $v\notin\Sigma$  whenever 
 $L$ and $L'$ have the same $O_F$-rank.
 For such a $v$,  the tensor product $O_{1, v}\otimes_{O_{F_v}}O_{2, v}^{\opp}$ is a maximal order, so the desired isomorphism follows from \cite[Theorem~18.7]{reiner:mo}.
\end{proof}

\begin{remark}\label{rem:distinct-indentify}
 It should be emphasized that for $v\in \Sigma$,  the structural invariant $\ulm^{(v, f_v)}(L)$ depends on the choice of the identification $f_v$. 
 Let $f_v':O_{2, v}\xrightarrow{\simeq} O_{1, v}$ be another identification. Then there exists $\beta_v\in B_{2, v}^\times$ such that $f_v'(x_v)= f_v(\beta_v x_v \beta_v^{-1})$ for all $x_v\in O_{2, v}$.  If we write $\ulm^{(v, f_v)}(L)=(r_1, r_2, t_1, t_2)$ as before, then it follows from Remark~\ref{rem:str-inv-comp-conj} that 
 \[   \ulm^{(v, f_v')}(L)=\begin{dcases*}
           \ulm^{(v, f_v)}(L), & if $\ord_{B_{2,v}}(\beta_v)$ is even;\\
           (r_2, r_1, t_2, t_1), & if $\ord_{B_{2, v}}(\beta_v)$ is odd.
       \end{dcases*} \]
Another way to characterize the parity of $\ord_{B_{2,v}}(\beta_v)$ is as follows. For each $i\in\{1, 2\}$, let $\grP_{i, v}$ be the unique maximal two-sided ideal of $O_{i, v}$. Let $\grp_v$ be the unique maximal ideal of $O_{F_v}$, and $\kappa_v\coloneqq O_{F_v}/\grp_v$ be the finite residue field of $O_{F_v}$. Both  quotient fields $O_{i, v}/\grP_{i, v}$ are quadratic extensions of $\kappa_v$, and $\ord_{B_{2,v}}(\beta_v)$ is even if and only if $f_v$ and $f_v'$ induce the same $\kappa_v$-isomorphism $O_{2, v}/\grP_{2, v}\to O_{1, v}/\grP_{1, v}$.  In summary, the set of  isomorphisms $O_{2, v}\xrightarrow{\simeq} O_{1, v}$ are separated into two equivalence classes depending on the induced $\kappa_v$-isomorphism between the residue fields, and the structural invariant of a global $(O_1, O_2)$-bilattice $L$ at $v\in \Sigma$ is well-defined only  after one such $\kappa_v$-isomorphism $O_{2, v}/\grP_{2, v}\to O_{1, v}/\grP_{1, v}$ has been specified.  In fact, Ribet takes the latter approach  in \cite[\S2]{Ribet-bimod}, where he fixes a quadratic extension $\widetilde{\kappa}_v/\kappa_v$ and defines \emph{an orientation} of $O_i$ at $v\in \Sigma$ relative to $\widetilde{\kappa}_v$ to be a (surjective) homomorphism $O_i\twoheadrightarrow \widetilde{\kappa}_v$. The global study of  Ribet bimodules (op.~cit., p.~12) is then set off by fixing the orientations of $O_1$ and $O_2$ at all finite places $v\in \Sigma$ with respect to prior chosen $\widetilde\kappa_v$'s. As a convention, we often suppress $f_v$ from the notation $\ulm^{(v, f_v)}(L)$  and abbreviate it as $\ulm^{(v)}(L)$ if the identification $f_v:O_{2, v}\xrightarrow{\simeq} O_{1, v}$ is clear from the context. 
\end{remark}

For the remainder of this subsection, we fix an $(O_1, O_2)$-bilattice $L$, and study the isomorphism classes of $(O_1, O_2)$-bilattices in the genus of  $L$. Let $V\coloneqq L\otimes_{O_F}F$ and $n\coloneqq\rank_{B_2} V$, which is necessarily even by Lemma~\ref{lem:emb-exist-neces-cond}. 
Put
\[A\coloneqq\End_{B_1\otimes_F B_2^{\opp}}(V),\quad \text{and} \quad  \Lambda\coloneqq\End_{O_1\otimes_{O_F}O_2^{\opp}}(L).        \]
Then $\Lambda$ is an $O_F$-order in the finite dimensional central simple $F$-algebra $A$. Indeed, we find from~\eqref{eq:defn-B-3} that $V\simeq\begin{bmatrix} B_3 \\ B_3 \end{bmatrix}^{\oplus \frac{n}{2}}$ as a left $\Mat_2(B_3)$-module.
It follows immediately that (see also \cite[Proposition~2.1]{Ribet-bimod})
\begin{equation}\label{eq:compute-C}
    A\simeq \Mat_{n/2}(B_3).
\end{equation} 
Let $\widehat{\Lambda}\coloneqq \Lambda\otimes_{O_F}\widehat{O}_F$ be the profinite completion of $\Lambda$ and $\widehat{A}\coloneqq A\otimes_{F}\widehat{F}$ be the $\bfS_\infty$-finite adele ring of $A$.
By \cite[Theorem~31.18]{curtis-reiner:1}, there is a bijection between the set of isomorphism classes of $(O_1, O_2)$-bilattices in the genus of $L$ and the double coset space $A^\times\backslash \widehat{A}^\times/\widehat{\Lambda}^\times$.  
Let $\Omega\subseteq \bfS_\infty$ be the set of real places of $F$ ramified in $B_3$. In other words, $\Omega$ consists of the real places of $F$ ramified exactly in one of $B_1$ and $B_2$, but not the other. If $n\geq 4$, then $A$ clearly satisfies the Eichler condition~\cite[Definition 34.3]{reiner:mo}. It follows from \cite[Lemma~14]{Yu-CF:Bull-AS} that the reduced norm map
\begin{equation}\label{eq:3.1-nrd}
      \Nrd: A^\times\backslash \widehat{A}^\times/\widehat{\Lambda}^\times \to
           F_{>_{\Omega} 0}^\times\backslash\widehat{F}^\times / \Nrd(\widehat{\Lambda}^\times)
\end{equation} 
is a bijection, where $F_{>_{\Omega} 0}^\times$ denotes the following subgroup of  $F^\times$: 
\begin{equation*}\label{eq:Omega-positive-elements}
    F_{>_{\Omega} 0}^\times\coloneqq \{x\in F^\times \mid w(x)>0\text{ for every } w\in\Omega\}.   
\end{equation*}
    
\begin{prop}\label{prop:isom-cls-latt}
 We have $\Nrd(\widehat{\Lambda}^\times)=\widehat{O}_F^\times$ for all $n\in 2\Z_{>0}$. In particular, if $n\geq 4$, the number of the isomorphism classes of $(O_1, O_2)$-bilattices in the genus of $L$ is equal to the order $h^\Omega(F)$ of the ray class group $F_{>_{\Omega} 0}^\times\backslash \widehat{F}^\times /\widehat{O}_F^\times$.
\end{prop}
\begin{proof}
     We check that $\Nrd(\Lambda_v^\times)=O_{F_v}^\times$ for every finite place $v$ of $F$.

    First, suppose that $v$ is split in either $B_1$ or $B_2$ (or both). Then at least one of the maximal $O_{F_v}$-orders $O_{1, v},O_{2, v}$ is isomorphic to $\Mat_2(O_{F_v})$, which implies that $O_{1, v}\otimes_{O_{F_v}} O_{2, v}^{\opp}\simeq \Mat_2(\grO_v)$ for some maximal order $\grO_v$ in a quaternion algebra over $F_v$. Thus $L_v\simeq\begin{bmatrix}
        \grO_v \\ \grO_v
    \end{bmatrix}^{\oplus \frac{n}{2}}$ as left $\Mat_2(\grO_v)$-modules, and hence $\Lambda_v\simeq\Mat_{n/2}(\grO_v)$ is a maximal order. It follows from \cite[Proposition~45.8]{curtis-reiner:2} that $\Nrd(\Lambda_v^\times)=O_{F_v}^\times$.

    Next suppose that $v$ is ramified in both $B_1$ and $B_2$.  For ease of notation we identify both  $B_{1, v}$ and $B_{2, v}$ with the unique quaternion division $F_v$-algebra  $B_v$,  and in turn identify both $O_{1, v}$ and $O_{2, v}$ with  the unique maximal order $\calO_v$ in $B_v$. Then we have $\Lambda_v=\End_{(\calO_v, \calO_v)}(L_v)$.
    Let $K_v$ be the unique unramified quadratic extension of $F_v$ with ring of integers $O_{K_v}$. As in Section~\ref{sec:1}, we choose a uniformizer $\Pi_v$ of $B_v$ to present $\calO_v$ as $O_{K_v}+O_{K_v}\Pi_v$ with multiplication rules as in \eqref{eq:quat-multi-rul}. Let 
    \[
    \O_v=(\calO_v, \id),\quad \P_v=(\calO_v, \tau),\quad \L_{1, v}=(\calO_v^{\oplus 2}, \varphi_1),\quad \L_{2, v}=(\calO_v^{\oplus 2}, \varphi_2) 
    \]
    be the four indecomposable $(\calO_v, \calO_v)$-bilattices introduced in Section~\ref{sec:1}; see \eqref{eq:defn-four-indecomp-latt}.
    From Theorem~\ref{thm:conj-class}, $L_v$ is isomorphic to a direct sum of copies of $\O_v$, $\P_v$, $\L_{1, v}$ and $\L_{2, v}$, so it is enough to check the equality $\Nrd(\Lambda_v^\times)=O_{F_v}^\times$ when $L_v$ is equal to $\O_v$, $\P_v$, $\L_{1, v}$ and $\L_{2, v}$ respectively. If $L_v=\O_{v}$ or $\P_v$, then $\Lambda_v\simeq O_{F_v}$, whence the desired equality. 
    Now suppose that $L_v=\L_{1, v}$ or $\L_{2, v}$.  Note that the images of both $\varphi_1$ and $\varphi_2$ lie in $\Mat_2(O_{K_v})$ (see \eqref{eq:defn-varphi-1} and \eqref{eq:defn-varphi-2}), so we have an embedding of $O_{F_v}$-algebras
    \[       O_{K_v}\hookrightarrow \Lambda_v=\End_{(\calO_v, \calO_v)}(L_v): a\mapsto
             \left(\begin{bmatrix}\alpha \\ \beta \end{bmatrix}\mapsto 
             \begin{bmatrix} a\alpha \\ a\beta  \end{bmatrix}       \right).
    \]
    This induces an embedding $K_v \hookrightarrow \Lambda_v\otimes F_v=A_v\simeq\Mat_2(F_v)$ of $F_v$-algebras; see \eqref{eq:compute-C}. Therefore, $O_{F_v}^\times\supseteq\Nrd(\Lambda_v^\times)\supseteq\Nm_{K_v/F_v}(O_{K_v}^\times)=O_{F_v}^\times$ as desired since $K_v/F_v$ is an unramified extension. 
\end{proof}

\begin{remark}
    A global $(O_1, O_2)$-bilattice $L$ is called \emph{admissible} if $L_{v}$ is admissible
    for every finite place $v\in \Sigma$ (see Theorem~\ref{thm:Ribet}). If $n=2$, Ribet \cite[Theorem~2.4]{Ribet-bimod} classifies the isomorphism classes of admissible $(O_1, O_2)$-bilattices in terms of certain oriented Eichler orders in the quaternion algebra $B_3$. This is generalized by Molina \cite[Theorem~3.7]{MR2931385} to the case where $O_1$ is an Eichler order and $O_2$ is a maximal order.
\end{remark}

\subsection{Superspecial $\calO$-abelian varieties} 
\label{subsec:ssp-O-1-ab-var}
We give a geometry application of global quaternion bilattices studied in the last subsection. Let $F$ be a number field, and $O_F$ be its ring of integers  (with $\bfS_\infty$ being the set of non-archimedean places of $F$). Let $B$ be a quaternion $F$-algebra, and $\calO$ be a maximal $O_F$-order in $B$. 

Fix a supersingular elliptic curve $E$ over an algebraically closed field $k$ of characteristic $p>0$. Its endomorphism algebra $D\coloneqq\End^{0}(E)=\End(E)\otimes\Q$ is the (unique) quaternion algebra over $\Q$ ramified exactly at $\{p, \infty\}$,  and its endomorphism ring $\calR\coloneqq\End(E)$ is a maximal order in $D$.
Henceforth, we always assume that $B$ is not isomorphic to $D_F\coloneqq D\otimes_{\Q}F$.
\begin{defn}\label{defn:O-1-ab-var}
    An $\calO$-\ab over $k$ is an ordered pair $(X, \iota)$, where $X$ is an \ab over $k$, and $\iota:\calO\hookrightarrow\End(X)$ is an embedding of rings.
\end{defn}
An $\calO$-\ab $(X, \iota)$ over $k$ is said to be superspecial if $X$ itself is superspecial.
Clearly,  the dimension $n$ of a superspecial $\calO$-\ab $(X, \iota)$ over $k$ is at least two since $B$ does not embed into $D$ by our assumption. Then by a theorem of Deligne, Ogus and Shioda \cite[\S 1.6, p.13]{li-oort}, $X$ is isomorphic to $E^n$.  For each superspecial $\calO$-\ab $(X, \iota)$ over $k$, we attach to it its Ribet bilattice 
\begin{equation}\label{eq:defn-ab-var-assoc-bi-latt}
    L\coloneqq\Hom(E, X),
\end{equation}
 which is naturally an $(\calO, \calR)$-bilattice, or equivalently, a left $\calO\otimes_{\Z}\calR^{\opp}$-lattice.  More explicitly, if we fix an isomorphism $X\simeq E^n$ and identify $\End(X)$ with $\Mat_n(\calR)$, then $L$ is identified with the  free right $\calR$-module $\calR^{\oplus n}$ of column vectors, with $\calO$ acting from the left via the embedding $\iota: \calO\hookrightarrow \Mat_n(\calR)$. As usual, we denote the $(\calO, \calR)$-bilattice thus obtained by $(\calR^{\oplus n}, \iota)$.  On the other hand, for any arbitrary $(\calO, \calR)$-bilattice $L'$, its ambient space $L'\otimes_{\Z}\Q$  is a $(B, D)$-bimodule, or equivalently, a left $B\otimes_{\Q}D^{\opp}$-module. This also places a restriction on the $D$-ranks of   $(B, D)$-bimodules as follows.

\begin{lemma}\label{lem:B-D-bimod}
Let  $d\coloneqq[F:\Q]$ be the degree of $F$. 
The following statements are equivalent for a positive integer $n$:
\begin{enumerate}[label=(\roman*)]
   \item there exists a $(B, D)$-bimodule $V$ of $D$-rank $n$;
    \item $(2d)\mid n$;
    \item there exists an embedding $\iota: \calO\hookrightarrow \Mat_n(\calR)$.
\end{enumerate}
  Moreover, suppose that the above equivalence conditions hold for $n$. Then there is a unique $(B, D)$-bimodule $V$ with $\rank_D(V)=n$  up to isomorphism, and every $(\calO, \calR)$-bilattice $L$ with $\rank_D(L\otimes_\Z \Q)=n$ is isomorphic to $(\calR^{\oplus n}, \iota)$ for a suitable embedding $\iota: \calO\hookrightarrow \Mat_n(\calR)$. 
\end{lemma}
\begin{proof}

 The equivalence between (i) and (ii)  follows directly from Lemma~\ref{lem:emb-exist-neces-cond}. Indeed, we have $B\otimes_{\Q}D^{\opp}=B\otimes_F D_F^\opp$ with $D_F\coloneqq D\otimes_\Q F$ as before, so every $(B, D)$-bimodule is canonically a $(B, D_F)$-bimodule, and vice versa. Moreover,  since $B\otimes_F D_F^\opp$ is $F$-central simple, there is at most one left $B\otimes_{\Q}D^{\opp}$-module of given $D$-rank up to isomorphism. 

 Clearly, (iii) implies (i) since the ambient space of the $(\calO, \calR)$-bilattice $(\calR^{\oplus n}, \iota)$ is a $(B, D)$-bimodule of $D$-rank $n$.
 Conversely,  a $(B, D)$-bimodule $V$ of $D$-rank $n$ necessarily contains an $(\calO, \calR)$-bilattice $L$ of full rank. 
 Now  a theorem of Eichler \cite[Theorem~34.9]{reiner:mo} shows that the class number of $\Mat_n(\calR)$ is $1$ whenever $n\geq 2$. It follows that every right $\calR$-lattice $L$ with $\rank_{D}(L\otimes_{\Z}\Q)\geq 2$ is necessarily free as a right $\calR$-module.  If we identify $L$ with 
 $\calR^{\oplus n}$, then the left $\calO$-module structure on $L$ is given by an embedding $\iota: \calO\hookrightarrow \Mat_n(\calR)$. This proves the equivalence between (i) and (iii), and also verifies that 
 every $(\calO, \calR)$-bilattice is isomorphic to an $(\calO, \calR)$-bilattice of the form $(\calR^{\oplus n}, \iota)$.
\end{proof}

The following corollary of Lemma~\ref{lem:B-D-bimod} is a global counterpart of 
Lemma~\ref{lem:conj-emb-bimodule}.
\begin{cor}\label{cor:glo-emb-bilatt}
    For any positive integer $n$ divisible by $2d$, the assignment $\iota\mapsto (\calR^{\oplus n}, \iota)$ induces a bijection between the following two sets:
  \[\left\{\parbox{3.8cm}{$\GL_n(\calR)$-conjugacy classes of embeddings $\calO\hookrightarrow
\Mat_n(\calR)$}\right\}\quad \xlongleftrightarrow{1-1}\quad \left\{\parbox{3.7cm}{Isomorphism
classes of $(\calO, \calR)$-bilattices of $\calR$-rank $n$}\right\} \]
\end{cor}

Lemma~\ref{lem:B-D-bimod} and Corollary~\ref{cor:glo-emb-bilatt} translate directly into the following result on superspecial $\calO$-abelian varieties over the algebraically closed field $k$.

\begin{lemma}\label{lem:emb-exist-neces-cond-2}
    (1)   Let $X$ be a superspecial abelian variety over $k$ of dimension $n$. Then there exists an embedding $\iota:\calO\hookrightarrow\End(X)$ if and only if $(2d)|n$.\\
     (2)  Fix a positive integer $n$ divisible by $2d$. 
    The assignment $(X, \iota)\mapsto L\coloneqq\Hom(E, X)$ induces a bijection from the set of isomorphism classes of superspecial $\calO$-\abs over $k$ of dimension $n$ to the set of isomorphism classes of $(\calO, \calR)$-bilattices of $\calR$-rank $n$.
\end{lemma}

In Section~\ref{subsec:glob-latt}, we have studied the classification of $(O_1, O_2)$-bilattices with both $O_i$ being maximal quaternion $O_F$-orders.  The current setting is slightly different since $\calO$ and $\calR$ are quaternion orders over different Dedekind domains (namely, $O_F$ and $\Z$ respectively).  This can be easily remedied by regarding every $(\calO, \calR)$-bilattice as an $(\calO, \calR_F)$-bilattice with $\calR_F\coloneqq \calR\otimes_\Z O_F$.  However, it may well happen that $\calR_F$ ceases to be a \emph{maximal} order after the base change. For example, if $v$ is an unramified place of $F$ above $p$ with uniformizer $\pi_v\in O_{F_v}$ and \emph{even} residue degree, then $\calR\otimes_\Z O_{F_v}$ is 
an Eichler order of level $\pi_v O_{F_v}$ in the split quaternion $F_v$-algebra $D\otimes_\Q F_v\simeq \Mat_2(F_v)$ by \cite[\S2.4]{li-xue-yu:unit-gp}.  Therefore, to proceed further, we need to make additional assumptions on $F$ to ensure that $\calR_F$ is maximal. 

\begin{lemma}\label{lem:max-order-under-base-change}
    The quaternion $O_F$-order $\calR_F= \calR\otimes_{\Z}O_F$ is maximal if and only if every place $v$ of $F$ above $p$ is unramified and has odd  residue degree. In particular, if $p$ splits completely in $F$, then $\calR_F$ is maximal.
\end{lemma}
\begin{proof}
    By \cite[Theorem~15.5.5]{voight-quat-book},
    the order $\calR_F$ is maximal if and only if its reduced discriminant $\disc(\calR_F)$ coincides with the reduced discriminant $\disc(D_F)$ of the quaternion $F$-algebra $D_F=D\otimes_\Q F$, which is equal to the product of all prime ideals of $O_F$ that are ramified in $D_F$. On the other hand, we have $\disc(\calR_F)=\disc(\calR)O_F=pO_F$. It follows that $\calR$ is maximal if and only if both of the following conditions hold for every finite place $v$ of $F$ above $p$:
    \begin{enumerate}[label=(\roman*)]
        \item $v$ is unramified; and 
        \item $D\otimes_\Q F_v$ is division.
    \end{enumerate}
 In light of condition (i), condition (ii) is equivalent to the residue degree of $v$ being odd.    
\end{proof}

For the rest of this section, we assume that  $\calR_F$ is maximal.  The set  of finite places of $F$ ramified in both $B$ and $D_F$ is precisely the set $\Sigma$  of places of $F$ above $p$ ramified in $B$.  For each  $v\in \Sigma$, we fix an isomorphism $ \calR\otimes_\Z O_{F_v}\simeq \calO_v$ so that the structural invariant $\ulm^{(v)}(L)$ at $v$ of every $(\calO, \calR_F)$-bilattice $L$ is defined; see Remark~\ref{rem:distinct-indentify}. Given a superspecial $\calO$-abelian $k$-variety $(X, \iota)$, we define its structural invariant 
$\ulm^{(v)}(X, \iota)$ at $v$ to be that of the corresponding    
$(\calO, \calR_F)$-bilattice $L\coloneqq\Hom(E, X)$, that is, 
\begin{equation}\label{eq:defn-str-inv-of-ab-var}
    \ulm^{(v)}(X, \iota)\coloneqq\ulm^{(v)}(L)=\ulm(L_{v}), \qquad \forall v\in \Sigma. 
\end{equation}
If $\ulm^{(v)}(X, \iota)=(r_1^{(v)}, r_2^{(v)}, t_1^{(v)}, t_2^{(v)})\in\Z_{\geq 0}^4$ and $\dim X=n$, then by \eqref{eq:rank-comp} we have 
\begin{equation}\label{eq:rank-comp-1} 
r_1^{(v)}+r_2^{(v)}+2t_1^{(v)}+2t_2^{(v)}=n/d.  
\end{equation}

\begin{eg}\label{eg:F}
Keep  the  notation and assumption as above.  Let $\Omega$ be the set of real places of $F$ split in $B$, and $h^{\Omega}(F)$ be the order of the ray class group $F_{>_{\Omega} 0}^\times\backslash \widehat{F}^\times /\widehat{O}_F^\times$.
    Let $n$ be a positive integer divisible by $2d$.
    For each $v\in \Sigma$, fix  a quadruple 
    \[  (r_1^{(v)}, r_2^{(v)}, t_1^{(v)}, t_2^{(v)})\in\Z_{\geq 0}^4
    \]
    satisfying \eqref{eq:rank-comp-1}. If $n>2d$, then from Lemma~\ref{lem:bilatt-genus} and Proposition~\ref{prop:isom-cls-latt}, there are exactly $h^{\Omega}(F)$  isomorphism classes of $(\calO, \calR)$-bilattices $L$ of $\calR$-rank $n$ (resp.~superspecial $\calO$-abelian $k$-varieties $(X, \iota)$ of dimension $n$) with $\ulm^{(v)}(L)= (r_1^{(v)}, r_2^{(v)}, t_1^{(v)}, t_2^{(v)})$ (resp.~$\ulm^{(v)}(X, \iota)= (r_1^{(v)}, r_2^{(v)}, t_1^{(v)}, t_2^{(v)})$) for all $v\in \Sigma$. In particular, suppose that $F=\Q$ and $p$ is ramified in $B$. Then for each fixed even integer $n\geq 4$ and a quadruple 
$(r_1, r_2, t_1, t_2)\in\Z_{\geq 0}^4$  with $r_1+r_2+2t_1+2t_2=n$,  there is exactly one isomorphism class of superspecial $\calO$-abelian $k$-varieties $(X, \iota)$ with $\ulm^{(p)}(X, \iota)=(r_1, r_2, t_1, t_2)$.
\end{eg}

\section{Self-dual hermitian local quaternion bilattices}\label{sec:self-dual-local-latt}
Let $R$ be a noetherian integral domain with fraction field $F$. Let $B_1$ be a semisimple $F$-algebra and $O_1$ be a maximal $R$-order in $B_1$. Let $B_2$ be a quaternion $F$-algebra and $O_2$ be a maximal $R$-order in $B_2$.  For each $b\in B_2$, we write $\bar{b}$ for the canonical involution of $b$.  Let $*$ be an $F$-linear involution on $B_1$ that stabilizes $O_1$.

By definition, a hermitian form (or pairing) on a finite free right $B_2$-module $V$ is a non-degenerate $F$-bilinear form $\langle~,~\rangle: V\times V\to B_2$ such that
\begin{equation}\label{eq:herm-form-defn}
        \langle x, y\rangle=\overline{\langle y, x\rangle}, \quad \text{and}\qquad
    \langle x\alpha, y\beta\rangle=\overline \alpha\langle x, y\rangle \beta,\quad\forall x,y\in V,\ \forall \alpha,\beta\in B_2.
\end{equation}

\begin{defn}\label{defn:herm-latt}
 (i)   A \emph{hermitian $(B_1, *, B_2)$-bimodule} is an ordered pair $(V, \langle~,~\rangle)$, where $V$ is a $(B_1, B_2)$-bimodule finite free over $B_2$, and $\langle~,~\rangle: V\times V\to B_2$ is a hermitian form on the free right $B_2$-module $V$ such that
    \begin{equation}\label{eq:herm-bimod-eq}
        \langle \alpha x, y\rangle=\langle x, \alpha^* y\rangle,\quad\forall x, y\in V,\text{ }\forall \alpha\in B_1.
    \end{equation}
(ii)    Similarly, a \emph{hermitian $(O_1, *, O_2)$-bilattice} $(L, \langle~,~\rangle)$ is an $(O_1, O_2)$-bilattice $L$ equipped with an $R$-bilinear pairing $\langle~,~\rangle: L\times L\to B_2$ such that the $(B_1, B_2)$-bimodule $L\otimes_{R} F$ together with the canonical $F$-linear extension of $\langle~,~\rangle$  forms a hermitian $(B_1, *, B_2)$-bimodule.\\
(iii)    An \emph{isometry} between two hermitian $(B_1, *, B_2)$-bimodules (resp.~$(O_1, *, O_2)$-bilattices) is an isomorphism between the underlying $(B_1, B_2)$-bimodules (resp.~$(O_1, O_2)$-bilattices) that preserves the hermitian pairings. 
\end{defn}

Let $(V, \langle~,~\rangle)$ be a hermitian $(B_1, *, B_2)$-bimodule. The $B_2$-dual space  $V^{\sharp}\coloneqq\Hom_{B_2}(V, B_2)$ has a canonical $(B_2, B_1)$-bimodule structure induced by the left $B_1$-module structure on $V$ and the left $B_2$-module structure on $B_2$. Using the involutions on $B_1$ and $B_2$, we endow $V^{\sharp}$ with a $(B_1, B_2)$-bimodule structure as follows.
For each $f\in V^{\sharp}$ and $\alpha\in B_1$, $\beta\in B_2$, we define $\alpha f, f\beta\in V^{\sharp}$ respectively by
\begin{equation}\label{eq:L-prime-bim-str}
    (\alpha f)(x)=f(\alpha^* x),\quad \text{and} \quad(f\beta)(x)=\overline{\beta}f(x),\qquad \forall x\in V.
\end{equation}
If $L$ is an $(O_1, O_2)$-bilattice in $V$, then $L^{\sharp}\coloneqq\Hom_{O_2}(L, O_2)$ is an $(O_1,O_2)$-bilattice in $V^{\sharp}$. 
Since $\langle~,~\rangle$ is nondegenerate,  the following map induced by the pairing $\langle~,~\rangle$ defines an isomorphism of $(B_1, B_2)$-bimodules:
\begin{equation}\label{dual-map-rat}
    V\xrightarrow{\simeq} V^{\sharp}=\Hom_{B_2}(V, B_2),\qquad x\mapsto (y\mapsto \langle x, y\rangle).
\end{equation}
Given an $(O_1, O_2)$-bilattice $L$ in $(V, \langle~,~\rangle)$, we define an $(O_1, O_2)$-bilattice $L^{\vee}$ in $V$ by
\begin{equation}\label{eq:defn-of-dual-bilatt}
    L^{\vee}\coloneqq\{x\in V\mid \langle x, L\rangle\subseteq O_2\}.
\end{equation}
Then the isomorphism \eqref{dual-map-rat} restricts to an isomorphism of $(O_1, O_2)$-bilattices as follows: 
\begin{equation}\label{dual-map}
    L^{\vee}\xrightarrow{\simeq}  L^{\sharp}=\Hom_{O_2}(L, O_2),\qquad  x\mapsto (y\mapsto \langle x, y\rangle).  
\end{equation}
 For this reason, we call $L^{\vee}$ the \emph{dual bilattice} of $L$ in $(V, \langle~,~\rangle)$. If $L=L^{\vee}$, we say that $(L, \langle~,~\rangle)$ is \emph{self-dual}, or that the hermitian pairing $\langle~,~\rangle$ is \emph{perfect} on $L$. 
In more down-to-earth language, a hermitian $(O_1, *, O_2)$-bilattice
$(L, \langle~,~\rangle)$ with $L$ free of $O_2$-rank $n$ is self-dual if and only if the Gram matrix of $\langle~,~\rangle$ under a (hence every) right $O_2$-basis of $L$ lies in $\GL_n(O_2)$. Finally, we remark that a self-dual hermitian $(O_1, *, O_2)$-bilattice $(L, \langle~,~\rangle)$ is necessarily $O_2$-valued, that is, $\langle x, y \rangle\in O_2$ for all $x, y\in L$.
 
Now suppose for the moment that $F$ is a global field with ring of integers $R=O_F$ and $B_1$ and $B_2$ are two quaternion $F$-algebras as in Section~\ref{subsec:glob-latt}. Let $(V,\langle~,~\rangle)$ be a hermitian $(B_1, *, B_2)$-bimodule with $*$ an orthogonal involution on $B_1$ \cite[Definition~2.5]{book-of-involution}. As usual,  to classify the isomorphism classes of self-dual $(O_1, O_2)$-bilattices in $(V, \langle~,~\rangle)$, we  first classify their genera. In other words, we classify self-dual hermitian $(O_{1, v}, *, O_{2, v})$-bilattices at every finite place $v$ of $F$.  This would be the main task of the current section. 

For the rest of this section,  let $F$ be a nonarchimedean local field with ring of integers $R=O_F$ and residue field $\F_{q}$. For uniformity, we assume  that  $\fchar(F)\neq 2$ throughout, although this assumption is not needed in Section~\ref{subsec:self-dual-latt-odd}.  Keep the assumption that $*$ is an orthogonal involution on the quaternion $F$-algebra $B_1$ stabilizing $O_1$. Then by \cite[Proposition~2.21]{book-of-involution} there exists $\gamma\in B_1^\times$ such that
\begin{equation}\label{eq:invol-star}
    \gamma+\ol\gamma=0,\quad \text{and} \quad \alpha^*=\gamma\ol{\alpha}\gamma^{-1},\quad\forall\alpha\in B_1.
\end{equation}
Clearly,  $\gamma$ is only determined by $*$ up to multiplication by an element of $F^\times$. Nevertheless, if $B_1$ is division,  the parity of the discrete valuation $\ord_{B_1}(\gamma)$ is uniquely determined by $*$. 
From Remark~\ref{rem:quater-bimod-rank}, every $(O_1, O_2)$-bilattice is  free as a left $O_1$-module and also as a right $O_2$-module.  We separate the local  classification of self-dual hermitian $(O_1, *, O_2)$-bilattices into the following cases: 
\begin{enumerate}[(i)]
    \item both $B_1$ and $B_2$ are division and $\ord_{B_1}(\gamma)$ is odd;
    \item both $B_1$ and $B_2$ are division and $\ord_{B_1}(\gamma)$ is even;
    \item $B_1$ is split, $B_2$ is division;
    \item $B_1$ is division, $B_2$ is split;
    \item both $B_1$ and $B_2$ are split, and the $O_2$-rank of the underlying $(O_1, O_2)$-bilattice  is even.
\end{enumerate}
One reason for the parity assumption on the $O_2$-rank in case (v) stems from global considerations as in Lemma~\ref{lem:emb-exist-neces-cond-2}. It turns out that this condition is also necessary for the non-degeneracy of the hermitian form, in essentially the same way that the dimension of a (non-degenerate) symplectic space is necessarily even; see Section~\ref{subsec:self-dual-latt-split} below. 
We treat cases (i), (ii) and (iii) in Section~\ref{subsec:self-dual-latt-odd}, \ref{subsec:self-dual-latt-even} and \ref{subsec:self-dual-latt-B-1-split} respectively, and cases (iv) and (v) are treated in Section~\ref{subsec:self-dual-latt-split}.

As remarked above, a self-dual hermitian $(O_1, *, O_2)$-bilattice is necessarily $O_2$-valued, so henceforth we  focus exclusively on  $O_2$-valued hermitian $(O_1, *, O_2)$-bilattices.

\subsection{Self-dual hermitian $(O_B, *, O_B)$-bilattices with $\ord_B(\gamma)$ odd}
\label{subsec:self-dual-latt-odd}
In cases (i) and (ii), both $B_1$ and $B_2$ are quaternion division algebras over the nonarchimedean local field $F$, and $O_1$ and $O_2$ are maximal $O_F$-orders in $B_1$ and $B_2$ respectively. For ease of notation, we identify $B_1$ with $B_2$ via a fixed $F$-isomorphism and put $B\coloneqq B_1=B_2$; the maximal orders are identified accordingly, and we put $O_B\coloneqq O_1=O_2$. 
 Let $K$ be the unique unramified quadratic extension of $F$ with ring of integers $O_K$. As in Section~\ref{sec:1}, we choose a suitable uniformizer $\Pi$ of $B$ to present $O_B$ as $O_K\oplus O_K\Pi$ with multiplication rules \eqref{eq:quat-multi-rul}, namely, 
\begin{equation}\label{eq:quat-multi-rul-1}
\Pi^2=\pi, \qquad \Pi a= a^\sigma \Pi, \quad \forall a\in O_K, 
\end{equation}
where $\pi$ is a fixed uniformizer of $F$, and $\sigma$ is the unique non-trivial automorphism of $K/F$.

We treat case (i) in the current section and leave case (ii) to Section~\ref{subsec:self-dual-latt-even}, so assume that $\ord_B(\gamma)$ is odd. In view of $\Pi^2=\pi$ and \eqref{eq:invol-star}, we may assume that $\ord_B(\gamma)=1$, that is, $\gamma=u\Pi$ for some $u\in O_B^\times$. 
We define another involution $'$ on $B$ by sending each $\alpha\in B$ to $\alpha'\coloneqq\Pi\overline{\alpha}\Pi^{-1}$. 
The two involutions $*$ and   $'$ are related by the equation  $\alpha^{*}=u\alpha'u^{-1}$ for all $\alpha\in B$. The condition $\gamma+\ol{\gamma}=0$ implies that  $u'=\Pi\ol{u}\Pi^{-1}=u$, and hence $u^*=uu'u^{-1}=u$. Given a hermitian $(O_B, *, O_B)$-bilattice $(L, \langle~,~\rangle)$, we define a pairing $\langle~,~\rangle_u: L\times L\to O_B$ by the formula
\begin{equation}\label{eq:defn-associ-pairing}
    \langle x, y \rangle_u\coloneqq \langle x, uy \rangle,\quad \forall x,y\in L.
\end{equation}
Then $\langle~,~\rangle_u$ is a hermitian form on the right $O_B$-lattice $L$, and  \eqref{eq:herm-bimod-eq} is equivalent to
\begin{equation}\label{eq:herm-bimod-eq-1}
    \langle \alpha x, y\rangle_u=\langle x, \alpha' y \rangle_u,\quad\forall x, y\in L,\,\forall\alpha\in O_B.
\end{equation}
In other words, $(L, \langle~,~\rangle_u)$ forms a hermitian $(O_B,\,', O_B)$-bilattice. Since $u$ is a unit of $O_B$, $(L, \langle~,~\rangle)$ is self-dual if and only if $(L, \langle~,~\rangle_u)$ is so. Therefore, the classification of (self-dual) hermitian $(O_B, *, O_B)$-bilattices is reduced to that of (self-dual) hermitian $(O_B,\,', O_B)$-bilattices.

We first write down some examples of self-dual hermitian $(O_B,\,*, O_B)$-bilattices of rank two that will serve as building blocks of more general  self-dual bilattices. Recall from \eqref{eq:defn-four-indecomp-latt} the four basic $(O_B, O_B)$-bilattices $\O\coloneqq(O_B, \id)$, $\P\coloneqq(O_B, \tau)$, $\L_1\coloneqq (O_B^{\oplus 2}, \varphi_1)$ and  $\L_2\coloneqq (O_B^{\oplus 2}, \varphi_2)$. Let $\L_3$ be the $(O_B, O_B)$-bilattice
\begin{equation}\label{eq:defn-L-3}
    \L_3\coloneqq \bbO\oplus \bbP=(O_B^{\oplus 2}, \varphi_3),\qquad \text{with}\quad \varphi_3\coloneqq\id\oplus\tau.
\end{equation}
Since $K/F$ is an unramified quadratic extension, there exists  $c\in O_K^\times$ such that $c+c^\sigma=0$.
  For each $1\leq i\leq 3$, let $(\M_i, \langle~,~\rangle_i)$ be the hermitian $(O_B, *, O_B)$-bilattice such that the underlying $(O_B, O_B)$-bilattice $\M_i$ is equal to $\L_i\coloneqq(O_B^{\oplus 2}, \varphi_i)$ and the Gram matrix $C_i$ of the associated pairing $\langle~,~\rangle_{i, u}$ defined by \eqref{eq:defn-associ-pairing} with respect to the standard right $O_B$-basis of $\M_i$ is given as follows
\begin{equation}
    C_1\coloneqq\begin{bmatrix}
        0 & c \\ -c & 0
    \end{bmatrix}, \qquad C_2\coloneqq\begin{bmatrix}
                      0 & c \\ -c & 0
                    \end{bmatrix}, \qquad C_3\coloneqq\begin{bmatrix}
                                            0 & 1 \\ 1 & 0
                                         \end{bmatrix}.
\end{equation}
 It is straightforward to check that each $\langle~,~\rangle_{i, u}$ satisfies \eqref{eq:herm-bimod-eq-1}.
Since these matrices all lie in $\GL_2(O_B)$, every $(\M_i, \langle~,~\rangle_i)$ is self-dual.

\begin{remark}\label{rem:indep-c}
    We point out that for each $i\in\{1, 2\}$, the isometry class of $(\M_i, \langle~,~\rangle_{i, u})$ (hence that of $(\M_i, \langle~,~\rangle_i)$) is independent of the choice of the element $c\in O_K^\times$ satisfying $c+c^\sigma=0$. Let $d$ be another element of $O_K^\times$ with 
    $d+d^\sigma=0$, and $(~,~)_{i, u}$ be the hermitian pairing on $\M_i=\L_i$ whose  Gram matrix  is given by $\begin{bmatrix}
        0 & d \\ -d & 0
    \end{bmatrix}$.  Since $K/F$ is  unramified, the norm map $\Nm_{K/F}: O_K^\times\to O_F^\times$ is surjective. In particular,  there exists $b\in O_K^\times$ such that $b^\sigma b=c/d \in O_F^\times$. 
 It is straightforward to check that the map
    \[
        \M_i\to \M_i,\quad \begin{bmatrix}
            x \\
            y
        \end{bmatrix}\mapsto \begin{bmatrix}
                                   bx  \\
                                   by
                                \end{bmatrix}
    \]
     defines an isometry from  $(\M_i, \langle~,~\rangle_{i, u})$ to $(\M_i, (~,~)_{i, u})$ for each $i\in \{1, 2\}$.
\end{remark}

Henceforth for each $i\in \{1, 2, 3\}$,  we  keep the  pairing $\langle~,~\rangle_{i, u}$ on $\M_i$ fixed, and simply write $\M_i$ for the hermitian $(O_B,\,', O_B)$-bilattice $(\M_i, \langle~,~\rangle_{i, u})$. Let $\boxplus$ denote the orthogonal direct sum of hermitian bilattices.
The main theorem of Section~\ref{subsec:self-dual-latt-odd} is stated as follows. 

\begin{thm}\label{thm:self-dual-latt-odd}
Keep the assumptions that $B$ is a quaternion division $F$-algebra and $\ord_{B}(\gamma)$ is odd.
Let $L$ be an $(O_B, O_B)$-bilattice with structural invariant $(r_1, r_2, t_1, t_2)\in\Z_{\geq 0}^4$, that is, 
   \begin{equation}\label{eq:under-latt}
    L\simeq \O^{\oplus r_1}\oplus \P^{\oplus r_2}\oplus \L_1^{\oplus t_1}\oplus \L_2^{\oplus t_2}.
    \end{equation}
Then there exists a perfect hermitian $(O_B, *, O_B)$-pairing $\langle~,~\rangle$ on $L$ if and only if $r_1=r_2$. In this case, $(L, \langle~,~\rangle_u)$ is isometric to the orthogonal direct sum
    \begin{equation}\label{eq:orth-split-odd}
        \M_1^{\boxplus t_1}\boxplus \M_2^{\boxplus t_2}\boxplus \M_3^{\boxplus r}\text{ with } r\coloneqq r_1=r_2. 
    \end{equation}
In particular, up to isometry, there exists a unique self-dual hermitian $(O_B, *, O_B)$-bilattice with given structural invariant $(r, r, t_1, t_2)$.
\end{thm}
The proof of this theorem will occupy the remainder of Section~\ref{subsec:self-dual-latt-odd}. We begin with some simple observations. By restriction of scalars, every $(O_B, O_B)$-bilattice $L$ is an $(O_K, O_K)$-bilattice. Since $K/F$ is unramified, we have an isomorphism $O_K\otimes_{O_F}O_K\simeq O_K\oplus O_K$ mapping each $a\otimes b\in O_K\otimes_{O_F}O_K$ to $(ab, a^\sigma b)\in O_K\oplus O_K$. This induces a decomposition $L=L_{\id}\oplus L_\sigma$ of $L$ into a direct sum of two $(O_K, O_K)$-sub-bilattices $L_{\id}$ and $ L_\sigma$, where 
\begin{equation}\label{eq:L-id-L-sigma}
     L_{\id}\coloneqq\{x\in L\mid ax=xa, \forall a\in O_K\},\quad
        L_{\sigma}\coloneqq\{x\in L\mid ax=xa^\sigma, \forall a\in O_K\}.
\end{equation}
\begin{lemma}\label{lem:pairing-property-odd}
    Let $(L, \langle~,~\rangle)$ be a hermitian $(O_B, *, O_B)$-bilattice, and $L=L_{\id}\oplus L_\sigma$ be the decomposition as above.  
    Then the associated pairing $\langle~,~\rangle_u$  has the following properties.
    \begin{enumerate}[(1)]
        \item If either $x, y\in L_{\id}$ or $x, y\in L_{\sigma}$, then $\langle x, y\rangle_u\in O_K\Pi$; in particular, $\langle x, x \rangle_u=0$. Moreover, if further $\Pi x=x\Pi$ and $\Pi y=y\Pi$, then $\langle x, y \rangle_u\in O_F\Pi$.
        \item If $x\in L_{\id}$ and $y\in L_{\sigma}$, then $\langle x, y \rangle_u\in O_K$. Moreover, if further $\Pi x=x\Pi$ and $\Pi y=y\Pi$, then $\langle x, y \rangle_u\in O_F$.
        \item If $x\in L_{\id}$, $y\in L_{\sigma}$ and $\Pi x=y$ (or $\Pi y=x$), then $\langle x, y \rangle_u$ is an element of $ O_K$ satisfying $\langle x, y \rangle_u^\sigma+\langle x, y \rangle_u=0$.
    \end{enumerate}
\end{lemma}
\begin{proof}
(1) We only consider the case $x, y\in L_{\id}$, as the case $x, y\in L_{\sigma}$ can be treated in exactly the same way.
Since $a'=\Pi \bar{a}\Pi^{-1}=a$ for every $a\in O_K$, 
 we have $\langle ax, y\rangle_u=\langle x, ay\rangle_u$. On the other hand, the assumption $x, y\in L_{\id}$ implies that
    \[
        \langle ax, y\rangle_u=\langle xa, y\rangle_u=a^\sigma\langle x, y\rangle_u,\quad
        \langle x, ay\rangle_u=\langle x, ya\rangle_u=\langle x, y\rangle_u a, \qquad \forall a\in O_K. 
   \]
Hence $a^\sigma\langle x, y\rangle_u=\langle x, y\rangle_u a$ for every $a\in O_K$, which implies that $\langle x, y\rangle_u\in O_K\Pi$. 
In particular, $\langle x, x\rangle_u\in O_F\cap O_K\Pi=\{0\}$. Lastly, since $\Pi'=-\Pi$, we have $\langle \Pi x, y\rangle_u=-\langle x, \Pi y\rangle_u$. 
Now suppose additionally that $\Pi x=x\Pi$ and $\Pi y=y\Pi$. Then $\langle x, y\rangle_u$ commutes with $\Pi$ as shown by the calculation below: 
\[
    \langle \Pi x, y\rangle_u=\langle x \Pi , y\rangle_u=-\Pi \langle x, y \rangle_u,\quad
    \langle x, \Pi y\rangle_u=\langle x, y\Pi \rangle_u=\langle x, y \rangle_u\Pi.
\]
We  conclude that $\langle x, y \rangle_u\in O_F\Pi$ under this additional assumption. 

(2) Suppose that $x\in L_{\id}$ and $y\in L_{\sigma}$. Similarly to the preceding case, we have
\[
    a^\sigma\langle x, y\rangle_u= \langle ax, y\rangle_u= \langle x, ay\rangle_u=\langle x, y\rangle_u a^\sigma,\qquad \forall a\in O_K, 
\]
which shows that $\langle x, y\rangle_u\in O_K$. If further $\Pi x=x\Pi$ and $\Pi y=y\Pi$, then $\langle x, y\rangle_u$ commutes with $\Pi$ as in (1), and hence $\langle x, y\rangle_u\in O_F$. 

(3) Suppose that $\Pi x=y$. Note that $\langle x, y\rangle_u\in O_K$ by (2), so $\langle y, x\rangle_u=\langle x, y\rangle_u^\sigma$. Then
\[
    \langle x, y\rangle_u=\langle x, \Pi x\rangle_u=-\langle \Pi x, x\rangle_u=-\langle y, x\rangle_u=-\langle x, y\rangle_u^\sigma
\]
as desired. The case $\Pi y=x$ can be treated in a similar way.
\end{proof}
Immediate consequences of these observations are the following useful facts.
\begin{cor}\label{cor:mat-under-can-base-L-1-3}
    Let $(L, \langle~,~\rangle)$ be a hermitian $(O_B, *, O_B)$-bilattice of $O_B$-rank $2$.
    \begin{enumerate}
        \item If $L\simeq \L_1$, that is, $L$ has a right $O_B$-basis $\{e, f\}$ satisfying
        \begin{equation}\label{eq:can-basis-L-1}
    ae=ea,\quad af=fa^\sigma, \quad\forall a\in O_K,\quad\text{and}\quad \Pi e=f,\quad \Pi f=\pi e,
    \end{equation}
    then the Gram matrix of $\langle~,~\rangle_u$ with respect to $\{e, f\}$ is $\begin{bmatrix}
        0 & d \\ -d & 0
    \end{bmatrix}$ for some $d\in O_K$ with $d+d^\sigma=0$.
        \item If $L\simeq \L_2$, that is, $L$ has a right $O_B$-basis $\{l, m\}$ satisfying
        \begin{equation}\label{eq:can-basis-L-2}
    al=la^\sigma,\quad am=ma, \quad\forall a\in O_K,\quad\text{and}\quad \Pi l=m,\quad \Pi m=\pi l,
    \end{equation}
    then the Gram matrix of $\langle~,~\rangle_u$ with respect to $\{l, m\}$ is $\begin{bmatrix}
        0 & d \\ -d & 0
    \end{bmatrix}$ for some $d\in O_K$ with $d+d^\sigma=0$. 
        \item If $L\simeq \L_3$, that is, $L$ has a right $O_B$-basis $\{\xi, \eta\}$ satisfying
           \begin{equation}\label{eq:can-basis-L-3}
        a\xi=\xi a,\quad a\eta=\eta a^\sigma, \quad\forall a\in O_K,\quad\text{and}\quad \Pi \xi=\xi \Pi,\quad \Pi \eta= \eta\Pi,
    \end{equation}
    then the Gram matrix of $\langle~,~\rangle_u$ with respect to $\{\xi, \eta\}$ is $\begin{bmatrix}
        0 & d \\ d & 0
    \end{bmatrix}$ for some $d\in O_F$.
    \end{enumerate}
    In the above cases,  $(L, \langle~,~\rangle)$ is self-dual if and only if $d$ is a unit, that is $d\in O_K^\times$ for cases (1) and (2), and $d\in O_F^\times$ for case (3). 
\end{cor}

\begin{lemma}\label{lem:self-dual-rank-two-odd}
    Let $(\L_i, \langle~,~\rangle)$ be a self-dual hermitian $(O_B, *, O_B)$-bilattice with $\L_i=(O_B^{\oplus 2}, \varphi_i)$ for $1\leq i\leq 3$. Then $(\L_i, \langle~,~\rangle_u)$ is isometric to  $\M_i$ for each $i$.
\end{lemma}

\begin{proof}
The cases $i=1, 2$ are proved by combining Corollary~\ref{cor:mat-under-can-base-L-1-3} (1) and (2) with  Remark~\ref{rem:indep-c}. To prove that $(\L_3, \langle~,~\rangle_u)$ is isometric to $\M_3$, observe that the  standard right $O_B$-basis $\{\xi\coloneqq (1, 0)^t, \eta\coloneqq (0, 1)^t\}$  of $\L_3$ satisfies \eqref{eq:can-basis-L-3}, so
    by Corollary~\ref{cor:mat-under-can-base-L-1-3} (3) the Gram matrix of $\langle~,~\rangle_u$ relative to $\{\xi, \eta\}$ is of the form $\begin{bmatrix}
        0 & d \\ d & 0
    \end{bmatrix}$, with $d\coloneqq\langle \xi, \eta\rangle_u$ lying in $O_F^\times$. 
   Now the right $O_B$-basis $\{\xi d^{-1}, \eta\}$ of $\L_3$ again satisfies \eqref{eq:can-basis-L-3}, under which the Gram matrix of $\langle~,~\rangle_u$ is $\begin{bmatrix}
        0 & 1 \\ 1 & 0
    \end{bmatrix}$. It follows that $(\L_3, \langle~,~\rangle_u) $ is isometric to $\M_3$.
\end{proof}

The following lemma together with Lemma~\ref{lem:self-dual-rank-two-odd} shows that $\{\M_1, \M_2, \M_3\}$ in fact forms a complete list of self-dual hermitian $(O_B,\,', O_B)$-bilattices of $O_B$-rank $2$ up to isometry.

\begin{lemma}\label{lem:dual-latt-odd}
    For the four indecomposable $(O_B, O_B)$-bilattices $\O=(O_B, \id)$, $\P=(O_B, \tau)$, $\L_1=(O_B^{\oplus 2},\varphi_1)$, and $\L_2=(O_B^{\oplus 2},\varphi_2)$, we have the following isomorphisms between these bilattices and their $O_B$-duals:
    \[
        \O^{\sharp}\simeq \P,\quad \P^{\sharp}\simeq \O, \quad \L_1^{\sharp}\simeq \L_1,\quad \L_2^{\sharp}\simeq \L_2.
    \]
\end{lemma}
 
\begin{proof}
 We present only the proofs of $\O^{\sharp}\simeq \P$ and $\L_1^{\sharp}\simeq \L_1$, as the proof of the remaining two isomorphisms can be carried out similarly. By Corollary~\ref{cor:one-to-one-corre}, it is enough to work out  the structures of the truncated bimodules $\O^{\sharp}/\O^{\sharp}\Pi$ and $\L_1^{\sharp}/\L_1^{\sharp}\Pi$. For $\O^{\sharp}$, there is a canonical isomorphism  of $O_F$-modules: 
   \[
       O_B\xrightarrow{\simeq} \O^{\sharp}=\Hom_{\mathrm{Mod}-O_B}(\O, O_B),\quad x\mapsto (y\mapsto xy).
   \]
From \eqref{eq:L-prime-bim-str}, the  $(O_B, O_B)$-bimodule structure on $\O^{\sharp}$ transports to an $(O_B, O_B)$-bimodule structure on $O_B$  given as follows:
   \[
       \alpha\cdot x \cdot\beta= \ol{\beta} x \id(\alpha^*)=\ol{\beta} x \alpha^*,\quad
       \forall \alpha, \beta, x\in O_B.
   \]
   Hence $\O^{\sharp}/\O^{\sharp}\Pi\simeq O_B/ (O_B\cdot\Pi)=O_B/\ol{\Pi} O_B=\F_{q^2}$. For every $a\in O_K$ and $x\in O_B$, we have 
   \[
       a\cdot x=x a^*=x u\Pi a^\sigma \Pi^{-1} u^{-1}=xua u^{-1}\equiv xa \bmod \Pi,\quad
       x\cdot a= \ol{a}x=a^\sigma x.
   \]
   It follows that $\O^{\sharp}/\O^{\sharp}\Pi\simeq W_{\tau}\simeq \P/\P \Pi$, where $W_{\tau}$ is given in Example~\ref{eg:indecomp-trunc}, so $\O^{\sharp}\simeq \P$ by Corollary~\ref{cor:one-to-one-corre}.

   To compute the $O_B$-dual of $\L_1$, we consider the $(O_B, O_B)$-bilattice $\tilde{L}_1\coloneqq (O_B^{\oplus 2}, \tilde{\varphi}_1)$, where $\tilde{\varphi}_1=\varphi_1(u)^{-1}\varphi_1\varphi_1(u)$. Then $\tilde{L}_1\simeq \L_1$ since $\tilde{\varphi}_1$ is $\GL_2(O_B)$-conjugate to $\varphi_1$, and hence $\tilde{L}_1^{\sharp}\simeq \L_1^{\sharp}$. 
   For $\tilde{L}_1^{\sharp}$, there is a canonical isomorphism of $O_F$-modules:
   \[
       O_B\oplus O_B \xrightarrow{\simeq} \tilde{L}_1^{\sharp}=\Hom_{\mathrm{Mod}-O_B}(\tilde{L}_1, O_B),\quad (x, y)\mapsto
       \left(\begin{bmatrix}
           z \\ w
       \end{bmatrix}\mapsto xz+yw  \right).
   \]
    As in the previous case, we carry the $(O_B, O_B)$-bimodule structure of $\tilde{L}_1^{\sharp}$ to an $(O_B, O_B)$-bimodule structure on $O_B\oplus O_B$ given as follows:
   \[
       \alpha\cdot (x, y)\cdot \beta= (\ol{\beta}x, \ol{\beta}y)\tilde{\varphi_1}(\alpha^*)=(\ol{\beta}x, \ol{\beta}y)\varphi_1(\alpha'),\quad
       \forall \alpha, \beta, x, y\in O_B.
   \]
   Hence $\tilde{L}_1^{\sharp}/\tilde{L}_1^{\sharp}\Pi\simeq (O_B\oplus O_B)/((O_B\oplus O_B)\cdot\Pi)= (O_B\oplus O_B)/\ol{\Pi}(O_B\oplus O_B)=\F_{q^2}\oplus\F_{q^2}$. For any  $a\in O_K$ and $x, y\in O_B$, we have 
   \begin{align*}
       a\cdot (x, y)&= (x, y)\varphi_1(a')=(x, y)\varphi_1(a)
    =(x, y)\begin{bmatrix} a & \\ & a^\sigma \end{bmatrix},\\
       (x, y)\cdot a&=(a^\sigma x, a^\sigma y).
   \end{align*}
  If we  put $W\coloneqq (O_B\oplus O_B)/((O_B\oplus O_B)\cdot \Pi))=\F_{q^2}\oplus\F_{q^2}$, then $W^0=0\oplus\F_{q^2}$ and $W^1=\F_{q^2}\oplus 0$, where $W^i$ is defined in \eqref{eq:defn-W-0-W-1}. Since 
   \[
       \Pi\cdot (x, y)=(x, y) \varphi_1(\Pi')=(x, y)\begin{bmatrix} 0 & -\pi \\ -1 & 0
       \end{bmatrix},
   \]
the image $\varepsilon\coloneqq \Pi+\pi O_B$   of $\Pi$ in $O_B/\pi O_B$ defines an isomorphism of right $\F_{q^2}$-vector spaces from $W^0$ to $W^1$, and hence $W\simeq W_{\varphi_1}$, where $W_{\varphi_1}\coloneqq \L_1/\L_1\Pi$ is given in Example~\ref{eg:indecomp-trunc}. By Corollary~\ref{cor:one-to-one-corre}, we have $\L_1^{\sharp}\simeq \tilde{L}_1^{\sharp}\simeq \L_1$.
\end{proof}
Let $(L, \langle~,~\rangle)$ be a hermitian $(O_B, *, O_B)$-bilattice with structural invariant $(r_1, r_2, t_1, t_2)$.
From Lemma~\ref{lem:dual-latt-odd}, in order for $(L, \langle~,~\rangle)$ to be self-dual, we must have $r_1=r_2$. Conversely, we  show that this condition is also sufficient for the existence of such a self-dual hermitian bilattice. For this, we need the following lemma.
\begin{lemma}\label{lem:isotropic-compt-odd}
Suppose that $(L, \langle~,~\rangle)$ is a hermitian $(O_B, *, O_B)$-bilattice with structural invariant $(r, r, t_1, t_2)$ so that $L=\L_1^{\oplus t_1}\oplus \L_2^{\oplus t_2}\oplus \L_3^{\oplus r}$, where $\L_3$ is given by \eqref{eq:defn-L-3}. If $(L, \langle~,~\rangle)$ is self-dual, then $\L_1^{\oplus t_1}$, $\L_2^{\oplus t_2}$, and $\L_3^{\oplus r}$ are all self-dual with respect to the restrictions of the hermitian pairing $\langle~,~\rangle$.
\end{lemma}
\begin{proof}
    First note that the involution $'$ keeps $\pi O_B$ stable, so it induces an involution on $\F_{q^2}[\varepsilon]=O_B/\pi O_B$, which we also denote by $'$ by an abuse of notation. Since $\Pi'=-\Pi$ and $'$ acts as identity on $O_K$, we have $\varepsilon'=-\varepsilon$ and $a'=a$  for all $a\in \F_{q^2}$. 
     As in Section~\ref{sec:1}, we write $\overline{L}$ for the truncation  $L/ L\Pi$ of $L$, which is a bimodule over $(O_B/\pi O_B, O_B/O_B\Pi)=(\F_{q^2}[\varepsilon], \F_{q^2})$. 
     Clearly, $\langle~,~\rangle_u$ induces a (possibly degenerate) hermitian form $H: \overline{L}\times \overline{L}\to O_B/ O_B\Pi=\F_{q^2}$ on the  right $\F_{q^2}$-vector space $\overline{L}$, that is, 
    \[
        H(x, y)=H(y, x)^\sigma,\quad H(xa, yb)=a^\sigma H(x, y) b, \qquad \forall x,y\in\overline{L},\ \forall a, b\in\F_{q^2}. 
    \]
    Moreover, by \eqref{eq:herm-bimod-eq-1} it satisfies
    \begin{equation}\label{eq:herm-bimod-eq-H}
        H(\alpha x, y)=H(x, \alpha' y), \quad \forall x, y\in\overline{L},\,\forall \alpha\in\F_{q^2}[\varepsilon].
    \end{equation}
    Since $O_B$ is a (non-commutative) discrete valuation ring with uniformizer $\Pi$, 
 a matrix $C\in \Mat_n(O_B)$ is invertible if and only if its reduction $\overline{C}\in\Mat_n(\F_{q^2})$ modulo $\Pi$ is invertible. Therefore,  $(L, \langle~,~\rangle_u)$ is self-dual if and only if  $H$ is non-degenerate. 
    For each $i=1, 2, 3$, let $T_i\coloneqq \L_i^{\oplus t_i}$ with $t_3\coloneqq r$. The lemma would be proved once we show that  the restrictions of $H$ to $\overline{T}_1$, $\overline{T}_2$ and $\overline{T}_3$ are all non-degenerate.

According to  Lemma~\ref{lem:linear-alg-descrip-trun-bimod},  $\overline{L}$ decomposes as a direct sum $\overline{L}^0\oplus \overline{L}^1$ of  $(\F_{q^2}, \F_{q^2})$-sub-bimodules, where
    \[
        \overline{L}^0\coloneqq \{x\in \overline{L}\mid ax=xa,\,\forall a\in\F_{q^2}\} ,\quad \overline{L}^1\coloneqq \{x\in \overline{L}\mid  ax=xa^\sigma,\,\forall a\in\F_{q^2} \}.
    \]
    Similarly, $\overline{T}_i=\overline{T}_i^0\oplus\overline{T}_i^1$ for each $i$. We note that $\overline{T}_3^0=\overline{\O}^{\oplus r}$ and $\overline{T}_3^1=\overline{\P}^{\oplus r}$ since $T_3=\L_3^{\oplus r}=\O^{\oplus r}\oplus \P^{\oplus r}$ by definition.
    It follows that there are further decompositions
    \[
        \overline{L}^0= \overline{T}_1^0\oplus \overline{T}_2^0\oplus \overline{T}_3^{0},\qquad 
        \overline{L}^1=\overline{T}_1^1\oplus \overline{T}_2^1\oplus\overline{T}_3^{1}.
    \]  
  We show that the following statements hold in the hermitian space $(\overline{L}, H)$: 
     \begin{enumerate}[itemsep=4pt]
         \item Both $\overline{L}^0$ and $\overline{L}^1$ are totally isotropic, that is, $H(\overline{L}^0, \overline{L}^0)=0=H(\overline{L}^1, \overline{L}^1)$.
     
         \item $\overline{T}_1^1$ is orthogonal to $\overline{T}_2^0\oplus \overline{T}_3^0$, that is, $H(\overline{T}_1^1,\, \overline{T}_2^0\oplus \overline{T}_3^0)=0$.
         
         \item $\overline{T}_2^0$ is orthogonal to $\overline{T}_1^1\oplus \overline{T}_3^1$, that is, $H(\overline{T}_2^0,\, \overline{T}_1^1\oplus \overline{T}_3^1)=0$.
     \end{enumerate}
    First suppose  that $x, y\in \overline{L}^0$. Since $a'=a$ for every $a\in\F_{q^2}$, we have
     \[
         a^\sigma H(x, y)=H(xa, y)=H(ax, y)=H(x, ay)=H(x, ya)=H(x, y)a,
     \] 
     which implies that $H(x, y)=0$ since $a$ is arbitrary. Similarly, $\overline{L}^1$ is totally isotropic. Next, note that $\varepsilon y=0$ for any   $y\in \overline{T}_2^0\oplus \overline{T}_3^0$,  so the equality $\varepsilon'=-\varepsilon$ implies that 
     \[
         H(\varepsilon x, y)=-H(x, \varepsilon y)=0, \qquad \forall x\in\overline{T}_1^0.
     \]
     Then statement (2) follows from the fact that $\varepsilon x$ ranges through $\overline{T}_1^1$ when $x$ ranges through $\overline{T}_1^0$. Lastly, statement (3) can be proved in exactly the same  way as (2).

      For each $i=1, 2, 3$, let $\mathscr{B}_i$ be an $\F_{q^2}$-basis of $\overline{T}_i^0$, and $\mathscr{C}_i$ be an $\F_{q^2}$-basis of $\overline{T}_i^1$.
   Then
     \begin{equation}\label{eq:basis-total-space-odd}
          \mathscr{B}_1\sqcup \mathscr{C}_1\sqcup \mathscr{B}_2\sqcup \mathscr{C}_2\sqcup \mathscr{B}_3\sqcup \mathscr{C}_3
     \end{equation}
     forms an $\F_{q^2}$-basis of the space $
     \overline{L}=\bigoplus_{j=1}^3(\overline{T}_j^0 \oplus \overline{T}_j^1)$. 
  Statements (1), (2) and (3) show respectively that
 \[
     H(\mathscr{B}_i, \mathscr{B}_j)=H(\mathscr{C}_i, \mathscr{C}_j)=0,\quad
     H(\mathscr{C}_1, \mathscr{B}_2)=H(\mathscr{C}_1, \mathscr{B}_3)=0,\quad
     H(\mathscr{B}_2, \mathscr{C}_1)=H(\mathscr{B}_2, \mathscr{C}_3)=0.
 \]
 It follows that the Gram matrix of $H$ with respect to the $\F_{q^2}$-basis \eqref{eq:basis-total-space-odd} of $\overline{L}$ can be put into a block matrix form as
\[
    \begin{bNiceMatrix}[first-row,first-col]
           & t_1 & t_1 & t_2 & t_2 & r & r \\
          t_1 & 0 & C_{12} & 0 & C_{14} & 0 & C_{16} \\
          t_1 & C_{21} & 0 & 0 & 0 & 0 & 0 \\
          t_2 & 0 & 0 & 0 & C_{34} & 0 & 0 \\ 
          t_2 & C_{41} & 0 & C_{43} & 0 & C_{45} & 0 \\
          r & 0 & 0 & 0 & C_{54} & 0 & C_{56} \\
          r & C_{61} & 0 & 0 & 0 & C_{65} & 0
    \end{bNiceMatrix}\in\Mat_n(\F_{q^2}), 
\] 
where $n\coloneqq 2r+2t_1+2t_2$, and
 we put the numbers $t_j$'s and $r$ on the side of the matrix to indicate the sizes of respective blocks. 
Suppose that $(L, \langle~,~\rangle_u)$ is self-dual so that this block matrix is invertible. Necessarily, $C_{12}$ is invertible 
     since all other blocks in the same column  are zero. Then $C_{21}=(C_{12}^\sigma)^t$ is also invertible since $H$ is a hermitian form. Likewise, both $C_{34}$ and $C_{43}$ are invertible. Since $C_{12}$ is invertible, performing suitable elementary column operations to the above block matrix, we can turn $C_{16}$ into $0$ while leaving $C_{56}$ unchanged. It follows that both $C_{56}$ and $C_{65}$ are invertible. Now the matrices
     \[
         \begin{bmatrix}
             0 & C_{12} \\
             C_{21} & 0
         \end{bmatrix},\quad\begin{bmatrix}
             0 & C_{34} \\
             C_{43} & 0
         \end{bmatrix},\quad\begin{bmatrix}
             0 & C_{56} \\
             C_{65} & 0
         \end{bmatrix}
     \]
     are all invertible, which implies that the restrictions of $H$ to $\overline{T}_1$, $\overline{T}_2$ and $\overline{T}_3$ are all non-degenerate as desired.
\end{proof}

\begin{proof}[Proof of Theorem~\ref{thm:self-dual-latt-odd}]
    The `only if' part follows from Lemma~\ref{lem:dual-latt-odd} as indicated above, while the `if' part is obtained by taking $(L, \langle~,~\rangle_u)$ to be the hermitian $(O_B, ', O_B)$-bilattice $\M_1^{\boxplus t_1}\boxplus \M_2^{\boxplus t_2}\boxplus \M_3^{\boxplus r}$ in \eqref{eq:orth-split-odd}. It remains to prove that if $(L, \langle~,~\rangle)$ is self-dual, then $(L, \langle~,~\rangle_u)$ is isometric to $\M_1^{\boxplus t_1}\boxplus \M_2^{\boxplus t_2}\boxplus \M_3^{\boxplus r}$. We proceed step-by-step.
    
    \emph{Step 1.} We claim that if $t_1>0$, then there exist right $O_B$-linearly independent elements $e, f\in L$ such that \eqref{eq:can-basis-L-1} holds and the $(O_B, O_B)$-sub-bilattice $N\coloneqq e O_B +f O_B$ of $L$ is self-dual with respect to the restriction of $\langle~,~\rangle_u$. 
    From \eqref{eq:under-latt}, there are $t_1$ direct summands of $L$ isomorphic to the $(O_B, O_B)$-bilattice $\L_1$. Let $\{e_1, f_1\}, \dots, \{e_{t_1}, f_{t_1}\}$ denote their respective standard right $O_B$-bases so that each $\{e_i, f_i\}$ satisfies \eqref{eq:can-basis-L-1}. By Lemma~\ref{lem:pairing-property-odd}~(2),  $\langle e_i, f_i\rangle_u\in O_K$ for each $1\leq i\leq t_1$.
    If $\langle e_i, f_i \rangle_u\in O_K^{\times}$ for some $i$, then $e\coloneqq e_i, f\coloneqq f_i$ satisfy the desired condition by Corollary~\ref{cor:mat-under-can-base-L-1-3}, so we assume that all $\langle e_i, f_i \rangle_u\in \pi O_K$ in the following, where $\pi=\Pi^2$ denotes a fixed uniformizer of $O_F$ as in \eqref{eq:quat-multi-rul-1}.  By Lemma~\ref{lem:isotropic-compt-odd}, $\bigoplus_{i=1}^{t_1}(e_i O_B\oplus f_i O_B)$ remains self-dual with respect to the restriction of $\langle~,~\rangle_u$. 
    It follows that at least one of the following elements
    \[
        \langle e_1, e_j\rangle_u,\quad \langle e_1, f_j\rangle_u,\quad\text{with}\quad 1\leq j\leq t_1
    \]
  lies in $O_B^\times$. By Lemma~\ref{lem:pairing-property-odd} (1), $\langle e_1, e_j\rangle_u\in O_K\Pi$ for all $1\leq j\leq t_1$. Hence there exists $1< s\leq t_1$ such that $d\coloneqq\langle e_1, f_s\rangle_u \in O_K^\times$ (see Lemma~\ref{lem:pairing-property-odd} (2)). 
  If $d-d^\sigma\in \pi O_K$, then we choose $b\in O_K^\times$ such that\footnote{We claim that there exists $x\in O_K^\times$ such that $x-x^\sigma\in O_K^\times$. Take a nonzero element $z$ of the $\F_q$-subspace $\{a\in\F_{q^2}\mid a+a^\sigma=0\}$ of $\F_{q^2}$. From the additive version of Hilbert Theorem 90, namely, $H^1(\Gal(\F_{q^2}/\F_q), \F_{q^2})=\{0\}$, there exists $y\in \F_{q^2}^\times$ such that $z=y-y^\sigma$. Let $x\in O_K^\times$ be a lift of $y$. Then we have $x-x^\sigma\in O_K^\times$ as desired.} $bd-(bd)^\sigma\in O_K^\times$. Thus, replacing $\{e_1, f_1\}$ by $\{e_1 b^\sigma, f_1 b^\sigma\}$, we may assume that $d-d^\sigma\in O_K^\times$.
    Since $e_1, f_1, e_s, f_s$ are right $O_B$-linearly independent, the elements
    \[
        e\coloneqq e_1+e_s,\quad f\coloneqq f_1+f_s
    \]
    are  right $O_B$-linearly independent as well, and are easily seen to satisfy \eqref{eq:can-basis-L-1}. Moreover, from the equality $\langle e_s, f_1\rangle_u=\langle e_s, \Pi e_1\rangle_u=-\langle \Pi e_s, e_1\rangle_u=-\langle f_s, e_1\rangle_u=-d^\sigma$, we get
    \[
        \langle e, f\rangle_u=\langle e_1+e_s, f_1+f_s\rangle_u=
         \langle e_1, f_1\rangle_u+ \langle e_s, f_s\rangle_u+d-d^\sigma\in O_K^\times,
    \]
    since both $\langle e_1, f_1\rangle_u$ and  $\langle e_s, f_s\rangle_u$ belong to $\pi O_K$ by our hypothesis. Hence $e, f$ are the desired elements in our claim by Corollary~\ref{cor:mat-under-can-base-L-1-3}.

    Similarly,  if $t_2>0$, then there exist right $O_B$-linearly independent elements $l, m\in L$ such that \eqref{eq:can-basis-L-2} holds and the $(O_B, O_B)$-sub-bilattice $N\coloneqq l O_B +m O_B$ of $L$ is self-dual with respect to the restriction of $\langle~,~\rangle_u$.

    \emph{Step 2.} We claim that if $r\coloneqq r_1=r_2>0$, then there exist right $O_B$-linearly independent elements $\xi, \eta\in L$ such that \eqref{eq:can-basis-L-3} holds and the $(O_B, O_B)$-sub-bilattice $N\coloneqq \xi O_B +\eta O_B$ of $L$ is self-dual with respect to the restriction of $\langle~,~\rangle_u$. Since $r\coloneqq r_1=r_2$, from \eqref{eq:under-latt}, we have $L\simeq \L_3^{\oplus r}\oplus \L_1^{\oplus t_1}\oplus \L_2^{\oplus t_2}$. Thus there exist $r$ direct summands of $L$ isomorphic to $\L_3$. Let $\{\xi_1, \eta_1\}, \dots, \{\xi_{r}, \eta_{r}\}$ denote their respective standard right $O_B$-bases. 
  Then each $\{\xi_i, \eta_i\}$ satisfies \eqref{eq:can-basis-L-3}. By Lemma~\ref{lem:isotropic-compt-odd}, $\bigoplus_{i=1}^{r}(\xi_i O_B\oplus \eta_i O_B)$ is self-dual with respect to the restriction of $\langle~,~\rangle_u$. Thus at least one of the following elements
    \[
        \langle \xi_1, \eta_j \rangle_u,\quad \langle \xi_1, \xi_j\rangle_u,\quad\text{with}\quad  1\leq j\leq r
    \]
   lies in $O_B^\times$. Note that $\langle \xi_1, \xi_j\rangle_u\in O_F\Pi$ by Lemma~\ref{lem:pairing-property-odd} (1), so we must have $\langle \xi_1, \eta_s\rangle_u \in O_F^\times$ for some $s$ (see Lemma~\ref{lem:pairing-property-odd} (2)). Then $\xi\coloneqq \xi_1, \eta\coloneqq \eta_s$ satisfy the condition in our claim by Corollary~\ref{cor:mat-under-can-base-L-1-3}.

    \emph{Step 3.} We complete the proof by induction on the $O_B$-rank $n\coloneqq 2r+2t_1+2t_2$ of the bilattice $L$ ($n$ is necessarily even). If $n=0$, there is nothing to prove. Now suppose that $n>0$ and the theorem is true for $n-2$. If $t_1>0$, then by Step 1 there exists an $(O_B, O_B)$-sub-bilattice $N$ of $L$ isomorphic to $\L_1$
    such that $(N, \langle~,~\rangle_u|_{N})$ is self-dual. Let $N^{\bot}$ be the  orthogonal complement of $N$ in $(L, \langle~,~\rangle_u)$, that is, 
    \begin{equation}\label{eq:orth-compl-defn}
                N^{\bot}\coloneqq\{x\in L\mid \langle x, y\rangle_u=0, \forall y\in N \},
    \end{equation}
   which is an $(O_B, O_B)$-sub-bilattice of $L$ by \eqref{eq:herm-bimod-eq-1}.
    Moreover, by Lemma~\ref{lem:orth-compl-direct-summ} below, we have $L=N\boxplus N^{\bot}$. From Lemma~\ref{lem:self-dual-rank-two-odd}, $(N, \langle~,~\rangle_u|_{N})$ is isometric to $\M_1$. On the other hand, the $O_B$-rank of $N^{\bot}$ is $n-2$ and $(N^{\bot}, \langle~,~\rangle_u|_{N^{\bot}})$ is necessarily self-dual, and so the induction hypothesis applies. If $t_2>0$ or $r>0$, we proceed in the same way.
\end{proof}
\begin{lemma}\label{lem:orth-compl-direct-summ}
    Let $L$ be a right $O_B$-lattice and $\langle~,~\rangle: L\times L\to O_B$ be a hermitian form on $L$. Let $N$ be an $O_B$-sub-lattice of $L$. If the restriction $\langle~,~\rangle|_{N}$ is perfect, then $L=N\boxplus N^{\bot}$, where $N^{\bot}$ is the orthogonal complement of $N$ in $L$ defined by \eqref{eq:orth-compl-defn}.
\end{lemma}
This lemma is a variant of a similar result on quadratic forms, and can be proved in the same way; see \cite[Proposition~3.2, Chapter~I]{MR491773} or \cite[Proposition~I.2]{Knebusch-1977}.

\subsection{Self-dual hermitian $(O_B, *, O_B)$-bilattices with $\ord_{B}(\gamma)$ even}
\label{subsec:self-dual-latt-even} 
Keep the notation introduced at the beginning of Section~\ref{subsec:self-dual-latt-odd} with the only exception that $\ord_{B}(\gamma)$ is now assumed to be even. In view of $\Pi^2=\pi$ and \eqref{eq:invol-star}, we may assume without loss of generality that $\ord_{B}(\gamma)=0$, that is, $\gamma=u\in O_B^\times$ is a unit. The discussion in Section~\ref{subsec:self-dual-latt-even} runs parallel to that of Section~\ref{subsec:self-dual-latt-odd}, employing similar technique of proofs.

Let $(L,\langle~,~\rangle)$ be a hermitian $(O_B, *, O_B)$-bilattice. We attach to $\langle~,~\rangle$ another pairing $\langle~,~\rangle_u$ as defined by the equation \eqref{eq:defn-associ-pairing}. Since $u^*=-u$ in this case,  $\langle~,~\rangle_u$ is a skew-hermitian form on the right $O_B$-lattice $L$, and the equality \eqref{eq:herm-bimod-eq} is equivalent to
\begin{equation}\label{eq:herm-bimod-eq-2}
    \langle \alpha x, y\rangle_u=\langle x, \ol{\alpha} y \rangle_u,\quad\forall x, y\in L,
    \text{ }\forall\alpha\in O_B.
\end{equation}
In other words, $(L, \langle~,~\rangle_u)$ forms a skew-hermitian $(O_B, \bar{\text{ }}, O_B)$-bilattice, and $(L, \langle~,~\rangle)$ is self-dual if and only if $(L, \langle~,~\rangle_u)$ is so.

Recall from \eqref{eq:defn-four-indecomp-latt} the four basic $(O_B, O_B)$-bilattices $\O\coloneqq(O_B, \id)$, $\P\coloneqq(O_B, \tau)$, $\L_1\coloneqq (O_B^{\oplus 2}, \varphi_1)$ and  $\L_2\coloneqq (O_B^{\oplus 2}, \varphi_2)$.  Define $(O_B, O_B)$-bilattices $\L_4$, $\L_5$ and $\L_6$ as follows
\begin{equation}\label{eq:defn-L-4-6}
     \L_4\coloneqq (O_B^{\oplus 2}, \varphi_4)=\O\oplus \O,\quad \L_5\coloneqq (O_B^{\oplus 2}, \varphi_5)=\P\oplus \P,\quad \L_6\coloneqq (O_B^{\oplus 4}, \varphi_6)=\L_1\oplus \L_2,
\end{equation}
where $\varphi_4\coloneqq \id\oplus\id$, $\varphi_5\coloneqq \tau\oplus\tau$, and  $\varphi_6\coloneqq \varphi_1\oplus \varphi_2$.
For each $4\leq i\leq 6$, let $(\M_i, \langle~,~\rangle_i)$ be the self-dual hermitian $(O_B, *, O_B)$-bilattice such that the underlying $(O_B, O_B)$-bilattice $\M_i$ is equal to $\L_i$, and the Gram matrix $C_i$ of the associated pairing $\langle~,~\rangle_{i, u}$ relative to the standard right $O_B$-basis of $\M_i$ is given as follows
\begin{equation}\label{eq:block-defn-even}
   C_4=\begin{bmatrix}
        0 & 1 \\ -1 & 0
    \end{bmatrix}, \qquad C_5=\begin{bmatrix}
                      0 & 1 \\ -1 & 0
                    \end{bmatrix}, \qquad C_6=\begin{bmatrix}
                                          0 & 0 &  0 & 1 \\ 
                                          0 & 0 &  -1 & 0 \\
                                          0 & 1 &  0  & 0 \\
                                          -1 & 0 & 0 & 0
                                          \end{bmatrix}.
\end{equation}
As in the last subsection, we simply write $\M_i$ for the skew-hermitian $(O_B, \bar{\text{ }}, O_B)$-bilattice $(\M_i, \langle~,~\rangle_{i, u})$ in the sequel. 
These $\M_i$'s serve as building blocks of self-dual skew-hermitian $(O_B, \bar{\text{ }}, O_B)$-bilattices  as stated in the following theorem.
\begin{thm}\label{thm:self-dual-latt-even}
Keep the assumptions that $B$ is a quaternion division $F$-algebra with $\fchar(F)\neq 2$ and $\ord_{B}(\gamma)$ is even.
Let $L$ be an $(O_B, O_B)$-bilattice with structural invariant $(r_1, r_2, t_1, t_2)\in\Z_{\geq 0}^4$, that is,
\begin{equation}\label{eq:under-latt-1}
    L\simeq \O^{\oplus r_1}\oplus \P^{\oplus r_2}\oplus \L_1^{\oplus t_1}\oplus \L_2^{\oplus t_2}.
\end{equation}
Then there exists a perfect hermitian $(O_B, *, O_B)$-pairing $\langle~,~\rangle$ on $L$ if and only if both $r_1$ and $r_2$ are even and $t_1=t_2$. In this case, $(L, \langle~,~\rangle_u)$ is isometric to the orthogonal direct sum
\begin{equation}\label{eq:orth-split-even}
        \M_4^{\boxplus \frac{r_1}{2}}\boxplus \M_5^{\boxplus \frac{r_2}{2}}\boxplus \M_6^{\boxplus t},\qquad \text{ with } t\coloneqq t_1=t_2. 
\end{equation}
In particular, up to isometry, there exists a unique self-dual hermitian $(O_B, *, O_B)$-bilattice with given structural invariant $(r_1, r_2, t, t)$ for  $r_1, r_2\in 2\Z_{\geq 0}$ and $t\in \Z_{\geq 0}$.
\end{thm}
The proofs of the following two lemmas are similar to that of Lemmas~\ref{lem:pairing-property-odd} and \ref{lem:dual-latt-odd}, so we omit the details.
\begin{lemma}\label{lem:pairing-property-even}
    Let $(L, \langle~,~\rangle)$ be a hermitian $(O_B, *, O_B)$-bilattice, and let $L_{\id}\coloneqq\{x\in L\mid ax=xa, \forall a\in O_K\}$ and $L_{\sigma}\coloneqq\{x\in L\mid ax=xa^\sigma, \forall a\in O_K\}$ be the $(O_K, O_K)$-sub-bilattices of $L$ as defined in \eqref{eq:L-id-L-sigma}.
    Then the associated pairing $\langle~,~\rangle_u$  has the following properties.
    \begin{enumerate}[(1)]
        \item If either $x, y\in L_{\id}$ or $x, y\in L_{\sigma}$, then $\langle x, y\rangle_u\in O_K$ and $\langle x, x \rangle_u^\sigma=-\langle x, x\rangle_u$. Moreover, if further $\Pi x=x\Pi$ and $\Pi y=y\Pi$, then $\langle x, y \rangle_u\in O_F$; in particular, $\langle x, x\rangle_u=0$.
        \item If $x\in L_{\id}$ and $y\in L_{\sigma}$, then $\langle x, y \rangle_u\in O_K\Pi$. Moreover, if further $x=\Pi y$ (or $y=\Pi x$), then $\langle x, y \rangle_u=0$.
    \end{enumerate}
\end{lemma}

\begin{lemma}\label{lem:dual-latt-even}
   For the four indecomposable $(O_B, O_B)$-bilattices $\O=(O_B, \id)$, $\P=(O_B, \tau)$, $\L_1=(O_B^{\oplus 2},\varphi_1)$, and $\L_2=(O_B^{\oplus 2},\varphi_2)$, we have the following isomorphisms between these bilattices and their $O_B$-duals:
    \[
    \O^{\sharp}\simeq \O,\quad \P^{\sharp}\simeq \P, \quad \L_1^{\sharp}\simeq \L_2,\quad \L_2^{\sharp}\simeq \L_1.
    \]
\end{lemma}

\begin{cor}\label{cor:mat-under-can-base-L-4-6}
    Let $(L, \langle~,~\rangle)$ be a hermitian $(O_B, *, O_B)$-bilattice of $O_B$-rank $2$ or $4$.
    \begin{enumerate}
        \item If $L\simeq \L_4$, that is, it has a right $O_B$-basis $\{\xi, \zeta\}$ satisfying
   \begin{equation}\label{eq:can-basis-L-4}
        a\xi=\xi a,\quad a\zeta=\zeta a, \quad\forall a\in O_K,\quad\text{and}\quad \Pi \xi=\xi \Pi,\quad \Pi \zeta= \zeta\Pi,
    \end{equation}
    then the Gram matrix of $\langle~,~\rangle_u$ with respect to $\{\xi, \zeta\}$ is $\begin{bmatrix}
        0 & b \\ -b & 0
    \end{bmatrix}$ for some $b\in O_F$.
        \item If $L\simeq \L_5$, that is, it has a right $O_B$-basis $\{\eta, \mu\}$ satisfying
    \begin{equation}\label{eq:can-basis-L-5}
        a\eta=\eta a^\sigma,\quad a\mu=\mu a^\sigma, \quad\forall a\in O_K,\quad\text{and}\quad \Pi \eta=\eta \Pi,\quad \Pi \mu= \mu\Pi,
    \end{equation}
    then the Gram matrix of $\langle~,~\rangle_u$ with respect to $\{\eta, \mu\}$ is $\begin{bmatrix}
        0 & b \\ -b & 0
    \end{bmatrix}$ for some $b\in O_F$.
        \item If $L\simeq \L_6$, that is, it has a right $O_B$-basis $\{e, f, l, m\}$ satisfying
   \begin{equation}\label{eq:can-basis-L-6}
   \begin{dcases}
        ae=ea,\quad af=fa^\sigma,\quad \forall a\in O_K,\quad \Pi e=f,\quad \Pi f=\pi e,\\
        al=la^\sigma,\quad am=ma,\quad \forall a\in O_K,\quad \Pi l=m,\quad \Pi m=\pi l,
    \end{dcases}
    \end{equation}
    then the Gram matrix of $\langle~,~\rangle_u$ with respect to $\{e, f, l, m\}$ is of the form
    \begin{equation}\label{eq:mat-under-can-basis-L-6}
       C= \begin{bmatrix}
          b_1 & 0 & c_1\Pi & c_2 \\
          0 & -\pi b_1 & -c_2 & -\pi c_1\Pi \\
          c_1\Pi & c_2^\sigma & b_2 & 0 \\
          -c_2^\sigma & -\pi c_1\Pi & 0 & -\pi b_2
        \end{bmatrix},
    \end{equation}
    for some $b_1, b_2, c_1, c_2\in O_K$ with $b_1+b_1^\sigma=b_2+b_2^\sigma=0$.
    \end{enumerate}
        In  the above cases, $(L, \langle~,~\rangle)$ is self-dual if and only if $b\in O_F^\times$ for cases (1) and (2), and $c_2\in O_K^\times$ for case (3).
\end{cor}

\begin{proof}
     Statements (1) and (2) are immediate by Lemma~\ref{lem:pairing-property-even}, so we only need to prove statement (3). By Lemma~\ref{lem:pairing-property-even}, we have 
    \[
        b_1\coloneqq\langle e, e\rangle_u\in O_K,\text{ } \langle e, f\rangle_u=\langle l, m\rangle_u=0,\text{ } c_2\coloneqq\langle e, m\rangle_u\in O_K,\text{ } b_2\coloneqq\langle l, l\rangle_u\in O_K,
    \]
    and $b_1+b_1^\sigma=b_2+b_2^\sigma=0$. Moreover, $\langle e, l\rangle_u=c_1\Pi \in O_K\Pi$ for some $c_1\in O_K$.
    By \eqref{eq:herm-bimod-eq-2}, we have
    \[\langle f, f\rangle_u=\langle \Pi e, \Pi e\rangle_u=\langle e, \ol{\Pi}\Pi e \rangle_u=-\pi \langle e, e\rangle_u=-\pi b_1;\] 
    similarly, $\langle m, m\rangle_u=-\pi \langle l, l\rangle_u=-\pi b_2$. Likewise,
    \[
        \langle f , l\rangle_u=\langle \Pi e, l\rangle_u=\langle e, \ol{\Pi} l\rangle_u=
        -\langle e, \Pi l\rangle_u=-\langle e, m\rangle_u=-c_2,
    \]
    and $\langle f, m\rangle_u=-\pi \langle e, l\rangle_u=-\pi c_1\Pi$.
    Combining these computations with the fact that $\langle~,~\rangle_u$ is skew-hermitian, we find that the Gram matrix of $\langle~,~\rangle_u$ with respect to the right $O_B$-basis $\{e, f, l, m\}$ of $\L_6$ is given by
    \eqref{eq:mat-under-can-basis-L-6}.
    Finally, taking the reduction of $C$ modulo $\Pi$, it is easy to see that $C$ is invertible if and only if $c_2\in O_K^\times$.   The proof of the corollary is now complete.
\end{proof}
\begin{lemma}\label{lem:self-dual-rank-two-even}
    Let $(\L_i, \langle~,~\rangle)$ be a self-dual hermitian $(O_B, *, O_B)$-bilattice with $\L_i$ defined by \eqref{eq:defn-L-4-6} for $4\leq i\leq 6$. Then $(\L_i, \langle~,~\rangle_u)$ is isometric to  $\M_i$ for every $i$.
\end{lemma}
\begin{proof}
    Let $\xi, \zeta \in \L_4$ be the standard right $O_B$-basis of $\L_{4}=(O_B^{\oplus 2}, \varphi_4)=O\oplus O$. Then the left $O_B$-module structure of $\L_4$ is given by \eqref{eq:can-basis-L-4}.
    By Corollary~\ref{cor:mat-under-can-base-L-4-6} (1), the Gram matrix of $\langle~,~\rangle_u$ relative to the basis $\{\xi, \zeta\}$ is $\begin{bmatrix} 0 & b \\ -b & 0 \end{bmatrix}$ with $b\in O_K^\times$.
 Now the right $O_B$-basis  $\{\xi'\coloneqq \xi b^{-1}, \zeta'\coloneqq \zeta\}$ of $\L_4$ still satisfies the relations \eqref{eq:can-basis-L-4}, under which  the Gram matrix of $\langle~,~\rangle_u$ is $\begin{bmatrix}
       0 & 1 \\ -1 & 0
    \end{bmatrix}$. It follows that $(\L_4, \langle~,~\rangle_u)$ is isometric to $\M_4$. The case $i=5$ can be treated in exactly the same way. 

    To prove that $(\L_6, \langle~,~\rangle_u)$ is isometric to $\M_6$, let $e, f, l,  m\in \L_6$ be the standard right $O_B$-basis of $\L_6=(O_B^{\oplus 4}, \varphi_6)=\L_1\oplus \L_2$. Then the left $O_B$-module structure of $\L_6$ is given by \eqref{eq:can-basis-L-6}.
    As before, we modify the basis $\{e, f, l, m\}$ to obtain another right $O_B$-basis of $\L_6$ satisfying \eqref{eq:can-basis-L-6}, under which the Gram matrix of $\langle~,~\rangle_u$ coincides with $C_6$ in \eqref{eq:block-defn-even}.
    However, this process is not as straightforward as before. Suppose that the Gram matrix of $\langle~,~\rangle_u$ under $\{e, f, l, m\}$ is given by $C$ as in \eqref{eq:mat-under-can-basis-L-6}. 
    We first point out that we can assume without loss of generality that $b_1\in O_K^\times$. Suppose otherwise that $b_1\in \pi O_K$. Recall that $c_2\in O_K^\times$, so we can choose $b\in O_K^\times$ such that $(bc_2)-(bc_2)^\sigma\in O_K^\times$. Note that $\langle eb^\sigma, m\rangle_u=b\langle e, m\rangle_u=bc_2$. Replacing $\{e, f, l, m\}$ by $\{eb^\sigma, fb^\sigma, l, m\}$ if necessary, we assume directly that $c_2-c_2^\sigma\in O_K^\times$. The right $O_B$-basis 
    \[
        e'=e+m,\quad f'=f+\pi l,\quad l'=l,\quad m'=m
    \]
   of $\L_6$ satisfies \eqref{eq:can-basis-L-6} and
   \begin{equation}
       \langle e', e'\rangle_u=\langle e+m, e+m\rangle_u
                               =b_1+c_2-c_2^\sigma-\pi b_2\in O_K^\times
   \end{equation}        
    since $b_1\in \pi O_K$ by our assumption. Replacing $\{e, f, l, m\}$ by $\{e', f', l', m'\}$, we assume that $b_1\in O_K^\times$. Under this assumption, the elements
    \[
        e''=e,\quad f''=f,\quad l''=l-e(b_1^{-1}c_1\Pi), \quad  m''=m-f(b_1^{-1}c_1\Pi)
    \]
    form another right $O_B$-basis of $\L_6$ satisfying \eqref{eq:can-basis-L-6}. Moreover,  we have
    \[
        \langle e'', l''\rangle_u=\langle e, l-e(b_1^{-1}c_1\Pi)\rangle_u=c_1\Pi-b_1(b_1^{-1}c_1\Pi)=0.
    \]
Therefore, we could have assumed $c_1=0$ at the very beginning.  More explicitly, from now on we assume (by an abuse of notation) that  $e, f, l,  m\in \L_6$ is a right $O_B$-basis of $\L_6$ that satisfies \eqref{eq:can-basis-L-6} with the Gram matrix given by \eqref{eq:mat-under-can-basis-L-6} and $c_1\Pi\coloneqq \langle e, l\rangle_u=0$. 
The previous intermediate  assumptions $c_2-c_2^\sigma\in O_K^\times$  and $b_1\in O_K^\times$  are no longer needed for the remaining part of the proof. 
    
    Since $b_1+b_1^\sigma=b_2+b_2^\sigma=0$,  by Lemma~\ref{lem:solve-eq} below there exists $d\in O_K$ such that
    \[
        (dc_2)-(dc_2)^\sigma-\pi b_2 dd^\sigma+ b_1=0.
    \]
    The elements
    \[
        e_1=e+md,\quad f_1=f+l(\pi d),\quad l_1=l,\quad m_1=m
    \]
    still satisfy \eqref{eq:can-basis-L-6} and form a right $O_B$-basis of $\L_6$. Computing directly, we get 
    \[
        \langle e_1, e_1\rangle_u=(dc_2)-(dc_2)^\sigma-\pi b_2 dd^\sigma+ b_1=0,\text{ and }
        \langle e_1, l_1\rangle_u=0.
    \]
    Put $c\coloneqq\langle e_1, m_1\rangle_u \in O_K^\times$. Then the Gram matrix of $\langle~,~\rangle_u$ under $\{e_1, f_1, l_1, m_1\}$ is given by
    \[
        \begin{bmatrix}
            0 & 0 & 0 & c \\
            0 & 0 & -c & 0 \\
            0 & c^\sigma & b_2 & 0 \\
            -c^\sigma & 0 & 0 & -\pi b_2 
        \end{bmatrix}.
    \]
  Since $K/F$ is unramified, the trace map $\Tr_{K/F}: O_K\to O_F$ is surjective. We can easily adapt the proof of \cite[Theorem~10.1]{Lang-Algebra} to show that the additive form of Hilbert theorem 90 holds for $O_K$, that is, $H^1(\Gal(K/F), O_K)=\{0\}$. This implies that 
   there exists $h\in O_K$ such that
    \[
        (hc)-(hc)^\sigma+b_2=0,
    \]
    which also follows directly from Lemma~\ref{lem:solve-eq}.
    Consider the right $O_B$-basis of $\L_6$ as follows:
    \[
    e_2=e_1,\quad f_2=f_1,\quad l_2=l_1-f_1h^\sigma,\quad m_2=m_1-e_1(\pi h^\sigma).
    \]
    These elements still satisfy  \eqref{eq:can-basis-L-6} and we have 
    \[
        \langle l_2, l_2\rangle_u=(hc)-(hc)^\sigma+b_2=0, \text{ and } \langle e_2, e_2\rangle_u=\langle e_2, l_2\rangle_u=0.
    \]
    Thus the Gram matrix of $\langle~,~\rangle_u$ under $\{e_2, f_2, l_2, m_2\}$ is
    \[
        \begin{bmatrix}
            0 & 0 & 0 & c \\
            0 & 0 & -c & 0 \\
            0 & c^\sigma & 0 & 0 \\
            -c^\sigma & 0 & 0 & 0
        \end{bmatrix}.
    \]
    Now  $\{e_2c^{-1}, f_2c^{-1}, l_2, m_2\}$ is a right $O_B$-basis of $\L_6$ satisfying \eqref{eq:can-basis-L-6} whose  Gram matrix under $\langle~,~\rangle_u$  coincides with $C_6$ in  \eqref{eq:block-defn-even} as desired.
\end{proof}

\begin{lemma}\label{lem:solve-eq}
    For any $c\in O_K^\times$ and $b_1, b_2\in O_K$ with $b_1+b_1^\sigma=b_2+b_2^\sigma=0$, there exists $d\in O_K$ such that
    \begin{equation}\label{eq:e22-1}
        (dc)-(dc)^\sigma-\pi b_2 dd^\sigma+ b_1=0.
    \end{equation}
\end{lemma}
\begin{proof}
    Without loss of generality, we assume that $c=1$. We construct inductively a sequence $\{d_n\}_{n\geq 1}$ in $O_K$ with $d_n\equiv d_{n-1} \bmod \pi^{n-1}$ such that
    \begin{equation}\label{eq:e22-2}
        d_n-d_n^\sigma-\pi b_2 d_n d_n^\sigma+ b_1\equiv 0 \bmod \pi^n.
    \end{equation}
 From the additive form of Hilbert Theorem 90, namely, $H^1(\Gal(\F_{q^2}/\F_{q}), \F_{q^2})=\{0\}$, there exists $d_1\in O_K$ such that $d_1-d_1^\sigma+b_1\equiv 0 \bmod \pi$. In other words, equation \eqref{eq:e22-2} holds for $n=1$. Suppose that $n>1$, and $d_{n-1}$ has already been defined. By the induction hypothesis, there exists an $a\in O_K$ such that
    \[
        d_{n-1}-d_{n-1}^\sigma-\pi b_2 d_{n-1} d_{n-1}^\sigma+ b_1=\pi^{n-1} a.
    \]
   The assumption on $b_1$ and $b_2$ now  implies that $a+a^\sigma=0$,  so we can choose $b\in O_K$ such that $b-b^\sigma+a\equiv 0 \bmod \pi$ as above. Put $d_n\coloneqq d_{n-1}+\pi^{n-1}b$ so that $d_n\equiv d_{n-1} \bmod \pi^{n-1}$. A direct computation shows that
    \[
        d_n-d_n^\sigma-\pi b_2 d_n d_n^\sigma+ b_1=\pi^{n-1}(b-b^\sigma+a-\pi b_2(d_{n-1}^\sigma b+ d_{n-1}b^\sigma+\pi^{n-1}bb^\sigma)).
    \]
    In particular, \eqref{eq:e22-2} holds  by our choice of $b$. Define $d_n$ recursively in this way for all $n$. By construction, $\{d_n\}$ is a Cauchy sequence with respect to the $\pi$-adic topology on $O_K$, and it converges to a solution of the equation \eqref{eq:e22-1} for $c=1$.
\end{proof}
\begin{lemma}\label{lem:isotropic-compt-even}
    Suppose that $(L, \langle~,~\rangle)$  is a hermitian $(O_B, *, O_B)$-bilattice with structural invariant $(r_1, r_2, t, t)$ so that $L=\O^{\oplus r_1}\oplus \P^{\oplus r_2}\oplus \L_6^{\oplus t}$, where $\L_6$ is given by \eqref{eq:defn-L-4-6}.
    If $(L, \langle~,~\rangle)$ is self-dual, then $\O^{\oplus r_1}$, $\P^{\oplus r_2}$, and $\L_6^{\oplus t}$ are all self-dual with respect to the restrictions of the hermitian pairing $\langle~,~\rangle$. 
\end{lemma}
\begin{proof}
Clearly,  the canonical involution $\bar{\text{ }}$ of $B$  induces an involution on $\F_{q^2}[\varepsilon]=O_B/\pi O_B$, which is still denoted by $\bar{\text{ }}$. More explicitly,  we have $\ol{\varepsilon}=-\varepsilon$ and $\ol{a}=a^\sigma$  for  $\varepsilon\coloneqq \Pi+\pi O_B$ and for all $a\in \F_{q^2}$. 
 As in Section~\ref{sec:1}, we write $\overline{L}$ for the truncated bimodule\footnote{Admittedly, there seems to be a conflict of notation since we use $\bar{\phantom{a}}$ for the canonical involution and also $\overline{\phantom{L}}$ for the truncated bimodule. However, we shall never apply the canonical involution to a bilattice, so the notation should be clear in that sense. } $L/ L\Pi$, which is a bimodule over $(O_B/\pi O_B, O_B/ O_B \Pi)=(\F_{q^2}[\varepsilon], \F_{q^2})$. 
 The perfect pairing  $\langle~,~\rangle_u$ induces a  non-degenerate skew-hermitian form on the right $\F_{q^2}$-vector space $\overline{L}$, that is,
    \[
        H(x, y)=-H(y, x)^\sigma,\quad H(xa, yb)=a^\sigma H(x, y) b, \qquad \forall x,y\in\overline{L},\,\forall a, b\in\F_{q^2}.
    \]
In addition, by \eqref{eq:herm-bimod-eq-2} it satisfies
    \begin{equation}\label{eq:herm-bimod-eq-H-even}
        H(\alpha x, y)=H(x, \ol{\alpha} y),\quad \forall x, y\in\overline{L},\,\forall\alpha\in\F_{q^2}[\varepsilon].
    \end{equation}  
 Let $T_1=\L_1^{\oplus t}$, $T_2=\L_2^{\oplus t}$, $T_3=\O^{\oplus r_1}$ and  $T_4=\P^{\oplus r_2}$. Then $\L_6^{\oplus t}=T_1\oplus T_2$, which we denote by $T$ for simplicity.
 Similar as in the proof of Lemma~\ref{lem:isotropic-compt-odd}, to prove the lemma, it is enough to  check that the restrictions of $H$ to $\overline{T}_3$, $\overline{T}_4$ and $\overline{T}$ are all non-degenerate.

    By Lemma~\ref{lem:linear-alg-descrip-trun-bimod}, $\overline{L}$ decomposes as a direct sum $\overline{L}^0\oplus \overline{L}^1$ of $(\F_{q^2}, \F_{q^2})$-sub-bimodules, where
    \[
        \overline{L}^0=\{x\in \overline{L}\mid ax=xa,\,\forall a\in\F_{q^2}\} ,\quad \overline{L}^1=\{x\in \overline{L}\mid  ax=xa^\sigma,\,\forall a\in\F_{q^2} \}.
    \]
    Similarly, $\overline{T}_i=\overline{T}_i^0\oplus\overline{T}_i^1$ for each $i$. Note that $\overline{T}_3=\overline{T}_3^0$, $\overline{T}_3^1=0$,  and $\overline{T}_4^0=0$,  $\overline{T}_4=\overline{T}_4^1$. Then
    \[
        \overline{L}^0=\overline{T}_3\oplus \overline{T}_1^0\oplus \overline{T}_2^0,\quad
        \overline{L}^1=\overline{T}_4\oplus \overline{T}_1^1\oplus \overline{T}_2^1.
    \]
     We check that the following statements hold in the skew-hermitian space $(\overline{L}, H)$:
     \begin{enumerate}[itemsep=4pt]
         \item $\overline{L}^0$ and $\overline{L}^1$ are orthogonal to each other, that is, $H(\overline{L}^0, \overline{L}^1)=0$.
         \item $\overline{T}_2^0$ and $\overline{T}_1^1$ are totally isotropic subspaces, that is, $H(\overline{T}_2^0, \overline{T}_2^0)=H(\overline{T}_1^1, \overline{T}_1^1)=0$.
         \item $\overline{T}_2^0$ is orthogonal to $\overline{T}_3$, $\overline{T}_1^1$ is orthogonal to $\overline{T}_4$, that is, $H(\overline{T}_2^0, \overline{T}_3)=H(\overline{T}_1^1, \overline{T}_4)=0$.
     \end{enumerate}
     Suppose that $x\in \overline{L}^0$ and $y\in \overline{L}^1$.  For every $a\in\F_{q^2}$, we have $\ol{a}=a^\sigma$, so
     \[
           a^\sigma H(x, y)=H(xa, y)=H(ax, y)=H(x, a^\sigma y)=H(x, ya)=H(x, y)a,
     \]
     which implies that $H(x, y)=0$ since $a$ is arbitrary. This proves statement (1). Next, note that $\varepsilon y=0$ for all $y\in \overline{T}_4\oplus \overline{T}_1^1$, so the equality $\ol{\varepsilon}=-\varepsilon$ implies that
     \[
         H(\varepsilon x, y)=-H(x, \varepsilon y)=0, \qquad\forall x\in \overline{L}^0.
     \]
     Since $\varepsilon x$ ranges through $\overline{T}_1^1$ when $x$ ranges through $\overline{L}^0$, it follows that $\overline{T}_1^1$ is totally isotropic and  $\overline{T}_1^1$ is orthogonal to $\overline{T}_4$. The proof of the remaining statements can be carried out similarly. 

     Let $ \mathscr{B}_1, \mathscr{B}_2, \mathscr{B}_3$ be $\F_{q^2}$-bases of $\overline{T}_1^0, \overline{T}_2^0, \overline{T}_3$ respectively, and $ \mathscr{C}_1, \mathscr{C}_2, \mathscr{C}_4$ be $\F_{q^2}$-bases of $ \overline{T}_1^1, \overline{T}_2^1, \overline{T}_4$ respectively. Then
     \begin{equation}\label{eq:basis-total-space-even}
         \mathscr{B}_3\sqcup \mathscr{C}_4\sqcup \mathscr{B}_1\sqcup \mathscr{C}_1\sqcup \mathscr{B}_2\sqcup \mathscr{C}_2 
    \end{equation}
     forms an $\F_{q^2}$-basis of the space $\overline{L}= \overline{T}_3\oplus\overline{T}_4\oplus \overline{T}_1^0\oplus \overline{T}_1^1\oplus \overline{T}_2^0\oplus \overline{T}_2^1 $.
    Statements (1), (2) and (3) show respectively that
      \[
      H(\mathscr{B}_i, \mathscr{C}_j)=0,\quad
      H(\mathscr{B}_2, \mathscr{B}_2)=H(\mathscr{C}_1, \mathscr{C}_1)=0,\quad
      H(\mathscr{B}_2, \mathscr{B}_3)=H(\mathscr{C}_1, \mathscr{C}_4)=0.
      \]
     It follows that the Gram matrix of $H$ with respect to the $\F_{q^2}$-basis \eqref{eq:basis-total-space-even} of $\overline{L}$ can be put into a block matrix form as
     \[
         \begin{bNiceMatrix}[first-row,first-col]
               & r_1 & r_2 & t & t & t & t \\
          r_1 & C_{11} & 0 & C_{13} & 0 & 0 & 0 \\ 
          r_2 & 0 & C_{22} & 0 & 0 & 0 & C_{26} \\
          t & C_{31} & 0 & C_{33} & 0 & C_{35} & 0 \\ 
          t & 0 & 0 & 0 & 0 & 0 & C_{46} \\
          t & 0 & 0 & C_{53} & 0 & 0 & 0 \\
          t & 0 & C_{62} & 0 & C_{64} & 0 & C_{66}
         \end{bNiceMatrix}\in\Mat_n(\F_{q^2}),
     \]
      where $n\coloneqq r_1+r_2+4t$.
     Suppose that $(L, \langle~,~\rangle_u)$ is self-dual so that the above  matrix is invertible. As in the proof of Lemma~\ref{lem:isotropic-compt-odd}, it is easy to see that 
     both $C_{35}$ and $C_{64}$ are invertible, which in turn implies that   the matrices
     \[
         C_{11},\quad  C_{22},\quad\begin{bmatrix}
             C_{33} & 0 & C_{35} & 0 \\
              0 & 0 & 0 & C_{46} \\
               C_{53} & 0 & 0 & 0 \\
               0 & C_{64} & 0 & C_{66}
         \end{bmatrix}
     \]
     are all invertible. This shows that the restrictions of $H$ to $\overline{T}_3$, $\overline{T}_4$ and $\overline{T}$ are all non-degenerate as desired.
\end{proof}

\begin{proof}[Proof of Theorem~\ref{thm:self-dual-latt-even}] 
    The `if' part is easy as we can take $(L, \langle~,~\rangle_u)$ to be the skew-hermitian $(O_B, \bar{\phantom{a}}, O_B)$-bilattice $\M_4^{\boxplus \frac{r_1}{2}}\boxplus \M_5^{\boxplus \frac{r_2}{2}}\boxplus \M_6^{\boxplus t}$ in \eqref{eq:orth-split-even}. For the `only if' part, the equality $t_1=t_2$ is immediate by Lemma~\ref{lem:dual-latt-even}. We show  that both $r_1$ and $r_2$ are even. By Lemma~\ref{lem:isotropic-compt-even}, the skew-hermitian $(O_B, \bar{\text{ }}, O_B)$-bilattice $(\O^{\oplus r_1}, \langle~,~\rangle_u|_{\O^{\oplus r_1}})$ is self-dual. From~\eqref{eq:herm-bimod-eq-2}, there exists a skew-hermitian matrix $g\in\GL_{r_1}(O_B)$ such that $\ol{\varphi(\alpha)}^t g=g \varphi(\ol{\alpha})$ for all $\alpha\in O_B$, where $\varphi\coloneqq \id^{\oplus r_1}$ denotes the diagonal embedding $O_B\hookrightarrow\Mat_{r_1}(O_B)$. Note that $\ol{\varphi(\alpha)}^t=\ol{\varphi(\alpha)}=\varphi(\ol{\alpha})$, so $\varphi(\ol{\alpha}) g=g \varphi(\ol{\alpha})$ for all $\alpha\in O_B$, which implies that $g\in \Mat_{r_1}(F)$.
    In particular, $g$ is skew-symmetric  in $\Mat_{r_1}(F)$ because it is  skew-hermitian.  Since $\fchar(F)\neq 2$, we conclude that $r_1$ is even from the non-degeneracy of $g$. 
     Similarly, $r_2$ is  even as well. 
    
    Now suppose that both $r_1$ and $r_2$ are even and $t\coloneqq t_1=t_2$. Let  $(L, \langle~,~\rangle_u)$ be a skew-hermitian $(O_B, \bar{\phantom{a}}, O_B)$-bilattice with structural invariant $(r_1, r_2, t, t)$ so that 
    the isomorphism \eqref{eq:under-latt-1} can be rewritten as
    \begin{equation}\label{eq:under-latt-even}
        L\simeq \L_4^{\oplus \frac{r_1}{2}}\oplus \L_5^{\oplus \frac{r_2}{2}}\oplus \L_6^{\oplus t},
    \end{equation}
    where $\L_4$, $\L_5$ and $\L_6$ are defined in \eqref{eq:defn-L-4-6}.
     We show that $(L, \langle~,~\rangle_u)$ is isometric to $\M_4^{\boxplus \frac{r_1}{2}}\boxplus \M_5^{\boxplus \frac{r_2}{2}}\boxplus \M_6^{\boxplus t}$
     step-by-step as for Theorem~\ref{thm:self-dual-latt-odd}. 

    \emph{Step 1.} We claim that if $r_1>0$, then there exist right $O_B$-linearly independent elements $\xi, \zeta\in L$ such that \eqref{eq:can-basis-L-4} holds and the $(O_B, O_B)$-sub-bilattice $N\coloneqq \xi O_B +\zeta O_B$ of $L$ is self-dual with respect to the restriction of $\langle~,~\rangle_u$. 
    From \eqref{eq:under-latt-even}, there are $r_1/2$ direct summands of $L$ isomorphic to the $(O_B, O_B)$-bilattice $\L_4$.
    Let $\{\xi_1, \zeta_1\}, \dots, \{\xi_{r_1/2}, \zeta_{r_1/2}\}$ denote their standard right $O_B$-bases respectively so that each $\{\xi_i, \zeta_i\}$ satisfies \eqref{eq:can-basis-L-4}. By Lemma~\ref{lem:isotropic-compt-even}, $\bigoplus_{i=1}^{r_1/2}(\xi_i O_B\oplus \zeta_i O_B)$ is self-dual with respect to the restriction of $\langle~,~\rangle_u$. It follows that  at least one of the following elements
    \[
        \langle \xi_1, \xi_j\rangle_u,\quad \langle \xi_1, \zeta_j\rangle_u,\qquad \text{with}\quad 1\leq j\leq r_1/2
    \]
    lies in $O_B^\times$, and hence in $O_F^\times$ by Lemma~\ref{lem:pairing-property-even} (1). Put $\xi\coloneqq\xi_1$. 
    If $\langle \xi_1, \xi_s\rangle_u\in O_F^\times$ for some $s$, we take $\zeta\coloneqq\xi_s$; or if $\langle \xi_1, \zeta_{s'}\rangle_u\in O_F^\times$ for some $s'$, we take $\zeta\coloneqq\zeta_{s'}$. Note that $\langle\xi_1, \xi_1 \rangle_u=0$ by Lemma~\ref{lem:pairing-property-even}~(1), so $s>1$ if $\zeta=\xi_s$.  Then $\xi, \zeta$ are right $O_B$-linearly independent, and are easily seen to satisfy our requirement by Corollary~\ref{cor:mat-under-can-base-L-4-6}.

    Similarly, if $r_2>0$, then there exist right $O_B$-linearly independent elements $\eta, \mu\in L$ such that \eqref{eq:can-basis-L-5} holds and the $(O_B, O_B)$-sub-bilattice $N\coloneqq \eta O_B +\mu O_B$ of $L$ is self-dual with respect to the restriction of $\langle~,~\rangle_u$.

    \emph{Step 2.} We claim that if $t_1>0$, then there exist right $O_B$-linearly independent elements $e, f, l, m\in L$ such that \eqref{eq:can-basis-L-6} holds and the $(O_B, O_B)$-sub-bilattice $N\coloneqq e O_B +f O_B+l O_B+m O_B$ of $L$ is self-dual with respect to the restriction of $\langle~,~\rangle_u$. 
    From \eqref{eq:under-latt-even}, there exists $t$ direct summands of $L$ isomorphic to the $(O_B, O_B)$-bilattice $\L_6$.
    Let $\{e_1, f_1, l_1, m_1\}$, \dots, $\{e_{t_1}, f_{t_1}, l_{t_1}, m_{t_1}\}$ denote their respective standard right $O_B$-bases. Then each $\{e_i, f_i, l_i, m_i\}$ satisfies \eqref{eq:can-basis-L-6}. First note that $\langle f_1, l_1\rangle_u\in O_K$ by Lemma~\ref{lem:pairing-property-even}~(1). If $\langle f_1, l_1\rangle_u\in O_K^\times$, then the elements $e\coloneqq e_1, f\coloneqq f_1, l\coloneqq l_1, m\coloneqq m_1$ are the desired elements in our claim by Corollary~\ref{cor:mat-under-can-base-L-4-6}, so we suppose in the following that $\langle f_1, l_1\rangle_u\in \pi O_K$. From Lemma~\ref{lem:isotropic-compt-even}, $\bigoplus_{i=1}^{t}(O_B e_i\oplus O_B f_i\oplus O_B l_i\oplus O_B m_i)$ is self-dual with respect to the restriction of $\langle~,~\rangle_u$, so at least one of the following elements
    \[
        \langle f_1, e_j\rangle_u,\quad \langle f_1, f_j\rangle_u,\quad
        \langle f_1, l_j\rangle_u,\quad \langle f_1, m_j\rangle_u,\qquad\text{with}\quad 
        1\leq j\leq t_1 
    \]
    lies in $O_B^\times$.
    By Lemma~\ref{lem:pairing-property-even}, for each $1\leq j\leq t_1$, we have $\langle f_1, e_j \rangle_u\in \Pi O_K$, $\langle f_1, m_j\rangle_u \in \Pi O_K$ and 
    \[
        \langle f_1, f_j\rangle_u=\langle \Pi e_1, \Pi e_j\rangle_u=\langle e_1, \ol{\Pi}\Pi e_j\rangle_u=-\pi \langle e_1, e_j\rangle_u\in \pi O_K.
    \]
    Note that $\langle f_1, l_j\rangle_u\in O_K$ for all  $1\leq j\leq t_1$  by Lemma~\ref{lem:pairing-property-even} (1), so $\langle f_1, l_s\rangle_u\in O_K^\times$ for some $s>1$ by our hypothesis. Put
    \[
        e\coloneqq e_1,\quad f\coloneqq f_1,\quad l\coloneqq l_1+l_s,\quad m\coloneqq m_1+m_s.
    \]
    Then $\langle f, l\rangle_u=\langle f_1, l_1+l_s\rangle_u=\langle f_1, l_1\rangle_u+\langle f_1, l_s\rangle_u \in O_K^\times$ since $\langle f_1, l_1\rangle_u\in \pi O_K$ by our hypothesis. Clearly, $e, f, l, m$ are right $O_B$-linearly independent and satisfy  \eqref{eq:can-basis-L-6}. 
    Then by Corollary~\ref{cor:mat-under-can-base-L-4-6}, $e, f, l, m$ are desired elements in our claim.
    
    \emph{Step 3.} Just as in the proof of Theorem~\ref{thm:self-dual-latt-odd}, we complete the proof of the theorem  by induction on the $O_B$-rank of  $L$.
\end{proof}

\subsection{Self-dual hermitian $(O_1, *, O_2)$-bilattices when $B_1$ is split and $B_2$ is division}
\label{subsec:self-dual-latt-B-1-split}

In this section, we treat case (iii), whose assumptions on the quaternion $F$-algebras $B_1$ and $B_2$ have just been stated in the title. Henceforth we fix an identification of $B_1$ with $\Mat_2(F)$ such that $O_1$ is identified with $\Mat_2(O_F)$.
Since the involution $*$  sends every $\alpha\in B_1$ to $\alpha^*=\gamma\overline{\alpha}\gamma^{-1}$  as in \eqref{eq:invol-star}, the assumption that $*$ stabilizes $O_1=\Mat_2(O_F)$  implies that $\gamma\in F^{\times}\GL_2(O_F)$. In view of \eqref{eq:invol-star}, we may and will assume that $\gamma\in\GL_2(O_F)$. 

Consider the $(O_1, O_2)$-bilattice $\M_0\coloneqq\begin{bmatrix} O_2 \\ O_2 \end{bmatrix}$,  equipped with the natural right $O_2$-module structure and the left $O_1$-module structure given by the natural inclusion $O_1=\Mat_2(O_F)\hookrightarrow\End_{\mathrm{Mod}-O_2}(\M_0)=\Mat_2(O_2)$.

\begin{lemma}\label{lem:bilatt-rank-even}
    Let $L$ be an $(O_1, O_2)$-bilattice of $O_2$-rank $n$. Then $n$ is necessarily even, and $L$ is isomorphic to $\M_0^{\oplus \frac{n}{2}}$.
\end{lemma}
\begin{proof}
    Since $B_2$ is divison and $B_1=\Mat_2(F)$, by Remark~\ref{rem:quater-bimod-rank} the $B_2$-rank of any $(B_1, B_2)$-bimodule is necessarily even. 
    It follows that the $O_2$-rank of the $(O_1, O_2)$-bilattice $L$ is even as well. Note that $O_1\otimes_{O_F}O_2^{\opp}=\Mat_2(O_2^{\opp})$ is a maximal order, so the isomorphism class of an $(O_1, O_2)$-bilattice is completely determined by its $O_2$-rank by \cite[Theorem~18.7]{reiner:mo}. We conclude that $L\simeq \M_0^{\oplus \frac{n}{2}}$.
\end{proof}

Given a hermitian $(O_1, *, O_2)$-bilattice $(L, \langle~,~\rangle)$, we define a pairing $\langle~,~\rangle_{\gamma}: L\times L\to O_2$ by the formula
\begin{equation}\label{eq:defn-lange-range-gamma}
    \langle x, y \rangle_{\gamma}\coloneqq \langle x, \gamma y\rangle,\quad\forall x, y\in L.
\end{equation}
Then $\langle~,~\rangle_{\gamma}$ is a skew-hermitian form on the right $O_2$-lattice $L$, and equation \eqref{eq:herm-bimod-eq} is equivalent to 
\begin{equation}\label{eq:skew-herm-bimod-eq}
    \langle \alpha x, y \rangle_{\gamma}=\langle x, \ol{\alpha} y \rangle_{\gamma},\quad\forall x, y\in L, \forall\alpha\in O_1.
\end{equation}
In other words, $(L, \langle~,~\rangle_{\gamma})$ forms a skew-hermitian $(O_1, \bar{\text{ }}, O_2)$-bilattice. Since we have assumed that $\gamma\in\GL_2(O_F)$, $(L, \langle~,~\rangle)$ is self-dual if and only if $(L, \langle~,~\rangle_{\gamma})$ is so.

Let $\langle~,~\rangle_{0}:\M_0\times \M_0\to O_2$ be the perfect hermitian $(O_1, *, O_2)$-pairing on $\M_0$ such that the Gram matrix of its associated pairing $\langle~,~\rangle_{0, \gamma}$ with respect to the standard right $O_2$-basis of $\M_0$ is given by
\begin{equation*}
        \begin{bmatrix}
            0  &  1 \\  -1 & 0
        \end{bmatrix}.
\end{equation*}
For simplicity, we still write $\M_0$ for the skew-hermitian $(O_1, \bar{\text{ }}, O_2)$-bilattice $(\M_0, \langle~,~\rangle_{0, \gamma})$.

\begin{prop}\label{prop:self-dual-B-1-split}
    Keep the assumption that $B_1$ is split and $B_2$ is division. Then
    for any even integer $n\in 2\Z_{>0}$, there exists a unique self-dual hermitian $(O_1, *, O_2)$-bilattice of $O_2$-rank $n$ up to isometry.
\end{prop}

\begin{proof}
Let $(L, \langle~,~\rangle)$ be a self-dual hermitian $(O_1, *, O_2)$-bilattice. It is enough to show that $(L, \langle~,~\rangle_{\gamma})$ is isometric to $\M_0^{\boxplus \frac{n}{2}}$.
Let $\{e_{ij}\mid 1\leq i, j\leq 2\}$ be the standard basis of $B_1=\Mat_2(F)$, where $e_{ij}$ denotes the matrix whose $(i, j)$-th entry is $1$ and all other entries are $0$. Then $L=L^1\oplus L^2$, where $L^1\coloneqq e_{11}L$ and $L^2\coloneqq e_{22}L$ are right $O_2$-sublattices of $L$. For any $x, y\in L^1$, since $e_{11}x=x$, $e_{11}y=y$ and $\overline{e}_{11}=e_{22}$, we have
    \[ 
       \langle x, y\rangle_{\gamma}=\langle e_{11}x, e_{11}y\rangle_{\gamma}=\langle x, e_{22}e_{11}y\rangle_{\gamma}=0.                
     \]
Similarly, $\langle x, y\rangle_{\gamma}=0$ for all $x, y\in L^2$. Define a pairing $\psi:L^1\times L^1\to O_2$ on $L^1$ by the formula
\[
         \psi(x, y)\coloneqq\langle x, e_{21}y\rangle_{\gamma},\quad\forall\, x, y\in L^1.
\]
Then $\psi$ is a hermitian form on the right $O_2$-lattice $L^1$ by \eqref{eq:skew-herm-bimod-eq}.
The map $L^1\to L^2$ sending $x$ to $e_{21}x$ defines an isomorphism of right $O_2$-lattices with the inverse $L^2\to L^1$ given by  $y\mapsto e_{12}y$. Since $\langle~,~\rangle$ is perfect and $\langle L^1, L^1\rangle_{\gamma}=\langle L^2, L^2\rangle_{\gamma}=0$, we find that $\psi$ is perfect as well. Then by \cite[Lemma~2.5]{terakado-xue-yu:2023} or \cite[Proposition~6.1]{Jacobowitz-HermForm}, there exists a right $O_2$-basis $\{x_1, \cdots, x_{n/2}\}$ of $L^1$ whose Gram matrix under $\psi$ is the identity matrix $I_{\frac{n}{2}}$. We put  $y_i\coloneqq e_{21}x_i$ for all $1\leq i \leq n/2$ so that $\{x_1, y_1, \cdots, x_{n/2}, y_{n/2}\}$ forms a right $O_2$-basis of $L$. Then for all $1\leq i, j\leq n/2$,
\[
    \langle x_i, x_j\rangle_{\gamma}=\langle y_i, y_j\rangle_{\gamma}=0,\qquad 
    \langle x_i, y_j\rangle_{\gamma}=\psi(x_i, x_j)=\delta_{ij},
\]
where $\delta_{ij}$ is the Kronecker symbol.
For each $1\leq i\leq n/2$, if we denote $N_i\coloneqq x_iO_2\oplus y_iO_2$,  then we have
\[
       (L, \langle~,~\rangle_\gamma)=(N_1, \langle~,~\rangle_\gamma|_{N_1})\boxplus
       \cdots\boxplus (N_{\frac{n}{2}}, \langle~,~\rangle_\gamma|_{N_{\frac{n}{2}}}),\quad \text{with}\quad (N_i, \langle~,~\rangle_\gamma|_{N_i})\simeq \M_0, \quad \forall 1\leq i \leq n/2. 
\]
 This shows that $(L, \langle~,~\rangle_{\gamma})$ is isometric to $\M_0^{\boxplus \frac{n}{2}}$ as desired.
\end{proof}

\subsection{Self-dual hermitian $(O_1, *, O_2)$-bilattices when $B_2$ is split}
\label{subsec:self-dual-latt-split}
Finally, we treat cases (iv) and (v), where $B_2$ is split. Both of these cases can be reduced to the situations studied in \cite[\S 2]{terakado-xue-yu:2023}. With applications to the global case in mind, we proceed under a slightly more general assumption on $B_1$.
Suppose that $F_1$ is a finite separable extension of $F$ of degree $d\coloneqq[F_1:F]$ with ring of integers $O_{F_1}$. Let $B_1$ be a quaternion $F_1$-algebra and $O_1$ be a maximal $O_{F_1}$-order in $B_1$. Let $*$ be an orthogonal involution on $B_1$ that stabilizes $O_1$. Then by \cite[Proposition~2.21]{book-of-involution} there exists $\gamma\in B_1^\times$ such that \eqref{eq:invol-star} holds. 

\begin{lemma}\label{lem:B-2-rank-divisible-by-d}
  Let $V$ be a finitely generated $(B_1, B_2)$-bimodule. 
  Then $V$ is necessarily free as a right $B_2$-module with $d\mid \rank_{B_2}(V)$. If further $B_1$ is division, then $(2d)\mid \rank_{B_2}(V)$.
\end{lemma}
\begin{proof}
      Note that $V$ is canonically a $(B_1, B_2\otimes_{F}F_1)$-bimodule.  The lemma follows directly from Remark~\ref{rem:quater-bimod-rank} by easy rank consideration.
\end{proof}

From Lemma~\ref{lem:B-2-rank-divisible-by-d}, an $(O_1, O_2)$-bilattice is necessarily a free right $O_2$-module and its $O_2$-rank is divisible by $d$. If further $B_1$ is division, then its $O_2$-rank is divisible by $2d$.
We study self-dual hermitian $(O_1, *, O_2)$-bilattices in the following cases:
\begin{itemize}
    \item[(iv)] $B_1$ is division, and $B_2$ is split;
    \item[(v)]  $B_1$ and $B_2$ are split, and the $O_2$-rank of the underlying $(O_1, O_2)$-bilattice is divisible by $2d$.
\end{itemize}
As indicated before, one reason for the requirement on the $O_2$-rank in case (v) comes from global consideration as in Lemma~\ref{lem:emb-exist-neces-cond-2}. In fact, the $B_2$-rank of a (non-degenerate) hermitian $(B_1, *, B_2)$-bimodule is necessarily divisible by $2d$, which essentially stems from the fact that the dimension of a non-degenerate symplectic space is necessarily even by  the arguments of \cite[\S2]{terakado-xue-yu:2023}. However, we will not use this fact elsewhere, so we omit details.

We now apply the Morita equivalence  to reduce the question of hermitian modules over split quaternion algebras to that of symplectic spaces as in \cite[\S2.4--2.5]{Shimura1963-AltHermForms}.
Fix an identification of $B_2$ with $\Mat_2(F)$ such that $O_2$ is identified with $\Mat_2(O_F)$. Let $\{e_{ij}\mid 1\leq i, j\leq 2\}$ be the standard basis of $B_2=\Mat_2(F)$ as in Section~\ref{subsec:self-dual-latt-B-1-split}.
Let $(V, \langle~,~\rangle)$ be a hermitian $(B_1, *, B_2)$-bimodule, whose $B_2$-rank is divisible by $2d$. 
Let $W\coloneqq Ve_{11}$. Then for any  $x, y\in W$, we have
 \[
     \langle x, y\rangle=\langle xe_{11}, ye_{11}\rangle=e_{22}\langle x, y\rangle e_{11}\in F e_{21}.
 \]
Let $\psi: W\times W\to F$ be the map such that for all $x,y\in W$,
\[        \langle x, y \rangle=\psi(x, y)e_{21}.              \]
By \cite[Proposition 2.9]{Shimura1963-AltHermForms}, $\psi$ is a non-degenerate alternating form on the $F$-vector space $W$. Since the $F$-dimension of $W$ is $2n$, it is easy to see that $W$ is a free left $B_1$-module of rank $n/2d$.
Thus $(W, \psi)$ forms an $F$-valued skew-hermitian $(B_1, *)$-module in the sense of \cite[Definition~2.7]{terakado-xue-yu:2023}, which is unique up to isometry by \cite[Propositions~2.8 and 2.9]{terakado-xue-yu:2023}.
\begin{prop}\label{prop:self-dual-split-case}
Keep the assumption  that $B_2$ is split. For any $n\in\Z_{>0}$ divisible by $2d$, there exists a unique hermitian $(B_1, *, B_2)$-bimodule $(V, \langle~,~\rangle)$ of $B_2$-rank $n$ up to isometry.
    \begin{enumerate}[(1)]
        \item If $B_1$ is split, then $(V,\langle~,~\rangle)$ contains a unique isometry class of self-dual hermitian $(O_1, * ,O_2)$-bilattices.
        \item If $B_1$ is division, then the following statements are equivalent:
        \begin{enumerate}[(i)]
            \item $(V,\langle~,~\rangle)$ contains a self-dual hermitian $(O_1, *,O_2)$-bilattice;
            \item $n/2d$ is even or $\ord_{B_1}(\gamma)$ is odd.
        \end{enumerate}
         Moreover, such a self-dual hermitian bilattice is unique up to isometry if it exists.
    \end{enumerate}
\end{prop}
\begin{proof}
Keep $(W, \psi)$ as before the proposition.  Let $L$ be an $(O_1, O_2)$-bilattice in $V$. Then $L e_{11}=L\cap W$ is a (full) left $O_1$-lattice in $W$. By the Morita equivalence, the map $L\mapsto L e_{11}$ defines a one-to-one correspondence between the set of $(O_1, O_2)$-bilattices in $(V,\langle~,~\rangle)$ and the set of left $O_1$-lattices in $(W, \psi)$. Moreover, by \cite[Proposition~2.10]{Shimura1963-AltHermForms}, $L$ is self-dual in $(V,\langle~,~\rangle)$ if and only if $L e_{11}$ is self-dual in $(W, \psi)$. Now our proposition follows immediately from \cite[Propositions~2.8 and 2.9]{terakado-xue-yu:2023}.
\end{proof}

\section{Principally polarized superspecial \abs with quaternion action}\label{sec:QM}
In this section, we apply the classification of self-dual hermitian local quaternion bilattices to the study of principally polarized superspecial \abs with quaternion action. We  obtain the main result of the paper as previously introduced in Theorem~\ref{thm:nec-suff-cond-intro} and reproduced and proved here in Theorem~\ref{thm:genus-char}.
Throughout this section, we fix the following notation and  assumptions. Let $F$ be a totally real number field, and $O_F$ be its ring of integers. Let $B$ be a totally indefinite quaternion $F$-algebra with canonical involution $\alpha\mapsto \ol{\alpha}$. Suppose that $B$ is equipped with a positive involution $*$, that is, an involution such that $\Tr_{B/\Q}(\alpha\alpha^*)>0$ for all nonzero $\alpha\in B$. 
Let $\calO$ be a maximal $O_F$-order in $B$ stable under $*$.

As in Section~\ref{subsec:ssp-O-1-ab-var}, we fix a supersingular elliptic curve $E$ over an algebraically closed field $k$ of characteristic $p>0$. Then $D\coloneqq \End^{0}(E)=\End(E)\otimes\Q$ is the (unique) quaternion algebra over $\Q$ ramified exactly at $\{p, \infty\}$ and $\calR\coloneqq\End(E)$ is a maximal order in $D$.

Following \cite[Definition~5.5]{xue-yu:counting-av}, a $\Q$-polarization on an abelian $k$-variety $X$ is an element $\lambda$ of $\Hom(X, X^{\vee})\otimes_{\Z}\Q$ such that $N\lambda$ is a polarization for some positive integer $N$.
\begin{defn}\label{defn:QM}
   A $\Q$-polarized (resp.~polarized, resp.~principally polarized) $\calO$-\ab over $k$ is a triple $(X,\lambda,\iota)$, where $(X, \iota)$ is an $\calO$-\ab over $k$ as in Definition~\ref{defn:O-1-ab-var}, and $\lambda:X\to X^{\vee}$ is a $\Q$-polarization (resp.~polarization, resp.~principal polarization) of $X$ such that
    \begin{equation}\label{eq:QM-eq}
        \lambda\circ\iota(\alpha^*)={\iota(\alpha)}^{\vee}\circ\lambda,\quad\forall\alpha
        \in \calO. 
    \end{equation}
  A $\Q$-polarization $\lambda$ on an $\calO$-abelian variety $(X, \iota)$ is a $\Q$-polarization on $X$ satisfying \eqref{eq:QM-eq} (provided that $*$ is clear from the context). 
\end{defn}

\subsection{Self-dual hermitian $(\calO, *, \calR)$-bilattices}
\label{subsec:rel-with-latt}
We reduce the study of principally polarized superspecial $\calO$-\abs over $k$ to the study of self-dual positive definite hermitian $(\calO, *, \calR)$-bilattices.  Let $(X, \iota)$ be a superspecial $\calO$-abelian variety over $k$, and $L\coloneqq\Hom(E, X)$ be the corresponding $(\calO, \calR)$-bilattice as in Lemma~\ref{lem:emb-exist-neces-cond-2}. We show that each polarization $\lambda$ on $(X,\iota)$ induces a positive definite  $\calR$-valued hermitian form $\langle~,~\rangle_\lambda: L\times L\to \calR$ satisfying \eqref{eq:herm-bimod-eq}, and vice versa. 
 Let $\phi_E$ be the canonical principal polarization on $E$.  Following \cite[(4.6)]{Ibukiyama-Karemaker-Yu-2025},  we attach a pairing $\langle~,~\rangle_\lambda: L\times L\to \calR$ on $L$ to each polarization $\lambda$ on $X$ defined as follows:
\begin{equation}\label{eq:defn-associ-pair-of-pol}
    \langle f_1, f_2\rangle_\lambda\coloneqq \phi_E^{-1} f_1^{\vee}\lambda f_2,
        \quad \forall f_1, f_2\in L.
\end{equation}
From \cite[Lemma~4.4 (1)]{Ibukiyama-Karemaker-Yu-2025}, $\langle~,~\rangle_\lambda$ is a positive definite hermitian form on the right $\calR$-lattice $L$. Now condition \eqref{eq:QM-eq} is easily seen to be equivalent to 
\begin{equation}\label{eq:herm-form-compa-iota}
     \langle \alpha x, y\rangle_\lambda=\langle x, \alpha^*y\rangle_\lambda,\quad\forall x, y\in L,\, \forall \alpha\in \calO.
\end{equation} 

\begin{lemma}\label{lem:pol-herm-pair}
    The map $\lambda\mapsto\langle~,~\rangle_\lambda$ defines a bijection from the set of  polarizations on $(X, \iota)$ to the set of positive definite  $\calR$-valued hermitian $(\calO, *, \calR)$-forms on $L$, under which $\lambda$ is principal if and only if $\langle~,~\rangle_\lambda$ is perfect.
\end{lemma}
\begin{proof}
    Since $B$ is not isomorphic to  $D_F\coloneqq D\otimes_\Q F$, the dimension $n\coloneqq \dim X$ is necessarily divisible by $2[F:\Q]$ according to Lemma~\ref{lem:emb-exist-neces-cond-2}. In particular, $n>1$.  By the theorem of Deligne, Ogus and Shioda, we may assume without loss of generality that $X=E^n$ so that $L\coloneqq\Hom(E, X)=(\calR^{\oplus n}, \iota)$. Identifying $X^{\vee}$ with $(E^{\vee})^n$ in the canonical way, we have a canonical principal polarization on $X$ given by $\phi_X\coloneqq\phi_E\times\cdots\times\phi_E$ ($n$ copies). 
    Let $\mathscr{P}(X)$ denote the set of polarizations on $X$ and $\mathscr{H}(\calR)$ denote the set of positive definite hermitian matrices in $\Mat_n(\calR)$. Then by \cite[\S2.2]{Ibukiyama-Katsura-Oort-1986} we have a bijection
       \begin{equation}\label{eq:pol-her-mat}
    \mathscr{P}(X)\xrightarrow{1-1}\mathscr{H}(\calR),\quad\lambda\mapsto g_\lambda\coloneqq \phi_X^{-1}\circ\lambda.
        \end{equation}
From \cite[Lemma~4.4 (2)]{Ibukiyama-Karemaker-Yu-2025}, the Gram matrix of $\langle~,~\rangle_\lambda$ under the  standard right $\calR$-basis of $L$ is exactly given by $g_\lambda$, that is,
\[
           \langle x, y\rangle_\lambda=\overline{x}^t g_\lambda y, \quad\forall x, y\in L=\calR^{\oplus n}.
\]
Moreover, $\lambda$ is principal if and only if  $g_\lambda\in\GL_n(\calR)$. Therefore, the bijection \eqref{eq:pol-her-mat} shows that the map $\lambda\mapsto\langle~,~\rangle_\lambda$ is a bijection from the set of polarizations on $X$ to the set of positive definite $\calR$-valued  hermitian forms on the right $\calR$-lattice $L$, under which $\lambda$ is principal if and only if $\langle~,~\rangle_\lambda$ is perfect. From the equivalence between conditions \eqref{eq:QM-eq} and \eqref{eq:herm-form-compa-iota}, 
the desired bijection in our lemma is just the restriction of this more general bijection.
\end{proof}

Combining Lemma~\ref{lem:pol-herm-pair} with Lemma~\ref{lem:emb-exist-neces-cond-2}, we immediately obtain the following result. 
\begin{lemma}\label{lem:relat-with-self-dual-latt}
    The assignment $(X, \lambda, \iota)\mapsto (L, \langle~,~\rangle_\lambda)$ with $L\coloneqq\Hom(E, X)$ and $\langle~,~\rangle_\lambda$ defined as \eqref{eq:defn-associ-pair-of-pol} induces a bijection from the set of isomorphism classes of (principally) polarized superspecial $\calO$-abelian $k$-varieties of dimension $n$ to the set of isometry classes of (self-dual) positive definite hermitian $(\calO, *, \calR)$-bilattices of $\calR$-rank $n$. 
\end{lemma}

We sketch a rational version of Lemma~\ref{lem:relat-with-self-dual-latt} and omit the routine verification. Let $(X, \iota)$ be a supersingular $\calO$-abelian variety over $k$. To a $\Q$-polarization $\lambda$ on $(X, \iota)$,  we attach a pairing $\langle~,~\rangle_\lambda: V\times V\to D$ on the $(B, D)$-bimodule $V\coloneqq\Hom(E, X)\otimes\Q$ defined as in \eqref{eq:defn-associ-pair-of-pol} except that ``$f_1, f_2\in L$'' is replaced by ``$f_1, f_2\in V$''.  Then $(V, \langle~,~\rangle_\lambda)$ forms a positive definite hermitian $(B, *, D)$-bimodule. Moreover, the assignment $(X, \lambda, \iota)\mapsto (V, \langle~,~\rangle_\lambda)$ induces a one-to-one correspondence between  the  isogenous classes of $\Q$-polarized supersingular $\calO$-abelian $k$-varieties of dimension $n$ and  the isometry classes of positive definite hermitian $(B, *, D)$-bimodules of $D$-rank $n$. In fact, we shall show in Proposition~\ref{prop:Rat-herm-bimod-unique} that  for each $n$ divisible by $2d$, there is exactly one isometry class of  positive definite hermitian $(B, *, D)$-bimodules of $D$-rank $n$. 
This in turn allows us to conclude in Corollary~\ref{cor:unique-isog-cls} that there is a unique isogenous class of $\Q$-polarized supersingular $\calO$-abelian varieties over $k$ of dimension $n$ for each such  $n$.   We start the proof of these results by the following simple lemma. 
Let $d\coloneqq [F:\Q]$ be the degree of $F$ over $\Q$.

\begin{lemma}\label{lem:inf-pl-ext}
   Let $V$ be a $(B, D)$-bimodule of $D$-rank $n$. Then up to isometry, there exists a unique positive definite hermitian $(B_{\R}, *, D_{\R})$-form on the $(B_{\R}, D_{\R})$-bimodule $V_{\R}$, where $B_\R\coloneqq B\otimes_{\Q}\R$, $D_{\R}\coloneqq D\otimes_{\Q}\R$ and $V_{\R}\coloneqq V\otimes_{\Q}\R$.
 \end{lemma}
 \begin{proof}
   From Lemma~\ref{lem:B-D-bimod}, $n$ is necessarily divisible by $2d$.
    Since $F$ is totally real, the infinite place $\infty$ of $\Q$ splits into $d$ real places $\infty_1, \cdots, \infty_d$ of $F$. Then
     \begin{equation}  \label{eq:decomp-B-1-infty}
     B_\R=B\otimes_{F}F\otimes_{\Q}\R=B\otimes_{F}(F_{\infty_1}\times\cdots\times F_{\infty_d})=B_{\infty_1}\times\cdots\times B_{\infty_d}.
    \end{equation}
    Consequently, it induces a decomposition $V_{\R}=V_{\infty_1}\oplus\cdots\oplus V_{\infty_d}$ with each $V_{\infty_i}\coloneqq V\otimes_F F_{\infty_i}$ a $(B_{\infty_i}, D_\R)$-bimodule of $D_\R$-rank $m\coloneqq n/d$. 
    For each $1\leq i\leq d$,  we have $V_{\infty_i}=e_i V_\R$, where $e_i\in B_{\infty_i}$ is the $i$-th primitive central idempotent $(0, \cdots, 0, 1, 0, \cdots, 0)$ of $B_\R$. 
   Let $\langle~,~\rangle$ be a hermitian $(B_\R, *, D_\R)$-form on $V_\R$. Since $e_i\in F_{\infty_i}$, we have $e_i^*=e_i$, and hence $\langle e_i V_\R, e_j V_\R \rangle=\langle V_\R, e_ie_j V_\R \rangle=0$ whenever $i\neq j$, that is, the decomposition~\eqref{eq:decomp-B-1-infty} induces an orthogonal decomposition
    \[
           V_\R= V_{\infty_1}\boxplus\cdots\boxplus V_{\infty_d}. 
    \]
    Therefore, the lemma would be proved if we show that $V_{\infty_i}$ admits a unique positive definite hermitian $(B_{\infty_i}, *, D_\R)$-form up to isometry for each $1\leq i\leq d$. 
    
    Since $(B, *)$ consists of a totally indefinite quaternion $F$-algebra $B$ together with a positive involution $*$, from \cite[\S21]{mumford:av} there exists an isomorphism $B_{\infty_i}\simeq\Mat_2(\R)$ carrying $*$ to the transpose $^t$ on matrices,  for  each $1\leq i\leq d$.
On the other hand,  $D_\R$ is isomorphic to the classical Hamilton quaternion $\R$-algebra $\mathbb{H}$.  
    Thus, it is enough to prove that there exists a unique positive definite hermitian $(\Mat_2(\R),\phantom{a}^t, \mathbb{H})$-bimodule of $\mathbb{H}$-rank $m\coloneqq n/d$ up to isometry.
  Let $\{e_{ij}\mid 1\leq i, j\leq 2\}$ be the standard basis of $\Mat_2(\R)$.
   Let $(W, \langle~,~\rangle)$ be a positive definite hermitian $(\Mat_2(\R), \phantom{a}^t, \mathbb{H})$-bimodule of $\mathbb{H}$-rank $m$.  Then $W=W_1\oplus W_2$, where $W_1\coloneqq e_{11}W$ and $W_2\coloneqq e_{22}W$. For any $x\in W_1$ and $y\in W_2$, since $e_{11}x=x$, $e_{22}y=y$, we have
   \[
            \langle x, y\rangle=\langle e_{11}x, e_{22}y\rangle=\langle x, e_{11}^t e_{22}y\rangle=0,
   \]
   that is, $W=W_1\boxplus W_2$. It is easy to see that the map $W_1\to W_2$ given by $x\mapsto e_{21}x$ is an isometry of hermitian right $\mathbb{H}$-modules.
  Now $(W_1, \langle~,~\rangle|_{W_1})$ is a positive definite hermitian right $\mathbb{H}$-module of rank $m/2$, which is well-known to be unique up to isometry \cite[\S4.2]{Shimura1963-AltHermForms}. Then the desired result follows from the Morita equivalence. 
 \end{proof}

\begin{lemma}
  Let $(V, \langle~,~\rangle)$ be a hermitian $(B, *, D)$-bimodule of $D$-rank $n$. Let $G$ be the unitary group of $(V, \langle~,~\rangle)$, that is, the linear algebraic group over $\Q$ such that
\begin{equation}\label{eq:defn-U}
    G(\Q)=\{ g\in\Aut_{(B, D)}(V)\mid \langle gx, gy \rangle=\langle x, y \rangle, \forall\, x, y\in V \}.
\end{equation}
Then we have the following isomorphism of algebraic $\overline{\Q}$-groups:
    \begin{equation}\label{eq:U-otimes-bar-Q}
        G\otimes\overline{\Q}\simeq \Sp_{n/d, \overline{\Q}}\times\cdots\times\Sp_{n/d, \overline{\Q}}
      \quad (d\coloneqq [F:\Q]\text{ copies}).
    \end{equation}
    In particular, $G$ is a simply connected semisimple algebraic $\Q$-group.
\end{lemma}
\begin{proof}
    We fix isomorphisms $D\otimes_{\Q}\overline{\Q}\simeq\Mat_2(\Q)$ and 
\begin{equation}\label{eq:B-1-otimes-bar-Q}
B\otimes_{\Q}\overline{\Q}=B\otimes_{F}F\otimes_{\Q}\overline{\Q}\simeq\Mat_2({\overline{\Q}})\times\cdots\times\Mat_2(\overline{\Q})\quad (d\text{ copies})
\end{equation}
such that the involution on every factor $\Mat_2(\overline{\Q})$ induced by $*$ is exactly the transpose $^t$ for matrices (see \cite[Proposition~8.3]{MR2192012}).  
Similar as in the proof of Lemma~\ref{lem:inf-pl-ext}, the decomposition \eqref{eq:B-1-otimes-bar-Q} induces an orthogonal decomposition of $V\otimes_{\Q}\overline{\Q}$ as follows
\[
       V\otimes_{\Q}\overline{\Q}=V_1\boxplus\cdots\boxplus V_d,
\]
where each $(V_i, \langle~,~\rangle|_{V_i})$ is a hermitian $(\Mat_2(\overline{\Q}),\, ^t, \Mat_2(\overline{\Q}))$-bimodule of $\Mat_2(\overline{\Q})$-rank $n/d$. Denote the unitary $\overline{\Q}$-group of $(V_i, \langle~,~\rangle|_{V_i})$ by $G_i$ for each $1\leq i\leq d$. Then
\[
         G\otimes\overline{\Q}=G_1\times\cdots\times G_d.
\] 
Let  $\{e_{ij}\mid 1\leq i, j\leq 2\}$ be the standard basis of $\Mat_2(\overline{\Q})$.
For each $1\leq i\leq d$, we put $W_i\coloneqq V_ie_{11}$. 
As in Section~\ref{subsec:self-dual-latt-split}, there exists a non-degenerate alternating form $\psi_i: W_i\times W_i\to \overline{\Q}$ on the $\overline{\Q}$-vector space $W_i$ such that
\[
        \langle x, y\rangle=\psi_i(x, y)e_{21}, \quad\forall x, y\in W_i.
\]
Then $(W_i, \psi_i)$ is a $\overline{\Q}$-valued skew-hermitian $(\Mat_2(\overline{\Q}),\, ^t)$-module of $\Mat_2(\overline{\Q})$-rank $n/2d$ in the sense of \cite[Definition~2.7]{terakado-xue-yu:2023}. Further, by \cite[\S2.3]{terakado-xue-yu:2023},  for each $1\leq i\leq d$, $(e_{11}W_i, \psi_i|_{e_{11}W_i})$ is a symplectic $\overline{\Q}$-space of dimension $n/d$, and we have an isomorphism $G_i\simeq \Sp_{n/d, \overline{\Q}}$ by the Morita equivalence. Then the desired isomorphism \eqref{eq:U-otimes-bar-Q} follows.
\end{proof}

 \begin{prop}\label{prop:Rat-herm-bimod-unique}
    For any positive integer $n$ divisible by $2d$, there exists a unique positive definite hermitian $(B, *, D)$-bimodule of $D$-rank $n$ up to isometry.
\end{prop}
\begin{proof}
 We first prove the existence. Fix a $(B, D)$-bimodule $V\coloneqq (D^{\oplus n}, \varphi)$ of $D$-rank $n$ given by an embedding $\varphi: B\hookrightarrow\Mat_n(D)$. With respect to the standard right $D$-basis of $V$,  a possibly degenerate hermitian $(B, *, D)$-form on $V$ is defined by a matrix $g\in \Mat_n(D)$ satisfying
    \begin{equation}\label{eq:e-511}
        g=\overline{g}^t,\quad \text{and}\quad {\overline{\varphi(\alpha)}}^tg=g\varphi(\alpha^*),\quad \forall \alpha\in B.
    \end{equation}
    In particular, the set of matrices $g$ satisfying \eqref{eq:e-511} forms a $\Q$-subspace $\scrH$ of $\Mat_n(D)$, which can also be interpreted as the solution space of a suitable homogeneous  $\Q$-linear system reformulated from \eqref{eq:e-511} by fixing $\Q$-bases of $B$ and $D$ (and in turn of $\Mat_n(D)$) respectively.  Extending the base field to $\R$, we immediately see that  $\scrH_\R\coloneqq \scrH\otimes_\Q\R$ consists precisely of all matrices $g\in\Mat_n(D_{\R})$ satisfying condition \eqref{eq:e-511}, or equivalently, satisfying the same linear system of equations.  Now Lemma~\ref{lem:inf-pl-ext} implies that $\scrH_\R^+\coloneqq\{g\in \scrH_\R\mid g>0\}$ forms a \emph{nonempty open subset} of $\scrH_\R$ with respect to the usual topology, where  $g>0$ stands for $g$ being positive definite. Since $\scrH$ is dense in $\scrH_\R$, there exists a matrix $g\in \scrH\cap \scrH_\R^+$, which defines a positive definite hermitian $(B, *, D)$-form on $V$.    
    
     For the uniqueness, let $(V, \langle~,~\rangle)$ be a positive definite hermitian $(B, *, D)$-bimodule of $D$-rank $n$ with unitary $\Q$-group $G$.
    From \cite[Theorems~6.4 and 6.6]{Platonov-Rapinchuk}, $H^1(\Q_{\ell}, G)=\{1\}$ for every prime $\ell$, and the map
    \[
           H^1(\Q, G)\to  \prod_{v\leq \infty} H^1(\Q_{v}, G)=H^1(\R, G)
    \]
    is injective. Therefore, by the local-global principle \cite[\S 6.6, p.~347]{Platonov-Rapinchuk}, we are reduced to proving that a positive definite hermitian $(B_{\R}, *, D_{\R})$-form on $V_\R\coloneqq V\otimes_{\Q}\R$ is unique up to isometry, which has already been done in Lemma~\ref{lem:inf-pl-ext}.
\end{proof}

\begin{cor}\label{cor:unique-isog-cls}
   For any positive integer $n$ divisible by $2d$, there is exactly one isogenous class of $\Q$-polarized supersingular $\calO$-abelian varieties over $k$ of dimension $n$.
\end{cor}

\subsection{Principally polarized superspecial $\calO$-abelian varieties} \label{subsec:p-ssp-O-1-ab-var}
 In this subsection, we (re)state and  prove the main theorem of this paper. In light of Lemma~\ref{lem:relat-with-self-dual-latt}, we proceed in the language of hermitian $(\calO, *, \calR)$-bilattices and then translate the results into the context of principally polarized superspecial $\calO$-abelian varieties.  As usual, the classification of global positive definite 
hermitian $(\calO, *, \calR)$-bilattices starts with that of the local ones, that is, the corresponding $\ell$-adic completions at all primes $\ell$, as the completion at the archimedean place has already been treated in  Lemma~\ref{lem:inf-pl-ext}. 

\begin{lemma}\label{lem:rank-of-summand}
Let $L$ be an $(\calO, \calR)$-bilattice of $\calR$-rank $n$, and $w$ be a finite place of $F$ lying above a prime number $\ell\in \mathbb{N}$. Then  the $w$-adic completion $L_{w}\coloneqq L\otimes_{O_F}O_{F_w}$ is an $(\calO_w, \calR_\ell)$-bilattice
that is both finite free over $\calO_{w}$ of rank $n/d$ and also finite free over $\calR_{\ell}$ of rank $\frac{n}{d} [F_{w}:\Q_{\ell}]$. 
\end{lemma}
\begin{proof}
      Clearly, the ambient space $L\otimes_{\Z}\Q$ of $L$ is a $(B, D)$-bimodule, or equivalently, a $(B, D\otimes_{\Q}F)$-bimodule. From Lemma~\ref{lem:emb-exist-neces-cond},  $L\otimes_\Z\Q$ is a free left  $B$-module of rank $n/d$, which implies that $L_w$ is free of the same rank over the maximal $O_{F_w}$-order $\calO_w$ in $B_w$. From the maximality of $\calR_\ell$ and an easy rank calculation, it follows that $L_w$ is a free right $\calR_\ell$-module of rank $\frac{n}{d} [F_{w}:\Q_{\ell}]$. 
  \end{proof}    

From \cite[\S 21]{mumford:av}, since $*$ is a positive involution on the totally indefinite quaternion $F$-algebra $B$, there exists $\gamma\in B^\times$ such that 
\begin{equation}\label{eq:positive-involution}
    \gamma^2\in F_{<0}^\times \quad\text{and}\quad
    \alpha^*=\gamma \ol{\alpha} \gamma^{-1},\quad\forall \alpha\in B.
\end{equation}
Here $F_{<0}^\times$ denotes the subset of totally negative elements of $F$. In particular, $*$ is an orthogonal involution on $B$.
 
\begin{lemma}\label{lem:self-dual-ell-neq-p}
     Let $\ell\neq p$ be a prime number.
     For an $(\calO, \calR)$-bilattice $L$ of $\calR$-rank $n$, the following statements are equivalent:
     \begin{enumerate}[(i)]
         \item there exists a perfect hermitian $(\calO_{\ell}, *, \calR_{\ell})$-pairing $\langle~,~\rangle_\ell: L_\ell\times L_\ell\to \calR_\ell$;
         \item $n/2d$ is even or $\ord_{B_{w}}(\gamma)$ is odd for every finite place $w$ above $\ell$ ramified in $B$.
     \end{enumerate}
      Such a perfect pairing $\langle~,~\rangle_\ell$ on $L_{\ell}$ is unique up to isometry if it exists.
\end{lemma}

\begin{proof}
   Let $w_1, \cdots, w_{s}$ be the finite places of $F$ above $\ell$.
   Then the decomposition
          \begin{equation*}
\calO_{\ell}\coloneqq\calO\otimes_{\Z}\Z_{\ell}=\calO\otimes_{O_F}O_F\otimes_{\Z}\Z_{\ell}
       =\calO\otimes_{O_F} (O_{F_{w_1}}\times\cdots\times O_{F_{w_{s}}})
       =\calO_{w_1}\times\cdots\times \calO_{w_{s}}
          \end{equation*}
 induces the following decomposition of the $(\calO_\ell, \calR_\ell)$-bilattice $L_\ell$:
\begin{equation*}
    L_{\ell}\coloneqq L\otimes_{\Z}\Z_{\ell}=L\otimes_{O_F}O_F\otimes_{\Z}\Z_{\ell}=L\otimes_{O_F}(O_{F_{w_1}}\times\cdots\times O_{F_{w_{s}}})=L_{w_1}\oplus
    \cdots\oplus L_{w_s},
\end{equation*}
where  $L_{w_i}\coloneqq L\otimes_{O_F}O_{F_{w_i}}$ is an $(\calO_{w_i}, \calR_{\ell})$-bilattice for each $1\leq i\leq s$.
As in the proof of Lemma~\ref{lem:inf-pl-ext}, to give a perfect hermitian $(\calO_{\ell}, *, \calR_{\ell})$-pairing on the $(\calO_{\ell}, \calR_{\ell})$-bilattice $L_{\ell}$, it is the same as giving a perfect hermitian $(\calO_{w_i}, *, \calR_{\ell})$-pairing on the $(\calO_{w_i}, \calR_{\ell})$-bilattice $L_{w_i}$ for each $1\leq i\leq s$. 
Now $\calR_\ell\cong\Mat_2(\Z_\ell)$ since $\ell\neq p$,  and the $\calR_\ell$-rank of each $L_{w}$ with $w\in\{w_1, \cdots, w_s\}$ is given by Lemma~\ref{lem:rank-of-summand}.  We apply Proposition~\ref{prop:self-dual-split-case} to obtain our lemma.
\end{proof}

      For the remainder of this section, we assume that $\calR_F\coloneqq\calR\otimes_{\Z}O_F$ is a maximal order in $D_F\coloneqq D\otimes_{\Q}F$, or equivalently, every place $v$ of $F$ above $p$ is unramified with odd residue degree; see Lemma~\ref{lem:max-order-under-base-change}. In particular, this assumption implies that the quaternion $F$-algebra $D_F$ is ramified at every place $v$ of $F$ above $p$, and the unique maximal order of its $v$-adic completion $D_p\otimes_{\Q_p} F_v$ coincides with $\calR_p\otimes_{\Z_p}O_{F_v}$ for every such $v$.    
Let $V$ be a $(B, D)$-bimodule, and $V_v\coloneqq V\otimes_{F}F_v$ be its $v$-adic completion, which is    naturally a $(B_v, D_p)$-bimodule, or equivalently,   a $(B_v,  D_p\otimes_{\Q_p}F_v)$-bimodule. We show that every hermitian $(B_v, *, D_p)$-form on $V_v$ lifts to a hermitian $(B_v, *, D_p\otimes_{\Q_p} F_{v})$-form. For this purpose, observe that 
    the trace map $\Tr_{F_v/\Q_p}: F_v\to \Q_p$ induces a $(D_p, D_p)$-bilinear map
\begin{equation*}
    \Tr\coloneqq\id\otimes\Tr_{F_v/\Q_p}: D_p\otimes_{\Q_p}F_v \to D_p.
\end{equation*}
Clearly, $\overline{\Tr(\alpha)}=\Tr(\overline{\alpha})$ for every $\alpha\in D_p\otimes_{\Q_p}F_v$, but it should be cautioned that 
$\Tr(\alpha\beta)\neq\Tr(\beta\alpha)$ in general. 
   
   \begin{lemma}\label{lem:lift-pairing}
  Keep the notation and assumption as above and fix a place $v$ of $F$ above $p$.  For every hermitian $(B_v, *, D_p)$-form $\langle~,~\rangle_v: V_v\times V_v\to D_p$ on $V_v$, there exists  a hermitian $(B_v, *, D_p\otimes_{\Q_p} F_{v})$-form $(~,~)_v:V_v\times V_v\to  D_p\otimes_{\Q_p} F_{v}$ such that
          \begin{equation*}
             \langle x, y\rangle_v=\Tr(x, y)_v,\quad\forall\, x, y\in V_v.
          \end{equation*}
      Moreover, an $(\calO_v, \calR_p)$-bilattice $L_v\subset V_v$ is self-dual with respect to $\langle~,~\rangle_v$ if and only if it is self-dual with respect to $(~,~)_v$ as an $(\calO_v, \calR_p\otimes_{\Z_p}O_{F_v})$-bilattice.
    \end{lemma}   
  
   \begin{proof}
  For brevity, let us put $H\coloneqq D_F$ so that  $H_v=D_p\otimes_{\Q_p} F_{v}$.
The trace map $\Tr_{F_v/\Q_p}: F_v\to \Q_p$ induces an isomorphism
\begin{equation}\label{eq:iso-induced-by-Tr-F-Q}
    F_v\to \Hom_{\Q_p}(F_v, \Q_p), \quad a\mapsto (b\mapsto \Tr_{F_v/\Q_p}(ab)).
\end{equation}
Let $\Hom_{\mathrm{Mod}-D_p}(H_v, D_p)$ denote the set of right $D_p$-linear maps $H_v\to D_p$, which has a natural $(D_p, D_p)$-bimodule structure induced by the left $D_p$-module structures of $D_p$ and $H_v$. Applying the  tensor functor $D_p\otimes_{\Q_p}\negmedspace-$ to the map in \eqref{eq:iso-induced-by-Tr-F-Q}  produces the following isomorphism of $(D_p, D_p)$-modules:
\begin{equation}\label{eq:iso-induced-by-Tr}
    H_v\to \Hom_{\mathrm{Mod}-D_p}(H_v, D_p), \quad \alpha\mapsto (\beta\mapsto \Tr(\alpha\beta)).
\end{equation}
For fixed $x, y\in V$, consider the right $D_p$-linear map
    \begin{equation*}
        H_v\to D_p,\quad \beta\mapsto \langle x, y\beta\rangle_v.
    \end{equation*}
From the isomorphism \eqref{eq:iso-induced-by-Tr}, there exists a unique $(x, y)_v\in H_v$ such that
    \begin{equation*}
        \langle x, y\beta\rangle_v=\Tr( (x, y)_v\beta), \quad\forall\, \beta\in H_v.
    \end{equation*}
 It is routine to show that $(~,~)_v$ defines a hermitian $(B_v, *, H_v)$-form on $V_v$. We only check $( x, y)_v=\overline{( y, x)}_v$ for all $x, y\in V_v$, and leave the other properties to the interested reader. Take any $b\in F_v$. Since $xb=bx$, $yb=by$ and $b^*=b$, we have
    \[     \langle x, yb\rangle_v=\overline{\langle yb, x\rangle}_v
                =\overline{\langle by, x\rangle}_v
                =\overline{\langle y, bx\rangle}_v
                =\overline{\langle y, xb\rangle}_v.
     \]
 On the other hand, by the definition of $(~,~)_v$, we have $\langle x, yb\rangle_v=\Tr((x, y)_vb)$, and hence
     \[
         \overline{\langle y, xb\rangle}_v=\overline{\Tr(( y, x)_vb)}
         =\Tr(\overline{ (y, x)}_vb).
     \]
It follows that $\Tr((x, y)_vb)=\Tr(\overline{( y, x)}_vb)$ for all $b\in F_v$. Since $\Tr: H_v\to D_p$ is right $D_p$-linear, this further implies that $\Tr((x, y)_v\beta)=\Tr(\overline{( y, x)}_v\beta)$ for all $\beta\in H_v$. From the isomorphism~\eqref{eq:iso-induced-by-Tr}, we conclude that $( x, y)_v=\overline{( y, x)}_v$ as desired. 

Lastly, by our assumption,  $F_v/\Q_p$ is unramified so that its different ideal equals $O_{F_v}$. Hence given $\alpha\in H_v=D_p\otimes_{\Q_p}F_v$, we have $\Tr(\alpha O_{F_v})\subseteq \calR_p$ if and only if  $\alpha\in \calR_p\otimes_{\Z_p}O_{F_v}$. In particular, for any $x\in V_v$, the following two conditions are equivalent: 
\[\langle x, L_v\rangle_v=\Tr( x, L_v)_v\subseteq \calR_p \quad \Longleftrightarrow\quad ( x, L_v)_v\subseteq \calR_p\otimes_{\Z_p}O_{F_v}.\]
This proves the second statement of our lemma; see~\eqref{eq:defn-of-dual-bilatt}.
   \end{proof}

     As in Section~\ref{subsec:ssp-O-1-ab-var}, let $\Sigma$ be the set of places of $F$ above $p$ ramified in $B$. Equivalently, $\Sigma$ is exactly the set of places of $F$ ramified in both $B$ and $D_F$.
     For each $v\in\Sigma$, we fix an identification $ \calR\otimes_{\Z}O_{F_v}\simeq\calO_v$ once and for all so that the structural invariant $\ulm^{(v)}(L)$ at $v$ of every $(\calO, \calR)$-bilattice $L$ is defined; see Remark~\ref{rem:distinct-indentify}. If $\ulm^{(v)}(L)=(r_1^{(v)}, r_2^{(v)}, t_1^{(v)}, t_2^{(v)})\in\Z_{\geq 0}^4$ and $\rank_{\calR}L=n$, then by~\eqref{eq:rank-comp} or \eqref{eq:rank-comp-1} we have 
     \begin{equation}\label{eq:e-521}
         r_1^{(v)}+r_2^{(v)}+2t_1^{(v)}+2t_2^{(v)}=n/d.
     \end{equation}

     \begin{lemma}\label{lem:self-dual-p}
     Let $\gamma\in B^\times$ be an element defining the positive involution $*$ on $B$ as in \eqref{eq:positive-involution}. 
    Let $L$ be an $(\calO, \calR)$-bilattice with structural invariants      
     \[\underline{m}^{(v)}(L)\coloneqq (r_1^{(v)}, r_2^{(v)}, t_1^{(v)}, t_2^{(v)}),\qquad\forall\, v\in\Sigma.\] 
    Then there exists a perfect hermitian $(\calO_{p}, *, \calR_{p})$-pairing on $L_p$ if and only if the following condition holds:
        \begin{equation}\label{eq:cond-for-quadruple}
\forall v\in \Sigma, \quad  \begin{dcases*}
        r_1^{(v)}=r_2^{(v)},  &  if $\ord_{B_{v}}(\gamma)$  is odd;\\
        2|r_1^{(v)}, 2|r_2^{(v)}, \text{ and } t_1^{(v)}=t_2^{(v)}, & if   $\ord_{B_{v}}(\gamma)$ is even. 
    \end{dcases*} 
    \end{equation}
     Such a perfect pairing on $L_p$ is unique up to isometry if it exists.
 \end{lemma}

\begin{remark}\label{rem:indenp}
    From Remark~\ref{rem:distinct-indentify}, condition~\eqref{eq:cond-for-quadruple} is in fact independent of  the choice of identifications $\calR\otimes_{\Z}O_{F_{v}}\simeq\calO_v$ for all $v\in\Sigma$.
\end{remark}

\begin{proof}
    As in the proof of Lemma~\ref{lem:self-dual-ell-neq-p}, to give a perfect hermitian $(\calO_{p}, *, \calR_{p})$-pairing on $L_p$, it is the same as  giving a perfect hermitian $(\calO_{v}, *, \calR_{p})$-pairing on $L_v$ for each place $v$ of $F$ above $p$.  In light of Lemma~\ref{lem:lift-pairing}, the current lemma follows immediately from Theorems~\ref{thm:self-dual-latt-odd}, \ref{thm:self-dual-latt-even} and Proposition~\ref{prop:self-dual-B-1-split}.
\end{proof}
 
The main theorem of this section is as follows.

\begin{thm}\label{thm:genus-char-2} 
    Keep the assumption that $\calR_F=\calR\otimes_{\Z}O_F$ is a maximal order in $D_F= D\otimes_\Q F$.
    Let $\Sigma$ be the set of places of $F$ above $p$ ramified in $B$, and $\Xi$ be the set of finite places of $F$ coprime to $p$ and ramified in $B$.
    Then the following statements hold true.
    \begin{enumerate}[(i)]
        \item If there exists a self-dual hermitian $(\calO, *, \calR)$-bilattice of $\calR$-rank $n$, then the following condition necessarily holds:
    \begin{equation}\label{eq:cond-for-S-1}
    n/2d\text{ is even or }\ord_{B_{w}}(\gamma)\text{ is odd for every }w\in \Xi.
    \end{equation}
        \item Conversely, let $n$ be a fixed positive integer divisible by $2d$. Suppose that condition \eqref{eq:cond-for-S-1} holds. For each $v\in\Sigma$, fix a quadruple
\[
      (r_1^{(v)}, r_2^{(v)}, t_1^{(v)}, t_2^{(v)})\in\Z_{\geq 0}^4 
    \]
    satisfying~\eqref{eq:e-521}.
    Then there exists a self-dual positive definite hermitian $(\calO, *, \calR)$-bilattice $(L, \langle~,~\rangle)$ of $\calR$-rank $n$ such that $\underline{m}^{(v)}(L)$ coincides with $(r_1^{(v)}, r_2^{(v)}, t_1^{(v)}, t_2^{(v)})$ for every $v\in\Sigma$
    if and only if condition \eqref{eq:cond-for-quadruple}  holds.
    \item Fix both $n$ and  the quadruples $(r_1^{(v)}, r_2^{(v)}, t_1^{(v)}, t_2^{(v)})$ for all $v\in\Sigma$ as in part (ii) and 
    suppose that both conditions \eqref{eq:cond-for-quadruple} and \eqref{eq:cond-for-S-1}  hold. Denote the narrow class number of $F$ by $h^{+}(F)$. 
    If further $n>2d$, then there are  exactly $h^{+}(F)$ isomorphism classes of $(\calO, \calR)$-bilattices $L'$ of $\calR$-rank $n$ with $\ulm^{(v)}(L')=(r_1^{(v)}, r_2^{(v)}, t_1^{(v)}, t_2^{(v)})$ for all $v\in\Sigma$, among which exactly one of them admits a perfect positive definite hermitian $(\calO, *, \calR)$-pairing. 
    \end{enumerate}
    
\end{thm}

 \begin{proof}
(i) (ii) Statement (i) and the `only if' part of statement (ii) follow immediately from Lemmas~\ref{lem:self-dual-ell-neq-p} and~\ref{lem:self-dual-p}. Indeed, let $(L, \langle~,~\rangle)$ be a self-dual hermitian $(\calO, *, \calR)$-bilattice of $\calR$-rank $n$. Then for each prime $\ell$ (including $\ell=p$), the $\ell$-adic completion $(L_\ell, \langle~,~\rangle_\ell)$ of $(L, \langle~,~\rangle)$ is a self-dual hermitian $(\calO_{\ell}, *, \calR_{\ell})$-bilattice, so Lemmas~\ref{lem:self-dual-ell-neq-p} and \ref{lem:self-dual-p} apply.

 We prove the `if' part of statement (ii).
 First, by the local-global principle for lattices \cite[Theorem~9.4.9]{voight-quat-book}, there exists an $(\calO, \calR)$-bilattice $M$ of $\calR$-rank $n$ such that $\underline{m}^{(v)}(M)=(r_1^{(v)}, r_2^{(v)}, t_1^{(v)}, t_2^{(v)})$ for all $v\in\Sigma$.
 In view of Lemma~\ref{lem:B-D-bimod}, we may assume that $M=(\calR^{\oplus n},\varphi)$ for some embedding $\varphi:\calO\hookrightarrow\Mat_n(\calR)$. Let $V\coloneqq M\otimes_{\Z}\Q$ be the ambient $(B, D)$-bimodule of $M$. The proof is divided into two steps. In the first step, we construct a positive definite hermitian $(B, *, D)$-form $\langle~,~\rangle: V\times V\to D$ with certain local property at $p$. In the second step, we modify the bilattice $M$ in $V$ to obtain the desired $(\calO, \calR)$-bilattice $L$ such that $L$ is  self-dual in $(V,\langle~,~\rangle)$ with $\underline{m}^{(v)}(L)=(r_1^{(v)}, r_2^{(v)}, t_1^{(v)}, t_2^{(v)})$ for all $v\in\Sigma$.

     \emph{Step 1.} We show that there exists a matrix $g\in\Mat_n(D)$ satisfying both of the following conditions: 
     \begin{align}
        g&=\overline{g}^t,\quad \text{and}\quad {\overline{\varphi(\alpha)}}^tg=g\varphi(\alpha^*),\quad \forall \alpha\in B;\label{eq:herm-bimod-cond}\\
        g&>0,\quad\text{ and}\quad g\in \GL_n(\calR_{p})\label{eq:add-cond},
     \end{align} 
    where we write $g>0$ to indicate that $g$ is positive definite. 
     As in the proof of Proposition~\ref{prop:Rat-herm-bimod-unique},
     let $\scrH\subseteq\Mat_n(D)$ be the $\Q$-vector space consisting of all matrices $g\in\Mat_n(D)$ satisfying  \eqref{eq:herm-bimod-cond}.
     Extending the base field to $\R$ and $\Q_p$ respectively, we make the following two observations:
     \begin{enumerate}[label=\arabic*)]
          \item The $\R$-vector space $\scrH_\R\coloneqq\scrH\otimes_{\Q}\R\subseteq\Mat_n(D_{\R})$ consists of all matrices $g\in\Mat_n(D_{\R})$ satisfying  \eqref{eq:herm-bimod-cond}, and $\scrH_{\R}^{+}\coloneqq\{g\in \scrH_\R\mid g>0 \}$ is an open subset of $\scrH_\R$ with respect to the usual topology.
          \item The $\Q_p$-vector space $\scrH_p\coloneqq \scrH\otimes_{\Q}\Q_p\subseteq\Mat_n(D_{p})$ consists of all matrices $g\in\Mat_n(D_{p})$ satisfying  \eqref{eq:herm-bimod-cond}, and $\scrH_{p}^{+}\coloneqq\{g\in \scrH_p\mid g\in\GL_n(\calR_{p}) \}$ is an open subset of $\scrH_p$ with respect to the $p$-adic topology.
     \end{enumerate}  
    As indicated in the proof of Proposition~\ref{prop:Rat-herm-bimod-unique}, we know from Lemma~\ref{lem:inf-pl-ext} that $\scrH_\R^{+}$ is nonempty.
    On the other hand, by Lemma~\ref{lem:self-dual-p} and our assumption on $\ulm^{(v)}(M)$ for all $v\in \Sigma$, there exists a perfect hermitian $(\calO_{p}, *, \calR_{p})$-pairing on $M_p$, which implies that $\scrH_p^{+}$ is nonempty.
     From the weak approximation theorem, the natural map $\scrH\hookrightarrow \scrH_\R\times \scrH_p$
     has dense image in the codomain, so there exists a matrix $g\in \scrH$ whose image lies in $\scrH_\R^{+}\times \scrH_p^{+}$. By our construction, the matrix $g$ satisfies both  conditions \eqref{eq:herm-bimod-cond} and \eqref{eq:add-cond}, and hence defines a positive definite hermitian $(B, *, D)$-form $\langle~,~\rangle$ on $V$ such that $M_p$ is a self-dual $(\calO_{p}, \calR_{p})$-bilattice in $(V_p,\langle~,~\rangle_p)$.

     \emph{Step 2.} Let $g\in\Mat_n(D)$ be a matrix satisfying both conditions~\eqref{eq:herm-bimod-cond} and \eqref{eq:add-cond}, and let $\langle~,~\rangle$ be its associated hermitian form on $V$. Note that $g\in\GL_n(\calR_{\ell})$ for almost all primes $\ell$. Let $T$ be the (finite) set of primes $\ell$ such that $g\notin\GL_n(\calR_{\ell})$. Since condition \eqref{eq:cond-for-S-1} holds by hypothesis, it follows from Proposition~\ref{prop:self-dual-split-case} that for each prime $\ell\in T$ there exists a self-dual hermitian $(\calO_{\ell},\calR_{\ell})$-bilattice $N_{\ell}$ in $(V_{\ell},\langle~,~\rangle_{\ell})$ (which is unique up to isometry). 
     By the local-global principle for lattices, there exists a unique $(\calO,\calR)$-bilattice $L$ in $V$ such that its $\ell$-adic completions satisfy
    \[
         L_{\ell}=\begin{cases}
             M_{\ell}, & \text{if } \ell\notin T, \\
             N_{\ell}, & \text{if } \ell\in T.
         \end{cases}
     \]
     By construction, the bilattice $L$ is self-dual with respect to $\langle~,~\rangle$. Moreover, $p\notin T$, so the $p$-adic completion of $L$ coincides with $M_p$, which shows that $\ulm^{(v)}(L)=\ulm^{(v)}(M)=(r_1^{(v)}, r_2^{(v)}, t_1^{(v)}, t_2^{(v)})$ for all $v\in\Sigma$ as desired.

     (iii) First, from the `if' part of (ii) proved above, there exists a self-dual positive definite hermitian $(\calO, *, \calR)$-bilattice $(L, \langle~,~\rangle)$ of $\calR$-rank $n$ with $\ulm^{(v)}(L)=(r_1^{(v)}, r_2^{(v)}, t_1^{(v)}, t_2^{(v)})$ for all $v\in\Sigma$. Let $L'$ be another $(\calO, \calR)$-bilattice of $\calR$-rank $n$ with the same structural invariants at all $v\in\Sigma$ so that $L'$ lies in the genus of $L$ by Lemma~\ref{lem:bilatt-genus}.
     We have to show that if $L'$ admits a perfect positive definite hermitian $(\calO, *, \calR)$-pairing, then it is isomorphic to $L$. In light of Lemma~\ref{lem:B-D-bimod}, we may assume that $L'$ is contained in $V\coloneqq L\otimes_{\Z}\Q$. Put 
     \[
               A\coloneqq\End_{(B, D)}(V)=\End_{(B, D_F)}(V),\quad \Lambda\coloneqq\End_{(\calO, \calR)}(L)=\End_{(\calO, \calR_F)}(L),
     \]
     where $D_F\coloneqq D\otimes_{\Q}F$ and $\calR_F\coloneqq\calR\otimes_{\Z}O_F$. 
     According to \cite[Theorem~31.18]{curtis-reiner:1}, there is a one-to-one correspondence between the set of isomorphism classes of $(\calO, \calR)$-bilattices in the genus of $L$ and the double coset space $ A^\times\backslash \widehat{A}^\times/ \widehat{\Lambda}^\times$; also see Section~\ref{subsec:glob-latt}. In particular, $L'=\widehat{g}L$ for some $\widehat{g}\in\widehat{A}^\times$. More precisely, if $\widehat{g}=(g_{w})_w$ with $w$ running through all finite places of $F$, then $L'$ is the unique $(\calO, \calR)$-bilattice in $V$ such that $L_{w}'=g_{w}L_w$ for every  $w$. Since $n/d\geq 4$ by hypothesis, it follows from Proposition~\ref{prop:isom-cls-latt} that there is a bijection
     \begin{equation}\label{eq:reduced-norm-map-1}
      \Nrd: A^\times\backslash \widehat{A}^\times/ \widehat{\Lambda}^\times\to F_{>0}^{\times}\backslash\widehat{F}^\times/\widehat{O}_F^{\times},
     \end{equation}
     where $F_{>0}^{\times}$ denotes the group of totally positive  elements of $F$. This shows that there are exactly $h^{+}(F)$  isomorphism classes of $(\calO, \calR)$-bilattices $L'$ of $\calR$-rank $n$ with the same structural invariants as $L$; see also Example~\ref{eg:F}. 
     To prove $L'$ is isomorphic to $L$ when it admits a  perfect positive definite hermitian $(\calO, *, \calR)$-pairing, it is enough to prove that $\Nrd(\widehat{g})$ represents the identity element of the narrow class group $F_{>0}^{\times}\backslash\widehat{F}^\times/\widehat{O}_F^{\times}$ of $F$ in this case. 
     
     In view of Proposition~\ref{prop:Rat-herm-bimod-unique}, the existence of perfect pairing on $L'$ is equivalent to the existence of $\alpha\in A^\times$ such that the $(\calO, \calR)$-bilattice $\alpha L'$ is self-dual with respect to $\langle~,~\rangle$.  By \cite[Proposition~4.1]{book-of-involution}, there exists a unique involution $\dagger$ on $A$, called the \emph{adjoint involution} of $\langle~,~\rangle$, such that 
     \[
        \langle gx, y\rangle=\langle x, g^{\dagger}y\rangle,\quad\forall x, y\in V,\,\forall g\in A.
     \]
     Moreover, $\Nrd(g^\dagger)=\Nrd(g)$ for all $g\in A$ since the involution  $g\mapsto g^\dagger$ is an anti-automorphism of the central simple $F$-algebra $A$; see~\eqref{eq:compute-C}. Let $\widehat{L}\coloneqq L\otimes_{O_F}\widehat{O}_F$ be the profinite completion of $L$.
     We compute the dual lattice of $\alpha L'$ as follows 
     \begin{align*}
         (\alpha L')^{\vee}=(\alpha\widehat{g}L)^{\vee}&=\{x\in V\mid \langle x, \alpha\widehat{g}L\rangle\subseteq  \calR\}\\
               &=\{x\in V\mid \langle (\alpha\widehat{g})^{\dagger}x, \widehat{L}\,\rangle\subseteq \widehat{\calR}\,\}\\
               &=\{x\in V\mid (\alpha\widehat{g})^{\dagger}x\in \widehat{L}^{\vee}=\widehat{L}\, \}\\
               &=((\alpha\widehat{g})^{\dagger})^{-1} L.
     \end{align*}
     Then $\alpha L'$ is self-dual in $(V, \langle~,~\rangle)$ if and only if $(\alpha\widehat{g})^{\dagger}\alpha\widehat{g}\in\widehat{\Lambda}^\times$.  Suppose that this is the case. 
     Since $\Nrd((\alpha\widehat{g})^{\dagger})=\Nrd(\alpha\widehat{g})$, we have $\Nrd(\alpha\widehat{g})^2\in\Nrd(\widehat{\Lambda}^\times)\subseteq\widehat{O}_F^{\times}$, so $\Nrd(\alpha\widehat{g})\in \widehat{O}_F^{\times}$. Note that $\Nrd(\widehat{g})=\Nrd(\alpha^{-1})\Nrd(\alpha\widehat{g})\in F_{>0}^\times\widehat{O}_F^{\times}$, so $\Nrd(\widehat{g})$ represents the identity element of $F_{>0}^{\times}\backslash\widehat{F}^\times/\widehat{O}_F^{\times}$ as desired. This finishes the proof of part (iii) and completes the proof of the theorem. 
     \end{proof}

By virtue of Lemmas~\ref{lem:emb-exist-neces-cond-2}, \ref{lem:pol-herm-pair} and \ref{lem:relat-with-self-dual-latt}, we apply Theorem~\ref{thm:genus-char-2} to obtain the following theorem on superspecial $\calO$-abelian varieties.

\begin{thm}\label{thm:genus-char}
     Keep the assumptions in Theorem~\ref{thm:genus-char-2}. Then  the following holds true.
     \begin{enumerate}[(i)]
         \item If there exists a principally polarized superspecial $\calO$-abelian variety over $k$ of dimension $n$, then condition \eqref{eq:cond-for-S-1} necessarily holds.
         \item Conversely, let $n$ be a fixed positive integer divisible by $2d$ and suppose that condition \eqref{eq:cond-for-S-1} holds for $n$. For each $v\in\Sigma$, fix a quadruple
    \[
         (r_1^{(v)}, r_2^{(v)}, t_1^{(v)}, t_2^{(v)})\in\Z_{\geq 0}^4 
    \]
      satisfying~\eqref{eq:e-521}.
     Then there exists a principally polarized $\calO$-abelian variety $(X,\lambda,\iota)$ over $k$ of dimension $n$ with $\ulm^{(v)}(X, \iota)=(r_1^{(v)}, r_2^{(v)}, t_1^{(v)}, t_2^{(v)})$ for all $v\in\Sigma$ if and only if condition \eqref{eq:cond-for-quadruple} holds.
    \item  Fix both $n$ and  the quadruples $(r_1^{(v)}, r_2^{(v)}, t_1^{(v)}, t_2^{(v)})$ for all $v\in\Sigma$ as in part (ii) and 
    suppose that both conditions \eqref{eq:cond-for-quadruple} and \eqref{eq:cond-for-S-1}  hold.
     If further $n>2d$, then there are exactly $h^+(F)$  isomorphism classes of (unpolarized) superspecial $\calO$-abelian $k$-varieties $(X', \iota')$ of dimension $n$ with $\ulm^{(v)}(X', \iota')=(r_1^{(v)}, r_2^{(v)}, t_1^{(v)}, t_2^{(v)})$ for all $v\in\Sigma$, among which exactly one of them is principally polarizable. 
     \end{enumerate}
\end{thm}
 
\begin{cor}\label{prop:ext-pol}
   Let $(X, \iota)$ be a superspecial $\calO$-abelian $k$-variety of dimension $n>2d$ with structural invariants 
   \[\underline{m}^{(v)}(X, \iota)=(r_1^{(v)}, r_2^{(v)}, t_1^{(v)}, t_2^{(v)}),\quad\forall\, v\in\Sigma.\]
   If $h^{+}(F)=1$, then $(X, \iota)$ is principally polarizable if and only if both conditions \eqref{eq:cond-for-S-1} and \eqref{eq:cond-for-quadruple} hold. 
\end{cor}

\section*{Acknowledgments}
The authors express their gratitude to Yiping Chen,  Yasuhiro Terakado and Chia-Fu Yu for stimulating discussions.  Particular thanks go to Terakado for bringing Ribet's work to our attention. 
Xue is partially supported by the National Natural Science Foundation of China grants No.~12331002 and No.~12271410.

\def\cprime{$'$}


\begin{thebibliography}{10}

\bibitem{MR1314422}
Maurice Auslander, Idun Reiten, and Sverre~O. Smal\o.
\newblock {\em Representation theory of {A}rtin algebras}, volume~36 of {\em
  Cambridge Studies in Advanced Mathematics}.
\newblock Cambridge University Press, Cambridge, 1995.

\bibitem{MR491773}
Ricardo Baeza.
\newblock {\em Quadratic forms over semilocal rings}, volume Vol. 655 of {\em
  Lecture Notes in Mathematics}.
\newblock Springer-Verlag, Berlin-New York, 1978.

\bibitem{centeleghe-stix-I}
Tommaso~Giorgio Centeleghe and Jakob Stix.
\newblock Categories of abelian varieties over finite fields, {I}: {A}belian
  varieties over {$\Bbb{F}_p$}.
\newblock {\em Algebra Number Theory}, 9(1):225--265, 2015.

\bibitem{Centeleghe-Stix-II}
Tommaso~Giorgio Centeleghe and Jakob Stix.
\newblock Categories of abelian varieties over finite fields {II}: abelian
  varieties over {$\Bbb F_q$} and {M}orita equivalence.
\newblock {\em Israel J. Math.}, 257(1):103--170, 2023.

\bibitem{curtis-reiner:2}
Charles~W. Curtis and Irving Reiner.
\newblock {\em Methods of representation theory. {V}ol. {II}}.
\newblock Pure and Applied Mathematics (New York). John Wiley \& Sons, Inc.,
  New York, 1987.
\newblock With applications to finite groups and orders, A Wiley-Interscience
  Publication.

\bibitem{curtis-reiner:1}
Charles~W. Curtis and Irving Reiner.
\newblock {\em Methods of representation theory. {V}ol. {I}}.
\newblock Wiley Classics Library. John Wiley \& Sons, Inc., New York, 1990.
\newblock With applications to finite groups and orders, Reprint of the 1981
  original, A Wiley-Interscience Publication.

\bibitem{deligne:ord}
Pierre Deligne.
\newblock Vari\'et\'es ab\'eliennes ordinaires sur un corps fini.
\newblock {\em Invent. Math.}, 8:238--243, 1969.

\bibitem{MR422290}
V.~G. Drinfeld.
\newblock Coverings of {$p$}-adic symmetric domains.
\newblock {\em Funkcional. Anal. i Prilo\v zen.}, 10(2):29--40, 1976.

\bibitem{Drozd-Kirichenko-Roiter-1967}
Ju.~A. Drozd, V.~V. Kiri\v{c}enko, and A.~V. Ro\u{\i}ter.
\newblock Hereditary and {B}ass orders.
\newblock {\em Izv. Akad. Nauk SSSR Ser. Mat.}, 31:1415--1436, 1967.

\bibitem{MR1648352}
Hiroaki Hijikata and Kenji Nishida.
\newblock When is {$\Lambda_1\otimes\Lambda_2$} hereditary?
\newblock {\em Osaka J. Math.}, 35(3):493--500, 1998.

\bibitem{Ibukiyama-Karemaker-Yu-2025}
Tomoyoshi Ibukiyama, Valentijn Karemaker, and Chia-Fu Yu.
\newblock When is a polarised abelian variety determined by its {$p$}-divisible
  group?
\newblock {\em Trans. Amer. Math. Soc. Ser. B}, 12:65--111, 2025.

\bibitem{Ibukiyama-Katsura-Oort-1986}
Tomoyoshi Ibukiyama, Toshiyuki Katsura, and Frans Oort.
\newblock Supersingular curves of genus two and class numbers.
\newblock {\em Compositio Math.}, 57(2):127--152, 1986.

\bibitem{Jacobowitz-HermForm}
Ronald Jacobowitz.
\newblock Hermitian forms over local fields.
\newblock {\em Amer. J. Math.}, 84:441--465, 1962.

\bibitem{Janusz-1979-JLMS}
Gerald~J. Janusz.
\newblock Tensor products of orders.
\newblock {\em J. London Math. Soc. (2)}, 20(2):186--192, 1979.

\bibitem{Poonen-et:av}
Bruce~W. Jordan, Allan~G. Keeton, Bjorn Poonen, Eric~M. Rains, Nicholas
  Shepherd-Barron, and John~T. Tate.
\newblock Abelian varieties isogenous to a power of an elliptic curve.
\newblock {\em Compos. Math.}, 154(5):934--959, 2018.

\bibitem{Knebusch-1977}
Manfred Knebusch.
\newblock Symmetric bilinear forms over algebraic varieties.
\newblock In {\em Conference on {Q}uadratic {F}orms---1976 ({P}roc. {C}onf.,
  {Q}ueen's {U}niv., {K}ingston, {O}nt., 1976)}, volume No. 46 of {\em Queen's
  Papers in Pure and Appl. Math.}, pages 103--283. Queen's Univ., Kingston, ON,
  1977.

\bibitem{book-of-involution}
Max-Albert Knus, Alexander Merkurjev, Markus Rost, and Jean-Pierre Tignol.
\newblock {\em The book of involutions}, volume~44 of {\em American
  Mathematical Society Colloquium Publications}.
\newblock American Mathematical Society, Providence, RI, 1998.
\newblock With a preface in French by J. Tits.

\bibitem{Lang-Algebra}
Serge Lang.
\newblock {\em Algebra}, volume 211 of {\em Graduate Texts in Mathematics}.
\newblock Springer-Verlag, New York, third edition, 2002.

\bibitem{li-oort}
Ke-Zheng Li and Frans Oort.
\newblock {\em Moduli of supersingular abelian varieties}, volume 1680 of {\em
  Lecture Notes in Mathematics}.
\newblock Springer-Verlag, Berlin, 1998.

\bibitem{li-xue-yu:unit-gp}
Qun Li, Jiangwei Xue, and Chia-Fu Yu.
\newblock Unit groups of maximal orders in totally definite quaternion algebras
  over real quadratic fields.
\newblock {\em Trans. Amer. Math. Soc.}, 374(8):5349--5403, 2021.

\bibitem{MR2192012}
James~S. Milne.
\newblock Introduction to {S}himura varieties.
\newblock In {\em Harmonic analysis, the trace formula, and {S}himura
  varieties}, volume~4 of {\em Clay Math. Proc.}, pages 265--378. Amer. Math.
  Soc., Providence, RI, 2005.

\bibitem{MR2931385}
Santiago Molina.
\newblock Ribet bimodules and the specialization of {H}eegner points.
\newblock {\em Israel J. Math.}, 189:1--38, 2012.

\bibitem{mumford:av}
David Mumford.
\newblock {\em Abelian varieties}, volume~5 of {\em Tata Institute of
  Fundamental Research Studies in Mathematics}.
\newblock Published for the Tata Institute of Fundamental Research, Bombay,
  2008.
\newblock With appendices by C. P. Ramanujam and Yuri Manin, Corrected reprint
  of the second (1974) edition.

\bibitem{MR674652}
Richard~S. Pierce.
\newblock {\em Associative algebras}, volume~9 of {\em Studies in the History
  of Modern Science}.
\newblock Springer-Verlag, New York-Berlin, 1982.
\newblock Graduate Texts in Mathematics, 88.

\bibitem{Platonov-Rapinchuk}
Vladimir Platonov and Andrei Rapinchuk.
\newblock {\em Algebraic groups and number theory}, volume 139 of {\em Pure and
  Applied Mathematics}.
\newblock Academic Press, Inc., Boston, MA, 1994.
\newblock Translated from the 1991 Russian original by Rachel Rowen.

\bibitem{reiner:mo}
I.~Reiner.
\newblock {\em Maximal orders}, volume~28 of {\em London Mathematical Society
  Monographs. New Series}.
\newblock The Clarendon Press Oxford University Press, Oxford, 2003.
\newblock Corrected reprint of the 1975 original, With a foreword by M. J.
  Taylor.

\bibitem{Ribet-bimod}
Kenneth~A. Ribet.
\newblock Bimodules and abelian surfaces.
\newblock In {\em Algebraic number theory}, volume~17 of {\em Adv. Stud. Pure
  Math.}, pages 359--407. Academic Press, Boston, MA, 1989.

\bibitem{Roggenkamp-Latt-II}
Klaus~W. Roggenkamp.
\newblock {\em Lattices over orders. {II}}.
\newblock Lecture Notes in Mathematics, Vol. 142. Springer-Verlag, Berlin-New
  York, 1970.

\bibitem{Shimura1963-AltHermForms}
Goro Shimura.
\newblock Arithmetic of alternating forms and quaternion hermitian forms.
\newblock {\em J. Math. Soc. Japan}, 15:33--65, 1963.

\bibitem{terakado-xue-yu:2023}
Yasuhiro {Terakado}, Jiangwei {Xue}, and Chia-Fu {Yu}.
\newblock {On the supersingular locus of Shimura varieties for quaternionic
  unitary groups}.
\newblock {\em arXiv e-prints}, page arXiv:2311.18354, November 2023.

\bibitem{terakado-yu-xue-2026}
Yasuhiro {Terakado}, Jiangwei {Xue}, and Chia-Fu {Yu}.
\newblock {Superspecial Points on Shimura Curves}.
\newblock {\em arXiv e-prints}, page arXiv:2608.16036, August 2026.



\bibitem{vigneras}
Marie-France Vign{\'e}ras.
\newblock {\em Arithm\'etique des alg\`ebres de quaternions}, volume 800 of
  {\em Lecture Notes in Mathematics}.
\newblock Springer, Berlin, 1980.

\bibitem{voight-quat-book}
John Voight.
\newblock {\em Quaternion algebras}, volume 288 of {\em Graduate Texts in
  Mathematics}.
\newblock Springer, Cham, [2021] \copyright 2021.
\newblock stable post-publication version (v.1.0.5, Janurary 10, 2024).

\bibitem{waterhouse:thesis}
William~C. Waterhouse.
\newblock Abelian varieties over finite fields.
\newblock {\em Ann. Sci. \'Ecole Norm. Sup. (4)}, 2:521--560, 1969.

\bibitem{xue-yu:counting-av}
Jiangwei Xue and Chiafu Yu.
\newblock On {C}ounting {C}ertain {A}belian {V}arieties {O}ver {F}inite
  {F}ields.
\newblock {\em Acta Math. Sin. (Engl. Ser.)}, 37(1):205--228, 2021.

\bibitem{Yu-CF:Bull-AS}
Chia-Fu Yu.
\newblock Notes on locally free class groups.
\newblock {\em Bull. Inst. Math. Acad. Sin. (N.S.)}, 12(2):125--139, 2017.

\bibitem{Yu-Grenoble-2021}
Chia-Fu Yu.
\newblock On reduction of moduli schemes of abelian varieties with definite
  quaternion multiplications.
\newblock {\em Ann. Inst. Fourier (Grenoble)}, 71(2):539--613, 2021.

\end{thebibliography}
\end{document}